\documentclass[10pt,a4paper]{article}
\usepackage[numbers,sort&compress]{natbib}
\usepackage[T1]{fontenc}
\usepackage[utf8]{inputenc}
\usepackage{authblk}
\usepackage{amsmath,amssymb,amsthm,esint,bm}
\usepackage{mathrsfs}
\usepackage{bookmark}
\usepackage{amsmath}
\usepackage{enumitem}
\allowdisplaybreaks[4]

\newtheorem{definition}{Definition}[section]
\newtheorem{theorem}[definition]{Theorem}
\newtheorem{lemma}[definition]{Lemma}
\newtheorem{proposition}[definition]{Proposition}

\theoremstyle{remark}
\newtheorem{remark}[definition]{Remark}
\numberwithin{equation}{section}

\newcommand{\abs}[1]{\lvert#1\rvert}
\newcommand{\Abs}[1]{\left\lvert#1\right\rvert}

\newcommand{\rn}{{\mathbb{R}^d}}

\allowdisplaybreaks[4]

\title{Regularity for degenerate/singular normalized $p$-Laplacian equations with Hamiltonian terms}

\author[a]{Wentao Huo}
\affil[a]{School of Mathematical Sciences, Nankai University, Tianjin 300071, P.R. China}
\date{\today}

\usepackage{hyperref}
\begin{document}
\maketitle
\footnotetext[1]{E-mail: huowentaoouc@163.com (W. Huo). 
}

\begin{abstract}
This paper focuses on the regularity of viscosity solutions to normalized $p$-Laplacian equations with variable-exponent double phase type degeneracy/singularity and Hamiltonian terms. Based on a new improved oscillation-type estimate combined with a localized analysis, we establish sharp interior $C^{1,\alpha}$ regularity estimates in a unified way. In addition, in the degenerate case, we obtain improved gradient H\"{o}lder regularity results at points where the Hamiltonian coefficient and source term vanish, and establish a Schauder-type estimate at local extrema. Notably, our results are still novel even restricted to single power-type singularity or degeneracy law.



\end{abstract}


\section{Introduction}\label{section1}
In this paper, we are concerned with the local regularity of viscosity solutions to the following quasilinear elliptic equations with nonhomogeneous degeneracy/singularity and Hamiltonian term
\begin{equation}\label{55mainmodel}
-\Phi({Du}, x)\Delta_{p}^{N} u+\mathcal{H}({Du}, x)=f(x) \quad  \text{in} \quad B_{1},
\end{equation}
where $B_{1} \subset \mathbb{R}^d$ $(d\geq 2)$, $p\in(1,\infty)$, the source term $f\in C({B_{1}}) \cap L^{\infty}(B_{1})$, and $\Delta_{p}^{N}$ denotes normalized or game-theoretic $p$-Laplacian, defined by
\begin{equation*}
	\Delta_{p}^{N}  u:=|Du|^{2-p}\Delta_{p} u
	=\Delta u+\left\langle D^2 u \frac{Du}{|Du|},\frac{Du}{|Du|} \right\rangle.
\end{equation*}
The function $\Phi:\rn\times B_{1}\rightarrow [0,\infty)$ enjoys an appropriate variable degeneracy/singularity law, that is,
\begin{equation}\label{12}
	K_{1}\Upsilon(|\xi|,x)\leq	{\Phi}(\xi,x) \leq K_{2} \Upsilon(|\xi|,x)
\end{equation}
for constants $0<K_{1}\leq K_{2}<\infty$,
where
\begin{equation}\label{13}
	\Upsilon(|\xi|,x):=|\xi|^{p(x)}+{a}(x)|\xi|^{q(x)},\quad (\xi,x)\in \rn\times B_{1}
\end{equation}
with functions $p(\cdot),q(\cdot)\in C(B_{1})$ and a modulating function $a(\cdot)$ fulfilling
\begin{equation}\label{14}
	-1<p_{\rm min}\leq p(x)\leq q(x)\leq q_{\rm max}<\infty\quad {\rm and}\quad 0\leq a(\cdot)\in C({B_{1}}).
\end{equation}
The Hamiltonian term $\mathcal{H}:\mathbb{R}^{d} \times B_{1}\rightarrow \mathbb{R}$ is continuous and there exist constants $\mathcal{K},\mathcal{M}>0$ and $0<m\leq 1+p_{\rm min}$ such
that
\begin{equation}\label{15}
	|\mathcal{H}(\xi,x)|\leq \mathcal{K}+\mathcal{M}|\xi|^{m}
\end{equation}
for every $\xi\in\mathbb{R}^{d}$, $x\in B_{1}$. 

Over the past years, the normalized $p$-Laplacian operator has attracted increasing attention, partly because of its application in image processing \cite{Elmoataz2015SIAM,Does2011CPAA,Kawohl2008AMS} and in the description of tug-of-war games \cite{Manfredi2010SIAMJMA,Peres2008Duke,Peres2009JAMS}. We remark that the normalized $p$-Laplacian can be regarded as the one-homogeneous version of standard $p$-Laplacian or as a 
mixture of the Laplacian and normalized infinity Laplacian, $\Delta_{\infty} u:=\left\langle D^2 u \frac{Du}{|Du|},\frac{Du}{|Du|} \right\rangle$. 
The lower regularity for solutions of the homogeneous or nonhomogeneous elliptic normalized $p$-Laplace equation was obtained in \cite{Luiro2013CPDE,Ruosteenoja2016Advcalc}. The ﬁrst contribution on the $C^{1,\alpha}$-regularity for $-\Delta_{p}^{N} u=f(x)$ is due to the seminal work of Attouchi–Parviainen–Ruosteenoja \cite{Attouchi2017JMPA}, where they showed that the solutions are locally $C^{1,\alpha}$ regular 
under the condition that $f\in L^{q}$ with $q\leq \infty$ possessing a suitably large lower bound; see also \cite{Banerjee2020CCM} for the case that $f\in L(d,1)$, where $L(d,1)$ denotes the standard Lorentz space. 

Mathematical models like \eqref{55mainmodel} are also strongly inspired by certain variational integrals of the calculus of variations with nonstandard growth satisfying a sort of double-phase structure \cite{Marcellini1989,Mingione2015,Mingione201522,Zhikov1993,Zhikov1995}. One of the principal features of model \eqref{55mainmodel} is its interplay between two distinct types of degeneracy/singularity rate, in accordance with the values of the modulating function. Observe that \eqref{55mainmodel} is a 
natural extension of canonical quasilinear elliptic equations with singularity/degeneracy, whose highly celebrated prototype is
\begin{equation}\label{eq1}
	-|Du|^{\gamma}\Delta_{p}^{N} u=f(x) \quad  \text{in} \quad B_{1}.
\end{equation}
In the restricted case $p\geq 2$ and $-1<\gamma\leq 0$, Birindelli and Demengel \cite{Birindelli2010JDE} showed the local H\"{o}lder regularity of the gradient for solutions to \eqref{eq1} by using approximations and a ﬁxed point argument. Afterwards, Attouchi \cite{Attouchi2018JDE} extends such result to the full range $p>1$ and $\gamma>-1$. A generalisation of the degeneracy law $\xi\rightarrow |\xi|^{\gamma}$ appeared in \cite{Santos2025AdvCalc}. In that paper, Andrade et al. considered viscosity solutions to degenerate fully nonlinear equations of the form
$$-\sigma(|Du|)\Delta_{p}^{N} u=f(x) \quad  \text{in} \quad B_{1}$$
under the assumptions that $\sigma:[0,\infty)\rightarrow [0,\infty)$ is a general modulus of continuity whose inverse $\sigma^{-1}$ is Dini continuous, and $f\in L^{\infty}(B_{1})$. In this setting, they proved that solutions are locally of class $C^{1}$. In the context of normalized $p$-Laplacian elliptic equations with nonhomogeneous variable exponent degeneracy, Wang et al. \cite{Jiang2025Potential} considered equations of the form
\begin{equation*}
	-\bigg[|Du|^{p(x,u)}+a(x)|Du|^{q(x,u)}\bigg]\Delta_{p}^{N} u=f(x) \quad  \text{in} \quad B_{1}
\end{equation*}
for $0<p_{1}\leq p(x,u)\leq q(x,u)\leq p_{2}<\infty$, $0 \leq a(\cdot)\in C(B_{1})$ and $f\in L^{\infty}(B_{1})$. Under such conditions, the local $C^{1,\alpha}$
regularity of viscosity solutions was established. In addition, they also obtain almost optimal pointwise H\"{o}lder regularity of gradient for degenerate free transmission problem related to normalized $p$-Laplacian. Very recently, \cite{Santos2026JDE} considered the general singular/degenerate operator $\Phi$ and derived local optimal $C^{1,\alpha}$ regularity of solutions. For the related regularity results in the fully nonlinear elliptic context, we refer to \cite{Imbert1,Silva2020,Silva2023,Fang,Fili,Baasandorj,Andrade2022,Bronzi2020,Ricarte} and the references therein.

However, there are very few regularity results with regard to the singular or degenerate quasilinear elliptic equations governed by the normalized $p$-Laplacian involving  Hamiltonian terms. To the best of our knowledge, the only systematic study in this direction is due to the recent work \cite{Junior Bessa20025ARXIV}, where they investigated \eqref{55mainmodel} with $\Phi(Du,x)=|Du|^{\gamma}$ and $\mathcal{H}(Du,x)=\left\langle \mathcal{B}(x),Du\right\rangle|Du|^{\gamma}+h(x)|Du|^{m}$. They proved interior $C^{1,\alpha}$ regularity of solutions under the conditions that $\gamma>0$, $\mathcal{B}\in C(\overline{B}_{1}.;\rn)$, $h\in C(\overline{B}_{1})$, $0<m<1+\gamma$, and $f\in C({B}_{1})\cap L^{\infty}(B_{1})$. In addition,  several extra aspects of \eqref{55mainmodel} have been explored as well, such as existence and uniqueness of solutions, the comparison principles, Aleksandrov–Bakelman–Pucci type estimate, the strong maximum principle, the Hopf Lemma, and Liouville-type theorems, please refer to  \cite{Junior Bessa20026ARXIV,Junior Bessa20025ARXIV} and the references therein. For the related regularity results of fully nonlinear elliptic equations, see \cite{Birindelli2014ESAIM,Birindelli2015,Huo2026,B-Demengel2016,Andrade} and the references therein. It should be emphasized that $\Delta_{p}^{N}u$ is gradient dependent and discontinuous, so the results of \cite{Andrade,Huo2026} cannot be applied directly to the  normalized $p$-Laplacian operator. This challenge necessitates the development of a new technique to address the problem effectively.

Inspired by the work mentioned above, we naturally consider  normalized $p$-Laplacian equation with variable-exponent double phase type degeneracy/singularity and Hamiltonian terms of the form \eqref{55mainmodel}. To the best of our knowledge, the sharp gradient H\"{o}lder regularity of viscosity solutions for \eqref{55mainmodel} is unknown in the current literature. Based on a new improved oscillation-type estimate combined with a localized analysis, we establish the sharp $C^{1,\alpha}$ regularity of solutions to \eqref{55mainmodel} in a unified way.

We now state the first main result concerning local H\"{o}lder continuity for gradient of viscosity solutions to \eqref{55mainmodel}.
\begin{theorem}\label{thm1}
	Let $u\in C(B_{1})$ be a viscosity solution of \eqref{55mainmodel} under hypotheses \eqref{12}-\eqref{15}. Then there exists $\alpha^{\prime}=\alpha^{\prime}(d,p,p_{\rm min},q_{\rm max})\in(0,1)$ such that $u\in C_{\rm loc}^{1,\alpha^{\prime}}(B_{1})$ with the following
	estimates:
	\begin{itemize}
		\item [{\rm$({{\rm i}})$}] if $0<m<1+p_{\rm min}$, then
		\begin{equation*}
			\|u\|_{C^{1, \alpha^{\prime}}(B_{1/2})}\leq C\left(1+\|u\|_{L^{\infty}\left(B_{1}\right)}+\left(\|f\|_{L^{\infty}\left(B_{1}\right)}+\mathcal{K}\right)^{\frac{1}{1+p_{\rm min}}}+\mathcal{M}^{\frac{1}{1+p_{\rm min}-m}}\right),
		\end{equation*}
		where the constant $C$ depends on $d,p,\alpha^{\prime},m,p_{\rm min}$;
		\item [{\rm$({{\rm ii}})$}] if $m=1+p_{\rm min}$, then
		\begin{equation*}
			\|u\|_{C^{1, \alpha^{\prime}}(B_{1/2})}\leq C\left(1+\|u\|_{L^{\infty}(B_{1})}\right),
		\end{equation*}
		where the constant $C$ depends in addition on $\|f\|_{L^{\infty}\left(B_{1}\right)}$, $\mathcal{K}$ and $\mathcal{M}$.
	\end{itemize}
\end{theorem}

Since the solutions of \eqref{55mainmodel} should not be expected to be more regular than $p$-harmonic functions, the maximal exponent $\alpha_{0}$ (see Remark \ref{ptiaohezhengzexing} below) is a natural upper bound for $C^{1,\alpha}$ regularity for equation \eqref{55mainmodel}. Our next result establishes an optimal geometric estimate of \eqref{55mainmodel}.
\begin{theorem}\label{thm2}
	  Let $u\in C(B_{1})$ be a viscosity solution of \eqref{55mainmodel} under hypotheses \eqref{12}-\eqref{15}.  Then for every $x_{0}\in B_{1/2}$, $u$ is $C^{1,\alpha}$ at $x_{0}$ with 
	\begin{equation}\label{exponent}
			\alpha\in 
		\begin{cases}
				(0,\alpha_{0})\cap \left(0,\frac{1}{1+q_{\rm max}}\right]&  \text{if}\;\;  p_{\rm min}\geq 0, \\
			(0,\alpha_{0})\cap \left(0,\frac{1}{1+q_{\rm max}-p_{\rm min}}\right] & \text{if}\;\;  -1<p_{\rm min}<0.
		\end{cases}
	\end{equation}
More precisely, for all $0<r<\frac{1}{2}$, there holds
\begin{itemize}
	\item [{\rm$({{\rm i}})$}] if $0<m<1+p_{\rm min}$, then
	\begin{equation*}
		\|u\|_{C^{1, \alpha}(B_{r}({x_0}))}\leq C\left(1+\|u\|_{L^{\infty}\left(B_{1}\right)}+\left(\|f\|_{L^{\infty}\left(B_{1}\right)}+\mathcal{K}\right)^{\frac{1}{1+p_{\rm min}}}+\mathcal{M}^{\frac{1}{1+p_{\rm min}-m}}\right),
	\end{equation*}
	where the constant $C$ depends on $d,p,\alpha,m,p_{\rm min}$;
	\item [{\rm$({{\rm ii}})$}] if $m=1+p_{\rm min}$, then
	\begin{equation*}
		\|u\|_{C^{1, \alpha}(B_{r}({x_0}))}\leq C\left(1+\|u\|_{L^{\infty}(B_{1})}\right),
	\end{equation*}
	where the constant $C$ depends in addition on $\|f\|_{L^{\infty}\left(B_{1}\right)}$, $\mathcal{K}$ and $\mathcal{M}$.
\end{itemize}
\end{theorem}
Due to the generality of degeneracy/singularity, our main 
results Theorems \ref{thm1} and \ref{thm2} above embraces and improves the regularity results previously obtained in \cite{Attouchi2018JDE}.  It is worth emphasizing that establishing sharp gradient H\"{o}lder estimates for normalized $p$-Laplacian equations simultaneously involving double-phase type degeneracy/singularity laws and Hamiltonian terms is highly nontrivial. The key novelty lies in developing a systematic approach to control the growth of the Hamiltonian and providing a new refined analysis adapted to the present framework. In addition, our findings are new even for the case $a(x)\equiv 0$. 

Next, we impose additional conditions on the source term and Hamiltonian term to give higher regularity results of \eqref{55mainmodel} for the degenerate case.  
More precisely, we assume that
\begin{equation}\label{21}
	\mathcal{H}(Du,x)=h(x)|Du|^{m}
\end{equation}
with $0<m\leq 1+p_{\rm min}$ and $h\in C(B_{1})\cap L^{\infty}(B_{1})$. Additionally, we assume 
\begin{equation}\label{22}
	f(0)=h(0)=0,
\end{equation}
and for some $K_{3},K_{4}>0$ and $\theta_{1},\theta_{2}\in (0,1)$, there holds
\begin{equation}\label{23}
	|f(x)|\leq {K}_{3}|x|^{\theta_{1}},\quad |h(x)|\leq {K}_{4}|x|^{\theta_{2}}
	,\quad x\in B_{1}.
\end{equation}
 
The result stated in Theorem \ref{thm2} above is sharp, however, this $C^{1,\alpha}$ regularity result alone does not give a complete description. This raises the question of whether the regularity can be improved, at least at some meaningful points. A first observation  comes from  an important example. For any $\alpha\in (0,1)$, the function $u(x)=|x|^{1+\alpha}$ satisfies
$$|Du|^{\gamma}\Delta_{p}^{N} u=(1+\alpha)^{1+\gamma}(\alpha(p-1)+d-1)|x|^{\alpha(1+\gamma)-1}.$$
Letting $f(x)=C|x|^{\theta_{1}}$ with $0<\theta_{1}:=\alpha(1+\gamma)-1$, one can see that,
$u$ is $C^{1,\frac{1+\theta_{1}}{1+\gamma}}$ at the origin (a critical point of $u$), and thus surpassing the optimal regularity exponent of $\frac{1}{1+\gamma}$. 

Another heuristic analysis comes from a scaling argument. To be precise, consider a simple case of $\Phi(Du,x)=|Du|^{\gamma}$, that is
\begin{equation*}
	-|Du|^{\gamma}\Delta_{p}^{N} u+h(x)|Du|^{m}=f(x) \quad {\rm in}\;\, B_{1}.
\end{equation*}
For any $r\in(0,1)$, define $v(x):=u(rx)/r^{\vartheta}$ for $x\in B_{1}$. It is clear that $v$ solves
\begin{equation*}
	-|Dv|^{\gamma}\Delta_{p}^{N} v+\tilde{h}(x)|Dv|^{m}=\tilde{f}(x) \quad {\rm in}\;\, B_{1},
\end{equation*}
where
\begin{equation*}
	\tilde{h}(x)=r^{2+\gamma-m-\vartheta(1+\gamma-m)}h(rx)\quad {\rm and}\quad \tilde{f}(x)=r^{2+\gamma-\vartheta(1+\gamma)}f(rx).
\end{equation*}
Thus, if we select
\begin{equation*}
	0<\vartheta\leq \min\left\{\frac{2+\gamma+\theta_{1}}{1+\gamma},\frac{2+\gamma-m+\theta_{2}}{1+\gamma-m}\right\},
\end{equation*}
then it follows from $|f(x)|\leq {K}_{3} |x|^{\theta_{1}}$, $|h(x)|\leq {K}_{4} |x|^{\theta_{2}}$ and $r\in (0,1)$ that
\begin{equation*}
	\abs{\tilde{f}(x)}\leq {K}_{3} r^{2+\gamma-\vartheta(1+\gamma)+\theta_{1}}|x|^{\theta_{1}}\leq {K}_{3} |x|^{\theta_{1}},
\end{equation*}
\begin{equation*}
	\abs{\tilde{h}(x)}\leq K_{4} r^{2+\gamma-m-\vartheta(1+\gamma-m)+\theta_{2}}|x|^{\theta_{2}}\leq {K}_{4} |x|^{\theta_{2}},
\end{equation*}
that is, $\tilde{f}$ and $\tilde{h}$ satisfy the same structural assumptions as $f$ and $h$.
Therefore, $v$ satisfies an equation with the same structure as $u$. Given $0<m\leq 1+\gamma$, we find
\begin{equation*}
	\min\left\{\frac{2+\gamma+\theta_{1}}{1+\gamma},\frac{2+\gamma-m+\theta_{2}}{1+\gamma-m}\right\}=1+\min\left\{\frac{1+\theta_{1}}{1+\gamma},\frac{1+\theta_{2}}{1+\gamma-m}\right\}=:1+\varsigma.
\end{equation*}
Hence, the optimal regularity can be expected to be $C^{1,\varsigma}$. In particular, if $\gamma<\min\left\{\theta_{1},\theta_{2}+m\right\}$, then it follows $\varsigma>1$ and so the optimal regularity can be expected to be $C^{2,\varsigma-1}$. Therefore, the H\"{o}lder continuity of source term and Hamiltonian coefficient directly influences, in a quantitative way, the regularity estimates for solutions to degenerate equations. The first breakthrough in this direction was made by Nascimento in \cite{Nascimento2025}, who investigated the degenerate fully nonlinear equation $\abs{Du}^{\gamma}F(D^2 u, x)=f(x)$ and showed a Schauder-type regularity estimate at local extrema by means of a geometric approach. Subsequently, a recent paper \cite{Huo22026} considered the following degenerate fully nonlinear equations with Hamiltonian terms
\begin{equation*}
	\Phi({Du}, x)F(D^2 u, x) +h(x)|Du|^{m}=f(x) \quad  \text{in} \quad B_{1},
\end{equation*}
where $\Phi$ satisfies \eqref{12}-\eqref{14} with $p_{\rm min}\geq 0$. They obtained improved gradient H\"{o}lder regularity results at points where the Hamiltonian coefficients and source terms vanish. Furthermore, they established Hessian continuity at local extrema, which is sharp with respect to the vanishing 
rate of the Hamiltonian coefficient and source term. Motivated by the analysis and work mentioned above, here we aim to establish similar higher regularity estimates for viscosity solutions to \eqref{55mainmodel}. 

The following result is an improved H\"{o}lder regularity of gradient at the origin point, at which the source term $f$ and Hamiltonian coefficient $h$ vanish at a prescribed rate.
\begin{theorem}\label{main}
	Let $u\in C(B_{1})$ be a bounded viscosity solution of \eqref{55mainmodel} under hypotheses \eqref{12}-\eqref{14} and \eqref{21}-\eqref{23} with $p_{\rm min}\geq 0$. Then $u$ is of class $u\in C^{1,\min\left\{\alpha_{0}^{-},\alpha_{1}\right\}}$ at the origin with
	\begin{equation*}
		\alpha_{1}:=
			\min\left\{\frac{1+\theta_{1}}{1+q_{\rm max}},\frac{1+\theta_{2}}{1+q_{\rm max}-m}\right\}
	\end{equation*}
	That is, given any
	\begin{equation}\label{3aerfadefanwei}
		\alpha\in
		(0,\alpha_{0})\cap \left(0,\alpha_{1}\right],
	\end{equation}
	there exists universal constants
	$0<r<\frac{1}{4}$ and $C>0$ such that
	\begin{equation*}
		\Abs{u(x)-u(0)-Du(0)\cdot x}\leq C|x|^{1+\alpha}\quad {\rm for\; all \;} x\in B_{r}.
	\end{equation*}
\end{theorem}
\begin{remark}\label{coro1}
	It should be noted that, in contrast to the sharp  $C^{1,\min\left\{\alpha_{0}^{-},\frac{1}{1+q_{\rm max}}\right\}}$ regularity for solutions with bounded source term and bounded Hamiltonian coefficient $h$ (see Theorem \ref{thm2}),
	Theorem \ref{main} exhibits a significant improvement in smoothness. Moreover, it is noteworthy that
	$$\min\left\{\frac{1+\theta_{1}}{1+q_{\rm max}},\frac{1+\theta_{2}}{1+q_{\rm max}-m}\right\}\geq 1,$$
	provided that $q_{\rm max}\leq \min\left\{\theta_{1},m+\theta_{2}\right\}$. Therefore, the solution $u$ to \eqref{55mainmodel} is of class $ C^{1,\alpha_{0}^{-}}$ at the origin, in other words, the solution of \eqref{55mainmodel} is asymptotically as regular as $p$-harmonic function.
\end{remark}

The final result provided in this work is a Hessian continuity at a local extrema.
\begin{theorem}\label{3aamain2}
	Let $u\in C(B_{1})$ be a bounded viscosity solution of \eqref{55mainmodel} under hypotheses \eqref{12}-\eqref{14} and \eqref{21}-\eqref{23} with $p_{\rm min}\geq 0$. Assume further $q_{\rm max}<\min\left\{\theta_{1},m+\theta_{2}\right\}$, and that origin is a local extrema of $u$, i.e., $u(0)\leq u(x)$ or $u(0)\geq u(x)$ in $B_{\nu}(0)$ for some $\nu\in \left(0,\frac{1}{4}\right)$. Then $u$ is twice differentiable at the origin and
	\begin{equation}\label{3ajielun2}
		|u(x)-u(0)|\leq C|x|^{2+\mu} \quad {\rm for\; all \;} x\in B_{r},
	\end{equation}
	where
	\begin{equation*}
		\mu:=\min\left\{\frac{\theta_{1}-q_{\rm max}}{1+q_{\rm max}},\frac{m+\theta_{2}-q_{\rm max}}{1+q_{\rm max}-m}\right\}
	\end{equation*}
	and universal constants $0<r<\nu$ and $C>0$. That is, $u$ is
	of class $u\in C^{2,\mu}$ at the origin with $|Du(0)|=|D^{2}u(0)|=0$.
\end{theorem}

The remainder of this paper is organized as follows. In Section \ref{section222}, we introduce the basic notions and some well-known results. In Section \ref{section333}, we present the proof of local $C^{1,\alpha^{\prime}}$ regularity estimate. Section \ref{section444} is dedicated to proving sharp H\"{o}lder regularity of the gradient stated in Theorem \ref{thm2}. In Section \ref{section555}, we establish the improved gradient regularity result. 
 Finally, in the last section, we complete the proof of Theorem \ref{3aamain2} regarding Schauder-type regularity estimate.
\section{Preliminaries}\label{section222}
In this section, we first introduce the notations and give the definition of viscosity solutions to the normalized $p$-Laplacian equation and \eqref{55mainmodel}, and subsequently collect some useful auxiliary results, which are pivotal for proving our main theorems.

\subsection{Notations and basic concepts}
In the entire paper, let $S^{d}$ be the set of all real $d\times d$ symmetric matrices and $B_{r}(x_{0})$ be the open ball with radius $r$ and centred at $x_{0}\in\rn$. 
In particular, we shall simply denote $B_{r}:=B_{r}(0)$. We use both $\langle \eta, \xi \rangle$ and $\eta\cdot \xi$ to denote the inner product of vectors $\eta, \xi\in\rn$. Also, we have the usual partial ordering: $A\leq B$ in $S^{d}$ means that $\left\langle A\xi,\xi\right\rangle\leq \left\langle B\xi,\xi\right\rangle$ for any $\xi \in \mathbb{R}^{d}$. In other words, $B-A$ is positive semidefinite. For $a,b\in \rn$, we denote by $a\otimes b$ the $d\times d$-matrix for which $(a\otimes b)_{ij}=a_{i}b_{j}$. The symbol $C$ denotes a positive constant whose value may vary from line to line, and only the relevant dependencies are specified in parentheses. Besides, a constant is said to be universal if it depends at most upon $d$, $p$ and the structure constants in \eqref{12}-\eqref{15}.

We define the viscosity solutions of the normalized \( p \)-Laplacian equation and \eqref{55mainmodel}.
\begin{definition}
	Let \( 1 < p < \infty \). An upper semicontinuous function \( u \) is a viscosity subsolution of the equation \( -\Delta_p^N u = f \) if for all \( x_0 \in B_1 \) and \( \varphi \in C^2(B_1) \) such that \( u - \varphi \) attains a local maximum at \( x_0 \), one has
	\[
	\begin{cases} 
		-\Delta_p^N \varphi(x_0) \leq f(x_0), & \text{if } D\varphi(x_0) \neq 0, \\ 
		-\Delta \varphi(x_0) - (p-2)\lambda_{\max}(D^2 \varphi(x_0)) \leq f(x_0), & \text{if } D\varphi(x_0) = 0 \text{ and } p \geq 2, \\ 
		-\Delta \varphi(x_0) - (p-2)\lambda_{\min}(D^2 \varphi(x_0)) \leq f(x_0), & \text{if } D\varphi(x_0) = 0 \text{ and } 1 < p < 2.
	\end{cases}
	\]
	A lower semicontinuous function \( u \) is a viscosity supersolution of the equation \( -\Delta_p^N u = f \) if for all \( x_0 \in B_1 \) and \( \varphi \in C^2(B_1) \) such that \( u - \varphi \) attains a local minimum at \( x_0 \), one has
	\[
	\begin{cases} 
		-\Delta_p^N \varphi(x_0) \geq f(x_0), & \text{if } D\varphi(x_0) \neq 0, \\ 
		-\Delta \varphi(x_0) - (p-2)\lambda_{\min}(D^2 \varphi(x_0)) \geq f(x_0), & \text{if } D\varphi(x_0) = 0 \text{ and } p \geq 2, \\ 
		-\Delta \varphi(x_0) - (p-2)\lambda_{\max}(D^2 \varphi(x_0)) \geq f(x_0), & \text{if } D\varphi(x_0) = 0 \text{ and } 1 < p < 2.
	\end{cases}
	\]
	We say that \( u \) is a viscosity solution of \( -\Delta_p^N u = f \) in \( B_1 \) if it is both a viscosity subsolution and a viscosity supersolution.
\end{definition}
\begin{definition} \label{dingyi2}
	Let \( 1 < p < \infty \). An upper semicontinuous function $u$ is a viscosity subsolution of \eqref{55mainmodel}, if for every \( x_0 \in B_1 \),
	\begin{itemize}	
		\item  [{\rm$\bullet$}] either there exists $\delta>0$ such that $u$ is constant in $B_{\delta}(x_{0})$ and $f(x)\geq 0$ for all $x\in B_{\delta}(x_{0})$;
		\item  [{\rm$\bullet$}]  or for all $\varphi \in C^2\left(B_{1}\right)$ such that $ u -\varphi$ attains a local maximum at $x_0$ and $D\varphi(x_{0})\neq 0$, it holds  
		\[
		-\Phi(D\varphi(x_0),x_0) \Delta_p^N \varphi(x_0)+\mathcal{H}(D\varphi(x_0),x_0) \leq f(x_0).
		\]
	\end{itemize}
A lower semicontinuous function \( u \) is a viscosity supersolution of \eqref{55mainmodel}, if for every \( x_0 \in B_1 \),
	\begin{itemize}	
	\item  [{\rm$\bullet$}] either there exists $\delta>0$ such that $u$ is constant in $B_{\delta}(x_{0})$ and $f(x)\leq 0$ for all $x\in B_{\delta}(x_{0})$;
	\item  [{\rm$\bullet$}]  or for all $\varphi \in C^2\left(B_{1}\right)$ such that $ u -\varphi$ attains a local minimum at $x_0$ and $D\varphi(x_{0})\neq 0$, it holds  
	\[
	-\Phi(D\varphi(x_0),x_0) \Delta_p^N \varphi(x_0)+\mathcal{H}(D\varphi(x_0),x_0) \geq f(x_0).
	\]
\end{itemize}
We say that \( u \) is a viscosity solution of \eqref{55mainmodel} if it is both a viscosity subsolution and a viscosity supersolution. 
\end{definition}
This definition of viscosity solution above was proposed by Birindelli and Demengel in \cite{Birindelli2004,Birindelli2006} for 
the singular case $-1<p_{\rm nin}<0$. Note that in the case $p_{\rm min}>0$, Definition \ref{dingyi2} is equivalent to the classical definition of viscosity solution (cf. \cite[Lemma 2.1]{Davil2010}). Throughout this work, we say that \( u \in C(B_1) \) is a normalized viscosity solution if $\|u\|_{L^{\infty}(B_{1})}\leq 1$.

\begin{remark}\label{xianxingyizhituoyuan}
The normalized \( p \)-Laplacian operator can be seen as 
a uniformly elliptic operator in the sense that
\[
\mathcal{P}_{\lambda, \Lambda}^{-}(D^2 u) \leq \Delta_p^{\mathrm{N}} u \leq \mathcal{P}_{\lambda, \Lambda}^{+}(D^2 u),
\]
where
\[
\mathcal{P}_{\lambda, \Lambda}^{-}(D^2 u) := \inf_{A \in A_{\lambda, \Lambda}} \operatorname{Tr}(AD^2 u) \quad {\rm and}\quad \mathcal{P}_{\lambda, \Lambda}^{+}(D^2 u) := \sup_{A \in A_{\lambda, \Lambda}} \operatorname{Tr}(AD^2 u)
\]
are the Pucci extremal operators, and \( A_{\lambda, \Lambda} \subset S^{d}\) is a set of symmetric \( d\times d \) matrices, whose eigenvalues belong to \([ \lambda, \Lambda ]\). Indeed, the normalized \( p \)-Laplacian operator can be written in the form
\[
\Delta_p^{\mathrm{N}} u = \operatorname{Tr} \left( \left(I + (p - 2) \frac{Du}{|Du|} \otimes \frac{Du}{|Du|}\right) D^2 u \right),
\]
then it is easy to verify that \(\lambda = \min \{1, p - 1\}\) and \(\Lambda = \max \{1, p - 1\}\).
\end{remark} 
\subsection{ Auxiliary Results}
We first recall the uniform H\"{o}lder estimate for gradient to the homogeneous normalized $p$-Laplacian equation.
\begin{lemma}\label{111raodongzhengzexing}{\rm(cf. \cite[Lemma 3.2]{Attouchi2017JMPA})} 
	Let \( v \) be a viscosity solution of
	\[
	-\Delta v - (p-2) \left( D^2 v \frac{Dv + \xi}{|Dv + \xi|}, \frac{Dv + \xi}{|Dv + \xi|} \right) = 0 \quad \text{in } B_1, \quad \xi \in \mathbb{R}^n,
	\]
	with \( \|v\|_{L^\infty(B_1)} \leq \frac{1}{2} \). For all \( 0 < r \leq \frac{1}{2} \), there exist constants \( C_0 = C_0(d,p) > 0 \) and \( \beta_1 = \beta_1(d,p) > 0 \) such that
	\[
	\|v\|_{C^{1,\beta_1}(B_r)} \leq C_0.
	\]
\end{lemma} 
The following result, established by Siltakoski, provides a fundamental connection between viscosity solutions of the homogeneous normalized $p$-Laplacian and weak solutions of the homogeneous $p$-Laplacian.
\begin{lemma}\label{dengjia}{\rm(cf. \cite[Theorem 5.9]{Siltakoski2018CVPDE})} 
	A function $u \in C(B_{1})$ is a viscosity solution to
	\[
	-\Delta_p^N u = 0 \quad \text{in} \quad B_{1}
	\]
	if and only if it is a weak solution to
	\[
	-\Delta_p u = 0 \quad \text{in} \quad B_{1}.
	\]
\end{lemma}
\begin{remark}\label{ptiaohezhengzexing}
	It is well known that $p$-harmonic functions are of class $C^{1,\alpha_{0}}$ for some maximal exponent $0<\alpha_{0}<1$ that depends only upon $d$ and $p$. This was shown independently by Uraltseva \cite{Uraltseva1968} and Uhlenbeck \cite{Uhlenbeck1977ActaMath} in 
	the case $p>2$, and later extended to the case $p > 1$, see \cite{DiBenedetto1983NonlinearAnal,Lewis1983Indiana}. 
\end{remark}

Next, we present the nearly optimal regularity for solutions to the normalized \( p \)-Laplacian equation.  
\begin{lemma}\label{111tiaohezhengzexing}{\rm(cf. \cite[Theorem 1.3]{Attouchi2017JMPA})}
 Fixing an arbitrary constant \(\varsigma \in (0,\alpha_{0}) \), where \(\alpha_{0}\) is the optimal H\"{o}lder exponent for gradients of \( p \)-harmonic functions in terms of a priori estimate. If \(p > 1\) and \( f \in C(B_1) \cap L^\infty(B_1) \), then viscosity solutions to \(-\Delta_p^N u = f\) are in  
$C_{\rm loc}^{1, \hat{\alpha}_0}(B_1)$ with  $\hat{\alpha}_0 = \alpha_{0}-\varsigma$.
In particular, if the equation is homogeneous, namely \( f = 0 \), then viscosity solutions to  $-\Delta_p^N u = 0$ are in $C_{\rm loc}^{1, \hat{\alpha}_0}(B_1)$ with  $\hat{\alpha}_0 = \alpha_{0}-\varsigma$.
\end{lemma} 

We end this subsection with the following assertion from \cite[Lemma 2.5]{Attouchi2018JDE} or \cite[Proposition 2.1]{Santos2026JDE}, which will be essential for handling the singular case in subsequent proofs.
\begin{proposition}\label{qiyizhuantuihua}
	Assume \eqref{12}-\eqref{15} hold with $-1<p_{\rm min}<0$. If $u$ is a viscosity solution of \eqref{55mainmodel}, then $u$ is a viscosity solution of
	\begin{equation}\label{model1111}
	\abs{Du}^{-p_{\rm min}}\Phi({Du}, x)\Delta_{p}^{N}u+\abs{Du}^{-p_{\rm min}}\mathcal{H}(Du,x) =\abs{Du}^{-p_{\rm min}}f(x) \quad  \text{in} \quad B_{1}.
	\end{equation}
\end{proposition}

\section{Local H\"{o}lder regularity of the gradient}\label{section333}
This section is devoted to the proof of Theorem \ref{thm1} on the local gradient H\"{o}lder continuity for viscosity solutions to equation \eqref{55mainmodel}.

\subsection{Compactness of solutions}
In this subsection, we show local H\"{o}lder continuity of viscosity solution to the following perturbed equations
\begin{equation}\label{531}
	-\Phi({Du+\xi}, x)\Delta_{p,\xi}^{N}u+\mathcal{H}({Du+\xi}, x)=f(x) \quad  \text{in} \quad B_{1},
\end{equation}
where $\xi$ is an arbitrary vector in $\rn$ and 
$$\Delta_{p,\xi}^{N}u:=\Delta u+(p-2)\left\langle D^2 u \frac{Du + \xi}{|Du + \xi|}, \frac{Du + \xi}{|Du + \xi|} \right\rangle.$$
Our proof relies on the celebrated Crandall-Ishii-Lions Lemma (see \cite[Theorem 3.2]{Crandle1}, \cite[Proposition II.3]{Lions1}) .

We first establish the following compactness result for the case $0<m\leq p_{\rm min}$.
\begin{proposition}\label{prop3.1} 
	Let $u\in C(B_{1})$ be a normalized viscosity solution of \eqref{531} under assumptions \eqref{12}-\eqref{15} with $0<m\leq p_{\rm min}$. Then there exists a constant $\mathcal{J}_{0}=\mathcal{J}_{0}(d,p,m,p_{\rm min},\|f\|_{L^{\infty}(B_{1})},\mathcal{K},\mathcal{M})>1$ such that
	\begin{itemize}
		\item [{\rm$({{\rm i}})$}] if $\abs{\xi}\geq \mathcal{J}_{0}$, then $u$ is locally Lipschitz continuous in $B_{1}$.
		 In addition, there holds that
		 $$|u(x)-u(y)|\leq C|x-y|$$
		 for all $x,y\in B_{3/4}$, where $C>0$ is a universal constant.
		\item [{\rm$({{\rm ii}})$}] If $\abs{\xi}<\mathcal{J}_{0}$, then $u\in C_{\rm loc}^{0,\gamma}(B_{1})$ for some $\gamma>0$. In addition, 
		there holds that
		$$|u(x)-u(y)|\leq C|x-y|^{\gamma}$$
		for all $x,y\in B_{3/4}$, where $C>0$ is a universal constant.
	\end{itemize}
\end{proposition}
\begin{proof}
(i) To prove the local Lipschitz regularity of $u$, let us fix $0<r<r_{1}<1$ and consider the quantity
	\begin{equation}\label{541}
		\mathcal{G}(x_{0}):=\sup\limits_{(x,y)\in B_{r_{1}}\times B_{r_{1}}}\left\{u(x)-u(y)-L_{1}w(\abs{x-y})-L_{2}\Big(\lvert x-x_{0}\rvert^{2}+\lvert y-x_{0} \rvert^{2} \Big)\right\}
	\end{equation}
	for each $x_{0}\in B_{r}$, where
	$$
	w(s)=
	\begin{cases}
		s- w_{0}s^{1+\beta}& \text{if} \;\; 0\leq s\leq s_{0}:=\left(\frac{1}{(1+\beta)w_{0}}\right)^{1/\beta}, \\
		w(s_{0}) &\text{if}\;\;  s>s_{0},
	\end{cases}
	$$
	with $\beta\in(0,1)$ and $w_{0}\in\left(0,\frac{1}{(1+\beta)2^{\beta}}\right)$. Observe that $s_{0}>2$,
	$$w(s)\geq 0,\quad 0\leq w^{\prime}(s)\leq 1,\quad w^{\prime\prime}(s)\leq 0,\quad \forall s\geq 0.$$
	
	Our goal is to prove that there exist constants $L_{1},L_{2}>1$ such that $\mathcal{G}\leq 0$ for all $x_{0}\in B_{r}$. We argue by contradiction by assuming that there exists $\hat{x}_{0}\in B_{r}$ so that $\mathcal{G}(\hat{x}_{0})>0$ for all $L_{1},L_{2}>1$.  Consider $\psi,\Psi:\overline{B}_{r_{1}}\times\overline{B}_{r_{1}}\rightarrow \mathbb{R}$, defined by
	\begin{equation*}
		\begin{cases}
			\psi(x,y):=L_{1}w(\abs{x-y})+L_{2}\Big(\lvert x-\hat{x}_{0}\rvert^{2}+\lvert y-\hat{x}_{0} \rvert^{2} \Big),\\
			\Psi\left(x,y\right):=u(x)-u(y)-\psi(x,y).
		\end{cases}
	\end{equation*}
	Let $\left(\hat{x},\hat{y}\right)\in \overline{B}_{r_{1}}\times\overline{B}_{r_{1}}$ be a maximum point for $\Psi(x,y)$, i.e., $\Psi\left(\hat{x},\hat{y}\right)>0$. Note that $\hat{x}\neq\hat{y}$; otherwise the maximum of $\Psi$ would be nonpositive. It follows from $\|u\|_{L^{\infty}(B_{1})}\leq 1$ that
	\begin{equation}\label{5342}
		L_{1}w(\abs{\hat{x}-\hat{y}})+L_{2}\Big(\lvert \hat{x}-\hat{x}_{0}\rvert^{2}+\lvert \hat{y}-\hat{x}_{0} \rvert^{2} \Big)<u(\hat{x})-u(\hat{y})\leq 2\|u\|_{L^{\infty}(B_{1})}\leq 2.
	\end{equation}
This together with the triangle inequality yields that
\begin{equation}\label{534}
	\abs{\hat{x}-\hat{y}}\leq \abs{\hat{x}-\hat{x}_{0}}+\abs{\hat{y}-\hat{x}_{0}}\leq \sqrt{2\left(\lvert \hat{x}-x_{0}\rvert^{2}+\lvert \hat{y}-x_{0} \rvert^{2} \right)}\leq \frac{2}{\sqrt{L_{2}}}.
\end{equation}
Choosing $L_{2}>\frac{8}{(r_{1}-r)^{2}}$ so that $\abs{\hat{x}-\hat{y}}<1$ and
$$|\hat{x}|\leq \lvert \hat{x}-\hat{x}_{0}\rvert+|\hat{x}_{0}|\leq \sqrt\frac{2}{L_{2}}
+r<\frac{r+r_{1}}{2},\quad |\hat{y}|\leq \lvert \hat{y}-\hat{x}_{0}\rvert+|\hat{x}_{0}|\leq \sqrt\frac{2}{L_{2}}+r<\frac{r+r_{1}}{2}.$$
This means that	
$\hat{x},\hat{y}$ belongs to the interior of $B_{(r+r_{1})/2}$.

 Next, we invoke the Crandall-Ishii-Lions lemma (see \cite[Theorem 3.2]{Crandle1}) to assure
	the existence of a limiting subjet $\left(\xi_{\hat{x}},X\right)$ of $u$ at $\hat{x}$ and a limiting superjet $\left(\xi_{\hat{y}},Y\right)$ of $u$ at $\hat{y}$, such that the matrices $X,Y\in S^{d}$ satisfy the matrix inequality
	\begin{equation}\label{434matrix}
		\left(
		\begin{array}{cc}
			X & 0 \\
			0 & -Y \\
		\end{array}
		\right)
		\leq  \left(
		\begin{array}{cc}
			B & -B \\
			-B & B \\
		\end{array}
		\right)+
		(2L_{2}+\epsilon)
		\left(
		\begin{array}{cc}
			I & 0 \\
			0 & I \\
		\end{array}
		\right)
	\end{equation}
	with $\epsilon\in(0,1)$, that only depends on the norm of $B$ and can be made sufficiently small. Here,
	\begin{equation*}
		\begin{split}
			\xi_{\hat{x}}&:=D_{x}\psi(\hat{x},\hat{y})= L_{1}w^{\prime}(\abs{\hat{x}-\hat{y}})\frac{\hat{x}-\hat{y}}{\abs{\hat{x}-\hat{y}}}+2L_{2}(\hat{x}-\hat{x}_{0}),\\
			\xi_{\hat{y}}&:=-D_{y}\psi(\hat{x},\hat{y})=L_{1}w^{\prime}(\abs{\hat{x}-\hat{y}})\frac{\hat{x}-\hat{y}}{\abs{\hat{x}-\hat{y}}}-2L_{2}(\hat{y}-\hat{x}_{0}),
		\end{split}
	\end{equation*}
\begin{equation}\label{538}
	\begin{split}
	B:=L_{1}\left[\frac{w^{\prime}(\abs{\hat{x}-\hat{y}})}{\abs{\hat{x}-\hat{y}}}I+\left(w^{\prime\prime}(\abs{\hat{x}-\hat{y}})-\frac{w^{\prime}(\abs{\hat{x}-\hat{y}})}{\abs{\hat{x}-\hat{y}}}\right)\frac{(\hat{x}-\hat{y})\otimes(\hat{x}-\hat{y})}{\abs{\hat{x}-\hat{y}}^{2}}\right].
	\end{split}
\end{equation}
	Before two viscosity inequalities are given, for simplicity, we denote
	\begin{equation*}
			\eta_1 := \xi_{\hat{x}} + \xi,\quad  \eta_2 = \xi_{\hat{y}} + \xi,
	\end{equation*}	
	\begin{equation*}
	A(\eta_1) := I + (p-2) \dfrac{\eta_1}{|\eta_1|} \otimes \dfrac{\eta_1}{|\eta_1|},\quad 
	A(\eta_2) := I + (p-2) \dfrac{\eta_2}{|\eta_2|} \otimes \dfrac{\eta_2}{|\eta_2|},
	\end{equation*}		
\begin{equation*}
	G_{\xi_{\hat{x}}}(X) := -\operatorname{Tr}(A(\eta_1)X),\quad 
	G_{\xi_{\hat{y}}}(Y) := -\operatorname{Tr}(A(\eta_2)Y).
\end{equation*}	
	Then we have the following two viscosity inequalities
	\begin{equation*}
		\Phi({\eta_{1}},\hat{x})G_{\xi_{\hat{x}}}(X)+\mathcal{H}(\eta_{1},\hat{x})\leq f(\hat{x}),
	\end{equation*}
	\begin{equation*}
		\Phi({\eta_{2}},\hat{y})G_{\xi_{\hat{y}}}(Y)+\mathcal{H}(\eta_{2},\hat{y})\geq f(\hat{y}).
	\end{equation*}
Then it follows that		
	\begin{equation}\label{53555}
		G_{\xi_{\hat{y}}}(Y)-G_{\xi_{\hat{x}}}(X)\geq \frac{f(\hat{y})-\mathcal{H}(\eta_{2},\hat{y})}{\Phi({\eta_{2}},\hat{y})}- \frac{f(\hat{x})-\mathcal{H}(\eta_{1},\hat{x})}{\Phi({\eta_{1}},\hat{x})}=:\mathcal{D}_{1}.
	\end{equation}
{\bf Estimate of $\mathcal{D}_{1}$}. Suppose $\abs{\xi}\geq \mathcal{J}_{0}$, where $\mathcal{J}_{0}$ will be determined later. We first choose $L_{1}>L_{2}$, then it follows from $\Abs{\hat{x}-\hat{x}_{0}}< \frac{1}{2}$ and $0\leq w^{\prime}(s)\leq 1$ for $s>0$ that
	\begin{equation*}
		\abs{\xi_{\hat{x}}}\leq L_{1}w^{\prime}(\abs{\hat{x}-\hat{y}})+2L_{2}\Abs{\hat{x}-\hat{x}_{0}}< 2L_{1}.
	\end{equation*}
	Now we take $\mathcal{J}_{0}:=3L_{1}$ so that
	\begin{align}\label{5xjie}
		|\eta_{1}|\geq\abs{\xi}- \abs{\xi_{\hat{x}}}\geq \mathcal{J}_{0}-2L_{1}=L_{1}>1.
	\end{align}
	In exactly the same way, we get
	\begin{equation}\label{5yjie}
		|\eta_{2}|\geq L_{1}>1.
	\end{equation}
Applying conditions \eqref{12}-\eqref{15}, in combination with \eqref{5xjie}, \eqref{5yjie}, and $0<m\leq p_{\rm min}$, we obtain
	\begin{equation}\label{532}
		\begin{split}
			|\mathcal{D}_{1}|\leq& \frac{\|f\|_{L^{\infty}(B_{1})}+\mathcal{K}+\mathcal{M}\abs{\eta_{2}}^{m}}{\abs{\eta_{2}}^{p(\hat{y})}}+\frac{\|f\|_{L^{\infty}(B_{1})}+\mathcal{K}+\mathcal{M}\abs{\eta_{1}}^{m}}{\abs{\eta_{1}}^{p(\hat{x})}}\\
			\leq&\frac{2\left(\|f\|_{L^{\infty}(B_{1})}+\mathcal{K}\right)}{L_{1}^{p_{\rm min}}}+\mathcal{M}\left(\abs{\eta_{1}}^{m-p_{\rm min}}+\abs{\eta_{2}}^{m-p_{\rm min}}\right)\\
			\leq &2\left(\|f\|_{L^{\infty}(B_{1})}+\mathcal{K}+\mathcal{M}\right).
		\end{split}
	\end{equation}
	Combining \eqref{53555} with \eqref{532}, we have
	\begin{equation}\label{536}
		\begin{split}
			G_{\xi_{\hat{y}}}(Y)-G_{\xi_{\hat{x}}}(X)
			\geq-2\left(\|f\|_{L^{\infty}(B_{1})}+\mathcal{K}+\mathcal{M}\right).
		\end{split}
	\end{equation}
	
	On the other hand, we estimate the upper bound of $G_{\xi_{\hat{y}}}(Y)-G_{\xi_{\hat{x}}}(X)$. Using the linearity of the trace operator, we obtain 
	\begin{equation}\label{upperbound}
		\begin{split}
			G_{\xi_{\hat{y}}}(Y)-G_{\xi_{\hat{x}}}(X)
			=\operatorname{Tr}\left(A(\eta_1)(X - Y)\right) + \operatorname{Tr}\left((A(\eta_1) - A(\eta_2))Y\right)=:I_{1}+I_{2}.
		\end{split}
	\end{equation}
	{\bf Estimate of $I_{1}$.} We apply matrix inequality \eqref{434matrix} to vectors of the form $(z,z)\in\mathbb{R}^{2d}$ with $\abs{z}=1$, to obtain
	\begin{equation}\label{superasati}
		\left\langle(X-Y)z,z\right\rangle\leq \left(4L_{2}+2\epsilon\right)|z|^{2}.
	\end{equation}
	This means that all the eigenvalues of $X-Y$ are less than or equal to $4L_{2}+2\epsilon$. In addition, applying \eqref{434matrix} to the vector $(\nu,-\nu)\in\mathbb{R}^{2d}$ where $\nu:=\frac{\hat{x}-\hat{y}}{\abs{\hat{x}-\hat{y}}}$, 
	we get
	\begin{equation*}
		\begin{split}
			\left\langle(X-Y)\nu,\nu\right\rangle\leq 4\left\langle A\nu,\nu\right\rangle+\left(4L_{2}+2\epsilon\right)|\nu|^{2}=4L_{1}w^{\prime\prime}(\abs{\hat{x}-\hat{y}})+4L_{2}+2\epsilon.
		\end{split}
	\end{equation*}	
	This yields that
	at least one eigenvalue of $X-Y$ is less than $4L_{1}w^{\prime\prime}(\abs{\hat{x}-\hat{y}})+4L_{2}+2\epsilon$. We choose $L_{1}>\frac{4L_{2}+2}{4\beta(1+\beta)w_{0}}$, then it follows from $\abs{\hat{x}-\hat{y}}<1$ and $\beta<1$ that
	\begin{align*}
		4L_{1}w^{\prime\prime}(\abs{\hat{x}-\hat{y}})+4L_{2}+2\epsilon&=-4L_{1}\beta(\beta+1)w_{0}\abs{\hat{x}-\hat{y}}^{\beta-1}+4L_{2}+2\epsilon\\
		&<-4L_{1}\beta(\beta+1)w_{0}+4L_{2}+2<0.
	\end{align*}
	This means that at least one eigenvalue of
	$X-Y$ is negative. Remembering the fact that the eigenvalues of $A(\eta_1)$ belong to $\left[\min\left\{1,p-1\right\},\max\left\{1,p-1\right\}\right]$. Combining the information above, we obtain
	\begin{equation}\label{I1}
		\begin{split}
			I_{1}\leq \sum_{i=1}^{d} \lambda_i(A(\eta_1))\lambda_i(X - Y)\leq C_{1}-\min\left\{1,p-1\right\}4L_{1}\beta(\beta+1)w_{0}\abs{\hat{x}-\hat{y}}^{\beta-1}, 
		\end{split}
	\end{equation}	
where $C_{1}:=\min\left\{1,p-1\right\}\left(4L_{2}+2\right)+\max\left\{1,p-1\right\}(d-1)(4L_{2}+2)$.\\
	{\bf Estimate of $I_{2}$.} Denote $\tilde{\eta}_i:=\frac{\eta_i}{|\eta_i|}$, $i=1,2$. A direct calculation yields that
	\begin{equation*}
		\begin{split}
			A(\eta_1) - A(\eta_2)=(p-2)\left(\tilde{\eta}_1 \otimes (\tilde{\eta}_1 - \tilde{\eta}_2) - \tilde{\eta}_2 \otimes (\tilde{\eta}_2 - \tilde{\eta}_1)\right).
		\end{split}
	\end{equation*}	
A combination of the definition of $\eta_{1}$ and $\eta_{2}$ with \eqref{534} yields that
\begin{equation}
	|\eta_1-\eta_2|=2L_{2}|(\hat{x}-\hat{x}_{0})+(\hat{y}-\hat{x}_{0})|<2L_{2}.
\end{equation}
This together with $\eta_{1},\eta_{2}\geq L_{1}$ leads to
\begin{equation*}
	\begin{split}
		|\tilde{\eta}_1 - \tilde{\eta}_2| = \left| \frac{\eta_1}{|\eta_1|} - \frac{\eta_2}{|\eta_2|} \right| 
		\leq 2\max \left\{  \frac{\abs{\eta_1 - \eta_2}}{|\eta_1|} ,  \frac{\abs{\eta_1 - \eta_2}}{|\eta_2|} \right\} \leq  \frac{4L_2}{L_{1}}.
	\end{split}
\end{equation*}	
	Applying the matrix inequality \eqref{434matrix} to the vectors $(0,z)\in\mathbb{R}^{2d}$ with $\abs{z}=1$,
	we get
	\begin{equation*}
		\left\langle-Yz,z\right\rangle\leq \left\langle Bz,z\right\rangle+2L_{2}+\epsilon.
	\end{equation*}
It follows from the definition of the matrix $B$ in \eqref{538}and $0\leq w^{\prime}(s)\leq 1$ for $s\geq 0$ that $$B\leq L_{1}\frac{w^{\prime}(\abs{\hat{x}-\hat{y}})}{\abs{\hat{x}-\hat{y}}}I\leq L_{1}\abs{\hat{x}-\hat{y}}^{-1}I.$$
Then it follows that
	\begin{equation*}
		\|Y\|\leq L_{1}\abs{\hat{x}-\hat{y}}^{-1}+2L_{2}+\epsilon.
	\end{equation*}
	Combining the previous displays with $L_{1}>L_{2}$, we arrive at
	\begin{equation}\label{I2}
		\begin{split}
			I_{2}& 
			\leq d\|Y\|\|A(\eta_1) - A(\eta_2)\|\\
			&\leq d|p-2| |\tilde{\eta}_1 - \tilde{\eta}_2|\left(|\tilde{\eta}_1|+|\tilde{\eta}_2|\right)\|Y\| \\
			&\leq \frac{8L_2}{L_{1}} d|p-2| \left(L_{1}\abs{\hat{x}-\hat{y}}^{-1}+2L_{2}+1\right)\\
			&\leq 8 d|p-2|\left(2L_{2}+1+L_{2}\abs{\hat{x}-\hat{y}}^{-1}\right).
		\end{split}
	\end{equation}	
	Substituting the estimates \eqref{I1} and \eqref{I2} into \eqref{upperbound}, we obtain
	\begin{equation*}
		\begin{split}
			G_{\xi_{\hat{y}}}(Y)-G_{\xi_{\hat{x}}}(X)\leq&C_{1}-\min\left\{1,p-1\right\}4L_{1}\beta(\beta+1)w_{0}\abs{\hat{x}-\hat{y}}^{\beta-1}\\
			&+8 d|p-2|\left(2L_{2}+1+L_{2}\abs{\hat{x}-\hat{y}}^{-1}\right).
			\end{split}
	\end{equation*}	
	This together with \eqref{536} yields that
	\begin{equation}\label{539}
		\begin{split}
			&4L_{1}\min\left\{1,p-1\right\}\beta(\beta+1)\omega_{0}\abs{\hat{x}-\hat{y}}^{\beta-1}\\
			\leq&C_{1}+16 d|p-2|\left(2L_{2}+1+L_{2}\abs{\hat{x}-\hat{y}}^{-1}\right)+2\left(\|f\|_{L^{\infty}(B_{1})}+\mathcal{K}+\mathcal{M}\right).
		\end{split}
	\end{equation}	
	Thus, if we choose $L_{1}$ large enough, we reach a contradiction with \eqref{539}. 
	
	As has been stated above, we verify that $\mathcal{G}(x_{0})\leq 0$ for each $x_{0}\in B_{r}$, which yields that
	$$|u(x_{0})-u(y_{0})|\leq L_{1}|x_{0}-y_{0}|+L_{2}|x_{0}-y_{0}|^{2}\leq C|x_{0}-y_{0}| $$
	for every $x_{0},y_{0}\in B_{r}$.\\
	(ii) Suppose $\abs{\xi}< \mathcal{J}_{0}$. Let $\zeta\in\mathbb{R}^{d}$ such that $\abs{\zeta}\geq 3\mathcal{J}_{0}$. Then it follows that $\abs{\zeta+\xi}\geq 2\mathcal{J}_{0}>1$. Applying \eqref{12}-\eqref{14} and $p_{\rm min}\geq 0$ to deduce
	\begin{equation}
		\Phi({\zeta+\xi},x)\geq \abs{\zeta+\xi}^{p(x)}\geq \abs{\zeta+\xi}^{p_{\rm min}}>1.
	\end{equation} 
	It follows from \eqref{15} and $0<m\leq p_{\rm min}$ that 
	$$\Abs{\frac{{\mathcal{H}({\zeta+\xi},x)}}{\Phi({\zeta+\xi},x)}}\leq  \frac{\mathcal{K}+\mathcal{M}\abs{\zeta+\xi}^{m}}{\abs{\zeta+\xi}^{p_{\rm min}}}\leq\mathcal{K}+\mathcal{M}.$$
	Therefore, we can see that $u$ is a viscosity solution of
	\begin{equation*}
		\begin{cases}
			\mathcal{P}^{+}_{\lambda,\Lambda}(D^{2}u)+\left(\mathcal{K}+\mathcal{M}\right)+\|f\|_{L^{\infty}(B_{1})}\geq 0\quad\quad {\rm in} \;\; B_{1}\cap\left\{|Du|\geq 3\mathcal{J}_{0}\right\}, \\
			\mathcal{P}^{-}_{\lambda,\Lambda}(D^{2}u)-\left(\mathcal{K}+\mathcal{M}\right)-\|f\|_{L^{\infty}(B_{1})}\leq 0 \quad\quad {\rm in} \;\; B_{1}\cap\left\{|Du|\geq 3\mathcal{J}_{0}\right\}.
		\end{cases}
	\end{equation*}	
	At this point, we are in an exact position to apply \cite[Theorem 1.1]{Silvestre2016JEMS} 
	to know that $u$ is local H\"{o}lder continuous. This completes the proof.	
\end{proof}

To proceed, we deal with the scenario $0\leq p_{\rm min}<m\leq 1+p_{\rm min}$. In this case, we need to impose the additional assumption to restrain the growth of the term $\mathcal{M}\abs{Du+\xi}^{m-p_{\rm min}}$.
\begin{proposition}\label{prop3.2}
	Let $u\in C(B_{1})$ be a normalized viscosity solution of \eqref{531} under assumptions \eqref{12}-\eqref{15} with $0\leq p_{\rm min}<m\leq p_{\rm min}$. There exists a universal constant $\kappa_{0}>0$ such that if
	\begin{equation}\label{control}
		\mathcal{M}\left(\abs{\xi}^{m-p_{\rm min}}+1\right)\leq \kappa
	\end{equation}
	for some $\kappa\leq \kappa_{0}$, then $u\in C_{\rm loc}^{0,\gamma}(B_{1})$ for some $\gamma\in(0,1)$. Furthermore, there exists a universal constant $C>0$ (independent of $\xi$), such that
		$$|u(x)-u(y)|\leq C|x-y|^{\gamma}$$
	for all $x,y\in B_{3/4}$.
\end{proposition}
\begin{proof}
	Let $M_{0}:=\left(\frac{\kappa_{0}}{\mathcal{M}}\right)^{\frac{1}{m-p_{\rm min}}}$. It follows from $m>p_{\rm min}$ and \eqref{control} that
\begin{equation}\label{bound}
	M_{0}\geq 1\quad {\rm and} \quad \abs{\xi}\leq M_{0}.
\end{equation}
The proof is similar to that of Proposition \ref{prop3.1}. Here we substitute $w(t)$ in \eqref{541} with $t^{\gamma}$ for $\gamma\in(0,1)$ to be fixed. As argued before, let us fix $0<r<r_{1}<1$, it suffices to show that there exist constants $L_{1},L_{2}> 1$ such that
\begin{equation*}
	\mathcal{G}_{1}(x_{0}):=\sup\limits_{(x,y)\in B_{r_{1}}\times B_{r_{1}}}\left\{u(x)-u(y)-L_{1}\lvert x-y \rvert^{\gamma}-L_{2}\Big(\lvert x-x_{0}\rvert^{2}+\lvert y-x_{0} \rvert^{2} \Big)\right\}\leq 0
\end{equation*}
for each $x_{0}\in B_{r}$. We argue by contradiction by assuming that there exists $\hat{x}_{0}\in B_{r}$ so that $\mathcal{G}_{1}(\hat{x}_{0})>0$ for all $L_{1},L_{2}>1$.  Consider $\psi_{1},\Psi_{1}:\overline{B}_{r_{1}}\times\overline{B}_{r_{1}}\rightarrow \mathbb{R}$, defined by
\begin{equation*}
	\begin{cases}
		\psi_{1}(x,y):=L_{1}\lvert x-y \rvert^{\gamma}+L_{2}\Big(\lvert x-x_{0}\rvert^{2}+\lvert y-x_{0} \rvert^{2} \Big),\\
		\Psi_{1}\left(x,y\right):=u(x)-u(y)-\psi_{1}(x,y).
	\end{cases}
\end{equation*}
Let $\left(\hat{x},\hat{y}\right)\in \overline{B}_{r_{1}}\times\overline{B}_{r_{1}}$ be a maximum point for $\Psi_{1}(x,y)$, i.e., $\Psi_{1}\left(\hat{x},\hat{y}\right)>0$. Note that $\hat{x}\neq\hat{y}$; otherwise the maximum of $\Psi_{1}$ would be nonpositive. By the same arguments as before, we can show $\abs{\hat{x}-\hat{y}}\leq \frac{2}{\sqrt{L_{2}}}$ and choose $L_{2}>\frac{8}{(r_{1}-r)^{2}}$ so that $\abs{\hat{x}-\hat{y}}<1$ and $\hat{x},\hat{y}\in B_{(r+r_{1})/2}$.

	 We are in a position to apply the Crandall-Ishii-Lions lemma (see \cite[Theorem 3.2]{Crandle1}) to assure
	the existence of a limiting subjet $\left(\xi_{\hat{x}}^{\prime},X\right)$ of $u$ at $\hat{x}$ and a limiting superjet $\left(\xi_{\hat{y}}^{\prime},Y\right)$ of $u$ at $\hat{y}$, such that the matrices $X,Y\in S^{d}$ satisfy the matrix inequality
	\begin{equation}\label{matrix}
		\left(
		\begin{array}{cc}
			X & 0 \\
			0 & -Y \\
		\end{array}
		\right)
		\leq  \left(
		\begin{array}{cc}
			B^{\prime} & -B^{\prime} \\
			-B^{\prime} & B^{\prime} \\
		\end{array}
		\right)+
		(2L_{2}+\epsilon)
		\left(
		\begin{array}{cc}
			I & 0 \\
			0 & I \\
		\end{array}
		\right)
	\end{equation}
	with $\epsilon\in(0,1)$, that only depends on the norm of $B^{\prime}$ and can be made sufficiently small. Here,
	\begin{equation*}
			\xi_{\hat{x}}^{\prime}:=\gamma L_{1}(\hat{x}-\hat{y})\lvert \hat{x}-\hat{y}\rvert^{\gamma -2}+2L_{2}(\hat{x}-x_{0}),  \quad \xi_{\hat{y}}^{\prime}:=\gamma L_{1}(\hat{x}-\hat{y})\lvert \hat{x}-\hat{y}\rvert^{\gamma -2}-2L_{2}(\hat{y}-x_{0}),
	\end{equation*}
\begin{equation}\label{Adingyi}
	B^{\prime}:=L_{1}\gamma\left[(\gamma-2)\abs{\hat{x}-\hat{y}}^{\gamma-4}\left((\hat{x}-\hat{y})\otimes(\hat{x}-\hat{y})\right)+\abs{\hat{x}-\hat{y}}^{\gamma-2}I\right].
\end{equation}
	Before two viscosity inequalities are given, for simplicity, we denote
\begin{equation*}
	\eta_1 ^{\prime}:= \xi_{\hat{x}}^{\prime} + \xi,\quad  \eta_2^{\prime} = \xi_{\hat{y}}^{\prime} + \xi,
\end{equation*}	
\begin{equation*}
	A(\eta_1^{\prime}) := I + (p-2) \dfrac{\eta_1^{\prime}}{|\eta_1^{\prime}|} \otimes \dfrac{\eta_1^{\prime}}{|\eta_1^{\prime}|},\quad 
	A(\eta_2^{\prime}) := I + (p-2) \dfrac{\eta_2^{\prime}}{|\eta_2^{\prime}|} \otimes \dfrac{\eta_2^{\prime}}{|\eta_2^{\prime}|},
\end{equation*}		
\begin{equation*}
	G_{\xi_{\hat{x}}}^{\prime}(X) := -\operatorname{Tr}(A(\eta_1^{\prime})X),\quad 
	G_{\xi_{\hat{y}}}^{\prime}(Y) := -\operatorname{Tr}(A(\eta_2^{\prime})Y).
\end{equation*}	
Then we have the following two viscosity inequalities
\begin{equation*}
	\Phi({\eta_{1}^{\prime}},\hat{x})G_{\xi_{\hat{x}}}^{\prime}(X)+\mathcal{H}(\eta_{1}^{\prime},\hat{x})\leq f(\hat{x}),
\end{equation*}
\begin{equation*}
	\Phi({\eta_{2}^{\prime}},\hat{y})G_{\xi_{\hat{y}}}^{\prime}(Y)+\mathcal{H}(\eta_{2}^{\prime},\hat{y})\geq f(\hat{y}).
\end{equation*}
Then it follows that		
\begin{equation}\label{535}
	G_{\xi_{\hat{y}}}^{\prime}(Y)-G_{\xi_{\hat{x}}}^{\prime}(X)\geq \frac{f(\hat{y})-\mathcal{H}(\eta_{2}^{\prime},\hat{y})}{\Phi({\eta_{2}^{\prime}},\hat{y})}- \frac{f(\hat{x})-\mathcal{H}(\eta_{1}^{\prime},\hat{x})}{\Phi({\eta_{1}^{\prime}},\hat{x})}=:\mathcal{D}_{1}^{\prime}.
\end{equation}
{\bf Estimate of $\mathcal{D}_{1}^{\prime}$}. We first claim that
\begin{equation}\label{shangxiajie}
	\frac{\gamma L_{1}}{2}\Abs{\hat{x}-\hat{y}}^{\gamma-1}\leq\abs{\xi_{\hat{x}}^{\prime}},\abs{\xi_{\hat{y}}^{\prime}}\leq 2\gamma L_{1}\Abs{\hat{x}-\hat{y}}^{\gamma-1}.
\end{equation}	
In fact, choose $L_{1}>\frac{L_{2}2^{\gamma}}{\gamma(r_{1}-r)^{\gamma-2}}$, it follows from $\Abs{\hat{x}-\hat{y}}\leq\frac{r_{1}-r}{2}$, 
and $\gamma<1$ that
$$2L_{2}\Abs{\hat{x}-x_{0}}\leq 2L_{2}\frac{r_{1}-r}{2}< \frac{\gamma L_{1}}{2}\left(\frac{r_{1}-r}{2}\right)^{\gamma-1}\leq \frac{\gamma L_{1}}{2}\Abs{\hat{x}-\hat{y}}^{\gamma-1}.
$$	
This along with the triangle inequality leads to that
\begin{equation*}
	\abs{\xi_{\hat{x}}^{\prime}}\geq \gamma L_{1}\lvert \hat{x}-\hat{y}\rvert^{\gamma -1}-2L_{2}\Abs{\hat{x}-x_{0}}\geq 
	 \frac{\gamma L_{1}}{2}\Abs{\hat{x}-\hat{y}}^{\gamma-1},
\end{equation*}
\begin{equation*}
	\abs{\xi_{\hat{x}}^{\prime}}\leq \gamma L_{1}\lvert \hat{x}-\hat{y}\rvert^{\gamma -1}+2L_{2}\Abs{\hat{x}-x_{0}}
	\leq 2\gamma L_{1}\Abs{\hat{x}-\hat{y}}^{\gamma-1}.
\end{equation*}
In exactly the same way, we derive
\begin{equation*}
	\frac{\gamma L_{1}}{2}\Abs{\hat{x}-\hat{y}}^{\gamma-1}\leq \abs{\xi_{\hat{y}}^{\prime}}\leq 2\gamma L_{1}\Abs{\hat{x}-\hat{y}}^{\gamma-1}.
\end{equation*}	
We proceed to select $L_{1}>\frac{4M_{0}}{\gamma}$, then a combination of \eqref{bound} with \eqref{shangxiajie} and $\Abs{\hat{x}-\hat{y}}^{\gamma-1}>1$ yields that
\begin{equation}\label{81}
	\begin{split}
		|\eta_{1}^{\prime}|&\leq |\xi_{\hat{x}}^{\prime}| + |\xi| \leq 2\gamma L_1 |\hat{x}-\hat{y}|^{\gamma-1} + M_0 \leq 2\gamma L_1 |\hat{x}-\hat{y}|^{\gamma-1} + \frac{\gamma L_1}{4} \\
		&\leq 2\gamma L_1 |\hat{x}-\hat{y}|^{\gamma-1} + \gamma L_1 |\hat{x}-\hat{y}|^{\gamma-1} = 3\gamma L_1 |\hat{x}-\hat{y}|^{\gamma-1},
	\end{split} 
\end{equation}
\begin{equation}\label{82}
	\begin{split}
		|\eta_{1}^{\prime}| &\geq |\xi_{\hat{x}}^{\prime}| - |\xi| \geq \frac{\gamma L_1}{2} |\hat{x}-\hat{y}|^{\gamma-1} - M_0 \geq \frac{\gamma L_1}{2} |\hat{x}-\hat{y}|^{\gamma-1} - \frac{\gamma L_1}{4} |\hat{x}-\hat{y}|^{\gamma-1} \\
		&= \frac{\gamma L_1}{4} |\hat{x}-\hat{y}|^{\gamma-1} > M_0 \, |\hat{x}-\hat{y}|^{\gamma-1} > 1.
	\end{split}
\end{equation}
By the same way, we can deduce
\begin{equation}\label{83}
1<M_{0}\Abs{\hat{x}-\hat{y}}^{\gamma-1}\leq |\eta_{2}^{\prime}|\leq 3\gamma L_{1}\Abs{\hat{x}-\hat{y}}^{\gamma-1}.
\end{equation}
Applying \eqref{12}-\eqref{15}, in combination with
\eqref{81}-\eqref{83} and $p_{\rm min}<m\leq 1+p_{\rm min}$, we arrive at
\begin{equation}\label{D1}
	\begin{split}
		\abs{\mathcal{D}_{1}^{\prime}}&\leq \frac{\|f\|_{L^{\infty}(B_{1})}+\mathcal{K}+\mathcal{M}\abs{\eta_{2}^{\prime}}^{m}}{\abs{\eta_{2}^{\prime}}^{p(\hat{y})}}+\frac{\|f\|_{L^{\infty}(B_{1})}+\mathcal{K}+\mathcal{M}\abs{\eta_{1}^{\prime}}^{m}}{\abs{\eta_{1}^{\prime}}^{p(\hat{x})}}\\
		&\leq 2\|f\|_{L^{\infty}(B_{1})}+2\mathcal{K}+
		2\mathcal{M}\left(3\gamma L_{1}\Abs{\hat{x}-\hat{y}}^{\gamma-1}\right)^{m-p_{\rm min}}\\
		&\leq  2\left(\|f\|_{L^{\infty}(B_{1})}+\mathcal{K}\right)+6\mathcal{M}L_{1}\Abs{\hat{x}-\hat{y}}^{\gamma-1}.
	\end{split}
\end{equation}

On the other hand, applying the linearity of the trace operator to obtain 
\begin{equation*}
	\begin{split}
	G_{\xi_{\hat{y}}}^{\prime}(Y)-G_{\xi_{\hat{x}}}^{\prime}(X)
	=\operatorname{Tr}\left(A(\eta_1^{\prime})(X - Y)\right) + \operatorname{Tr}\left((A(\eta_1^{\prime}) - A(\eta_2^{\prime}))Y\right)=:I_{1}^{\prime}+I_{2}^{\prime}.
	\end{split}
\end{equation*}
{\bf Estimate of $I_{1}^{\prime}$.} For vectors of the form $(z,z)\in\mathbb{R}^{2d}$ with $\abs{z}=1$, we apply the matrix inequality \eqref{matrix} to obtain
\begin{equation*}
	\left\langle(X-Y)z,z\right\rangle\leq \left(4L_{2}+2\epsilon\right)|z|^{2}.
\end{equation*}
This means that all the eigenvalues of $X-Y$ are less than or equal to $4L_{2}+2\epsilon$. In addition, applying \eqref{matrix} to the vector $(\nu,-\nu)\in\mathbb{R}^{2d}$ where $\nu:=\frac{\hat{x}-\hat{y}}{\abs{\hat{x}-\hat{y}}}$, 
we get
\begin{equation*}
	\begin{split}
		\left\langle(X-Y)\nu,\nu\right\rangle\leq 4\left\langle A\nu,\nu\right\rangle+\left(4L_{2}+2\epsilon\right)|\nu|^{2}=4L_{1}\gamma(\gamma-1)\abs{\hat{x}-\hat{y}}^{\gamma-2}+4L_{2}+2\epsilon.
	\end{split}
\end{equation*}	
This yields that
at least one eigenvalue of $X-Y$ is less than $4L_{1}\gamma(\gamma-1)\abs{\hat{x}-\hat{y}}^{\gamma-2}+4L_{2}+2\epsilon$. We choose $L_{1}>\frac{4L_{2}+2}{4\gamma(1-\gamma)}$, then it leads to
$$4L_{1}\gamma(\gamma-1)\abs{\hat{x}-\hat{y}}^{\gamma-2}+4L_{2}+2\epsilon<4L_{1}\gamma(\gamma-1)+4L_{2}+2<0.$$
This means that at least one eigenvalue of
$X-Y$ is negative. Then it follows that
\begin{equation*}
	\begin{split}
		I_{1}^{\prime} \leq \sum_{i=1}^{d} \lambda_i(A(\eta_1^{\prime}))\lambda_i(X - Y)\leq C_{1}+ 4L_{1}\gamma(\gamma-1)\min\left\{1,p-1\right\}\abs{\hat{x}-\hat{y}}^{\gamma-2},
	\end{split}
\end{equation*}	
where $C_{1}:=\min\left\{1,p-1\right\}\left(4L_{2}+2\right)+\max\left\{1,p-1\right\}(d-1)(4L_{2}+2)$.\\
{\bf Estimate of $I_{2}^{\prime}$.} Denote $\tilde{\eta}_i^{\prime}:=\frac{\eta_i^{\prime}}{|\eta_i^{\prime}|}$, $i=1,2$. We have
\begin{equation*}
	\begin{split}
		A(\eta_1^{\prime}) - A(\eta_2^{\prime})=(p-2)\left(\tilde{\eta}_1^{\prime} \otimes (\tilde{\eta}_1^{\prime} - \tilde{\eta}_2^{\prime}) - \tilde{\eta}_2^{\prime} \otimes (\tilde{\eta}_2^{\prime} - \tilde{\eta}_1^{\prime})\right)
	\end{split}
\end{equation*}	
	Further, applying \eqref{matrix} to the vectors $(0,z)\in\mathbb{R}^{2d}$ with $\abs{z}=1$,
we obtain
\begin{equation}\label{855}
	\left\langle-Yz,z\right\rangle\leq \left\langle B^{\prime}z,z\right\rangle+(2L_{2}+\epsilon)|z|^{2}.
\end{equation}
It follows from the definition of the matrix $B^{\prime}$ in \eqref{Adingyi} that $B^{\prime}\leq L_{1}\gamma\abs{\hat{x}-\hat{y}}^{\gamma-2}I$. This together with \eqref{855} yields that
\begin{equation}
	\|Y\|\leq L_{1}\gamma\abs{\hat{x}-\hat{y}}^{\gamma-2}+2L_{2}+\epsilon.
\end{equation}
Moreover, applying the definitions of $\eta_1$ and $\eta_2$, together with \eqref{81}-\eqref{83}, we arrive at
\begin{equation*}
	\begin{split}
		|\tilde{\eta}_1^{\prime} - \tilde{\eta}_2^{\prime}|
		\leq 2\max \left\{  \frac{\abs{\eta_1^{\prime} - \eta_2^{\prime}}}{|\eta_1^{\prime}|} ,  \frac{\abs{\eta_1^{\prime} - \eta_2^{\prime}}}{|\eta_2^{\prime}|} \right\} 
		\leq \frac{16L_2}{\gamma L_{1}|\hat{x}-\hat{y}|^{\gamma-1}}\leq \frac{16}{\gamma |\hat{x}-\hat{y}|^{\gamma-1}},
	\end{split}
\end{equation*}	
Gathering the previous estimates with $\abs{\hat{x}-\hat{y}}^{1-\gamma}<1$, we obtain
\begin{equation*}
	\begin{split}
		I_{2}^{\prime}&\leq d|p-2| |\tilde{\eta}_1^{\prime} - \tilde{\eta}_2^{\prime}|\left(|\tilde{\eta}_1^{\prime}|+|\tilde{\eta}_2^{\prime}|\right)\|Y\| \\
		&\leq \frac{32 d|p-2|}{\gamma|\hat{x}-\hat{y}|^{\gamma-1}} \left(L_{1}\gamma\abs{\hat{x}-\hat{y}}^{\gamma-2}+2L_{2}+1\right)\\
		&\leq 32 d|p-2|\gamma^{-1}\left(2L_{2}+1\right)+32 d|p-2| L_{1}\abs{\hat{x}-\hat{y}}^{-1}.
	\end{split}
\end{equation*}	
Then it follows that
\begin{equation*}
	\begin{split}
	G_{\xi_{\hat{y}}}^{\prime}(Y)-G_{\xi_{\hat{x}}}^{\prime}(X)\leq& C_{1}+ 4L_{1}\gamma(\gamma-1)\min\left\{1,p-1\right\}\abs{\hat{x}-\hat{y}}^{\gamma-2}\\
	&+32 d|p-2|\gamma^{-1}\left(2L_{2}+1\right)+32 d|p-2| L_{1}\abs{\hat{x}-\hat{y}}^{-1}.
	\end{split}
\end{equation*}	
This along with \eqref{D1} yields that
\begin{equation*}
	\begin{split}
	&2L_{1}\abs{\hat{x}-\hat{y}}^{\gamma-2}\left(2\gamma(1-\gamma)\min\left\{1,p-1\right\}-16d|p-2|\abs{\hat{x}-\hat{y}}^{1-\gamma}-3\mathcal{M}\Abs{\hat{x}-\hat{y}}\right)	\\
	\leq& C_{1}+32 d|p-2|\gamma^{-1}\left(2L_{2}+1\right)+2\left(\|f\|_{L^{\infty}(B_{1})}+\mathcal{K}\right)
	\end{split}
\end{equation*}	
Choose $$L_{2}\geq \max\left\{4\left(\frac{32 d|p-2|}{\min\left\{1,p-1\right\}\gamma(1-\gamma)}\right)^{2/(1-\gamma)},\left(\frac{12\mathcal{M}}{\gamma(1-\gamma)\min\left\{1,p-1\right\}}\right)^{2}\right\}.$$ With this choice, we apply $\abs{\hat{x}-\hat{y}}\leq \frac{2}{\sqrt{L_{2}}}$ to derive
\begin{equation*}
	\begin{split}
		\abs{\hat{x}-\hat{y}}^{1-\gamma}\leq \left(\frac{2}{\sqrt{L_{2}}}\right)^{1-\gamma}\leq \frac{\min\left\{1,p-1\right\}\gamma(1-\gamma)}{32d|p-2|},\quad \abs{\hat{x}-\hat{y}}\leq \frac{2}{\sqrt{L_{2}}}\leq \frac{\min\left\{1,p-1\right\}\gamma(1-\gamma)}{6\mathcal{M}}.
	\end{split}
\end{equation*}	
Gathering the previous estimates with $\abs{\hat{x}-\hat{y}}^{\gamma-2}>1$, we obtain
\begin{equation}\label{key11}
	\begin{split}
		2L_{1}\min\left\{1,p-1\right\}\gamma(1-\gamma)< C_{1}+32 d|p-2|\gamma^{-1}\left(2L_{2}+1\right)+2\left(\|f\|_{L^{\infty}(B_{1})}+\mathcal{K}\right).
	\end{split}
\end{equation}	
Finally, we choose $L_{1}\geq \frac{C_{1}+32 d|p-2|\gamma^{-1}\left(2L_{2}+1\right)+2\left(\|f\|_{L^{\infty}(B_{1})}+\mathcal{K}\right)}{2\min\left\{1,p-1\right\}\gamma(1-\gamma)}$ and reach a contradiction with \eqref{key11}.	The proof is complete.	
\end{proof}
The following lemma treats the singular case, where we only need to consider $\xi =0$.
\begin{proposition}\label{jin2}
Assume \eqref{12}-\eqref{15} hold with $-1<p_{\rm min}<0$.
	Let $u\in C(B_{1})$ be a normalized viscosity solution of \eqref{531} with $|\xi|=0$. 
	Then $u\in C_{\rm loc}^{0,1}(B_{1})$. In addition, there holds that
	$$|u(x)-u(y)|\leq C|x-y|$$
	for all $x,y\in B_{3/4}$, where $C>0$ is a universal constant.
\end{proposition}
\begin{proof}
The proof follows closely the proof of Proposition \ref{prop3.1}-(i) with $\xi=0$. Thus we use the same notation as in the proof of Proposition \ref{prop3.1} and only concentrate 
on the differences. We first choose $L_{1}>\max \{4,L_{2}\}$ and $L_{2}\geq \left(2^{2+\beta}\omega_{0}(1+\beta)\right)^{2/\beta}$ so that
	$\abs{\xi_{\hat{x}}}< 2L_{1}$ and
	\begin{align*}
		\abs{\xi_{\hat{x}}}\geq& L_{1}w^{\prime}(\abs{\hat{x}-\hat{y}})-\frac{L_{2}}{2}=L_{1}\left(1-w_{0}(1+\beta)\abs{\hat{x}-\hat{y}}^{\beta}\right)-\frac{L_{2}}{2}\\
		\geq&L_{1}\left(1-\omega_{0}(1+\beta)\left(\frac{2}{\sqrt{L_{2}}}\right)^{\beta}\right)-\frac{L_{2}}{2}\geq \frac{3L_{1}}{4}-\frac{L_{2}}{2}\geq \frac{L_{1}}{4}>1.
	\end{align*}
	Also, we have $1<\abs{\xi_{\hat{y}}}\leq 2L_{1}$. Since $0<\beta<1$, there exists a constant $\theta\in(\beta,1)$.
	Applying the assumptions \eqref{12}-\eqref{15}, in combination with $-1< p_{\rm min}<0$, $0<m\leq 1+p_{\rm min}$, and $\Abs{\hat{x}-\hat{y}}^{\theta-1}>1$, we deduce that
	\begin{equation}\label{998}
		\begin{split}
			G_{\xi_{\hat{y}}}(Y)-G_{\xi_{\hat{x}}}(X) 
			&\geq-\frac{\|f\|_{L^{\infty}(B_{1})}+\mathcal{K}+\mathcal{M}\abs{\xi_{\hat{x}}}^{m}}{\abs{\xi_{\hat{x}}}^{p(\hat{x})}}-\frac{\|f\|_{L^{\infty}(B_{1})}+\mathcal{K}+\mathcal{M}\abs{\xi_{\hat{y}}}^{m}}{\abs{\xi_{\hat{y}}}^{p(\hat{y})}}\\
			 &\geq-2\frac{\|f\|_{L^{\infty}(B_{1})}+\mathcal{M}}{(2L_{1})^{p_{\rm min}}}-\mathcal{M}\left(\abs{\xi_{\hat{x}}}^{m-p_{\rm min}}+\abs{\xi_{\hat{y}}}^{m-p_{\rm min}}\right)\\
			&\geq -4\left(\|f\|_{L^{\infty}(B_{1})}+\mathcal{K}\right)L_{1}^{-p_{\rm min}}-2\mathcal{M}(2L_{1})^{m-p_{\rm min}}\\
			&\geq -4\left(\|f\|_{L^{\infty}(B_{1})}+\mathcal{K}\right)L_{1}^{-p_{\rm min}}-4\mathcal{M}L_{1}\Abs{\hat{x}-\hat{y}}^{\theta-1}.
		\end{split}
	\end{equation}
Since $\xi=0$, we have $\eta_1=\xi_{\hat{x}}$ and $\eta_2=\xi_{\hat{y}}$. It follows from $|\eta_1-\eta_2|\leq 2L_{2}$ and $\abs{\eta_{1}},\abs{\eta_2}\geq \frac{L_{1}}{4}$ that
\begin{equation*}
	\begin{split}
		|\tilde{\eta}_1 - \tilde{\eta}_2| = \left| \frac{\eta_1}{|\eta_1|} - \frac{\eta_2}{|\eta_2|} \right| 
		\leq 2\max \left\{  \frac{\abs{\eta_1 - \eta_2}}{|\eta_1|} ,  \frac{\abs{\eta_1 - \eta_2}}{|\eta_2|} \right\} 
		\leq \frac{16L_2}{L_{1}},
	\end{split}
\end{equation*}	
This together with $L_{1}>L_{2}$ yields that
\begin{equation*}
	\begin{split}
		I_{2}&=\operatorname{Tr}\left((A(\eta_1) - A(\eta_2))Y\right)\\
		&\leq d|p-2| |\tilde{\eta}_1 - \tilde{\eta}_2|\left(|\tilde{\eta}_1|+|\tilde{\eta}_2|\right)\|Y\| \\
		&\leq \frac{32L_2}{L_{1}} d|p-2| \left(L_{1}\abs{\hat{x}-\hat{y}}^{-1}+2L_{2}+1\right)\\
		&\leq 32 d|p-2| \left(2L_{2}+1\right)+32 d|p-2|L_{2}\abs{\hat{x}-\hat{y}}^{-1}.
	\end{split}
\end{equation*}	
Recall that 
\begin{equation}\label{999}
	\begin{split}
		I_{1}\leq \sum_{i=1}^{d} \lambda_i(A(\eta_1))\lambda_i(X - Y)\leq C_{1}-\min\left\{1,p-1\right\}4L_{1}\beta(\beta+1)w_{0}\abs{\hat{x}-\hat{y}}^{\beta-1}, 
	\end{split}
\end{equation}	
where $C_{1}:=\min\left\{1,p-1\right\}\left(4L_{2}+2\right)+\max\left\{1,p-1\right\}(d-1)(4L_{2}+2)$. Combining the information above, we obtain
\begin{equation*}
	\begin{split}
		G_{\xi_{\hat{y}}}(Y)-G_{\xi_{\hat{x}}}(X)\leq& C_{1}-\min\left\{1,p-1\right\}4L_{1}\beta(\beta+1)w_{0}\abs{\hat{x}-\hat{y}}^{\beta-1}\\
		&+32 d|p-2| \left(2L_{2}+1\right)+32 d|p-2|L_{2}\abs{\hat{x}-\hat{y}}^{-1}.
	\end{split}
\end{equation*}	
A combination of \eqref{998} with \eqref{999} yields that
	\begin{equation*}
		\begin{split}
		&4L_{1}\abs{\hat{x}-\hat{y}}^{\beta-1}\left(\min\left\{1,p-1\right\}\beta(\beta+1)w_{0}-\mathcal{M}\Abs{\hat{x}-\hat{y}}^{\theta-\beta}\right)\\	
		\leq& C_{1}+4\left(\|f\|_{L^{\infty}(B_{1})}+\mathcal{K}\right)L_{1}^{-p_{\rm min}}+32 d|p-2| \left(2L_{2}+1\right)+32 d|p-2|L_{2}\abs{\hat{x}-\hat{y}}^{-1}.
	\end{split}
	\end{equation*}
We now select $L_{2}\geq 4\left(\frac{2\mathcal{M}}{\min\left\{1,p-1\right\}\beta(\beta+1)w_{0}}\right)^{2/(\theta-\beta)}$, then it follows from $\abs{\hat{x}-\hat{y}}\leq \frac{2}{\sqrt{L_{2}}}$ and $\theta\in(\beta,1)$ that
	\begin{equation*}
		\mathcal{M}\Abs{\hat{x}-\hat{y}}^{\theta-\beta}\leq \mathcal{M} \left(\frac{2}{\sqrt{L_{2}}}\right)^{\theta-\beta}\leq \frac{1}{2}\min\left\{1,p-1\right\}\beta(\beta+1)w_{0}.
	\end{equation*}
	Then it follows that
		\begin{equation*}
		\begin{split}
			&2L_{1}\min\left\{1,p-1\right\}\beta(\beta+1)w_{0}\abs{\hat{x}-\hat{y}}^{\beta-1}\\	
			\leq& C_{1}+4\left(\|f\|_{L^{\infty}(B_{1})}+\mathcal{K}\right)L_{1}^{-p_{\rm min}}+32 d|p-2| \left(2L_{2}+1\right)+32 d|p-2|L_{2}\abs{\hat{x}-\hat{y}}^{-1}
		\end{split}
	\end{equation*}
	In view of $-1<p_{\rm min}<0$, taking $L_{1}$ large enough, we obtain a contradiction. This completes the proof.
\end{proof}
\subsection{Approximation lemma}
This section is solely dedicated to the proof of a key approximation lemma for the degenerate case via compactness arguments, which plays a paramount role in our forthcoming geometric argument.
\begin{lemma}\label{bijin} 
	Let $u\in C(B_{1})$ be a normalized viscosity solution of \eqref{531} under assumptions \eqref{12}-\eqref{15} with $p_{\rm min}\geq 0$. Then, for any $\varepsilon>0$, there exists $\sigma\in (0,1)$ 
	such that if
	\begin{equation*}
		\max\left\{\|f\|_{L^{\infty}(B_{1})},\mathcal{K},\mathcal{M}\left(\abs{\xi}^{(m-p_{\rm min})_{+}}+1\right) \right\}\leq \sigma,
	\end{equation*}
	then one can find a function $v\in C_{\rm loc}^{1,\beta_{0}}(B_{3/4})$ for some $\beta_{0}\in(0,1)$, satisfying
	\begin{equation*}
		\|u-v\|_{L^{\infty}(B_{1/2})}\leq \varepsilon.
	\end{equation*}
\end{lemma}
\begin{proof}
	We use a contradiction argument. If the claim fails, then there exist $\varepsilon_{0}>0$ and sequences
	of functions $\{\Phi_{j}\}_{j\in \mathbb{N}}$, $\{\mathcal{H}_{j}\}_{j\in \mathbb{N}}$, $\{f_{j}\}_{j\in \mathbb{N}}$, $\{u_{j}\}_{j\in \mathbb{N}}$ and a sequence of vectors $\{\xi_{j}\}_{j\in \mathbb{N}}$ satisfying
	\begin{itemize}
		\item   [{\rm$({{\rm i}})$}] $u_{j}\in C({B_{1}})$ is a viscosity solution of 
		\begin{equation}\label{72}
			-\Phi_{j}({Du_{j}+\xi_{j}}, x)\Delta_{p,\xi_{j}}^{N} u_{j}+\mathcal{H}_{j}(Du_{j}+\xi_{j},x) =f_{j}(x) \quad  \text{in} \quad B_{1}
		\end{equation}
	with $ \|u_{j}\|_{L^{\infty}(B_{1})}\leq 1$, where $f_{j}\in C({B_{1}})$;
		\item [{\rm$({{\rm ii}})$}] $\Phi_{j}$ satisfies \eqref{12} and \eqref{13} with $p_{j}(\cdot),q_{j}(\cdot)\in C(B_{1})$, $0\leq p_{\rm min}\leq p_{j}(x)\leq q_{j}(x)\leq q_{\rm max}<\infty$ and $0\leq a_{j}(\cdot)\in C(B_{1})$;
		\item [ {\rm$({{\rm iii}})$}] $\mathcal{H}_{j}:\mathbb{R}^{d} \times B_{1}\rightarrow \mathbb{R}$ is continuous and there exist constants $\mathcal{K}_{j},\mathcal{M}_{j}>0$ such
		that
		\begin{equation}\label{73}
			|H_{j}(t,x)|\leq \mathcal{K}_{j}+\mathcal{M}_{j}|t|^{m} \quad  {\rm for\; every\;}(t,x) \in\mathbb{R}^{d}\times B_{1};
		\end{equation}
		\item [ {\rm$({{\rm iv}})$}] and
		\begin{equation}\label{71}
			\max\left\{\|f_{j}\|_{L^{\infty}(B_{1})}, \mathcal{K}_{j},\mathcal{M}_{j}\left(\abs{\xi_{j}}^{(m-p_{\rm min})_{+}}+1\right) \right\}\leq \frac{1}{j}.
		\end{equation}	
	\end{itemize}
	However, for any $v\in C_{\rm loc}^{1,\beta_{0}}(B_{3/4})$, it holds
	 \begin{equation}\label{340}
			\|u_{j}-v\|_{L^{\infty}(B_{1/2})}>\varepsilon \quad {\rm for\;any}\;j\in\mathbb{N}.
	\end{equation}
	
 It follows from Propositions \ref{prop3.1} and \ref{prop3.2} that the sequence $\{u_{j}\}_{j\in\mathbb{N}}\subset C_{\rm loc}^{0,\tau}(B_{1})$ for some $\tau\in (0,1)$. By applying Arzel${\rm \grave{a}}$-Ascoli theorem, we conclude that, up to a subsequence,  $u_{j}$ converges locally uniformly in $B_{1}$ to some continuous function $u_{\infty}$ in the $C^{0}$-topology. 

If the sequence $\{\xi_{j}\}_{j\in \mathbb{N}}$ is unbounded, take a subsequence, still denoted by $\left\{\xi_{j}\right\}_{j\in\mathbb{N}}$, for which $\abs{\xi_{j}}\rightarrow {\infty}$ as $j\rightarrow\infty$ and then a  converging subsequence from $e_{j}:=\frac{\xi_{j}}{|\xi_{j}|}$, $e_{j}\rightarrow e_{\infty}$. 
Next our goal is to prove $u_{\infty}$ is a viscosity solution to
\begin{equation}\label{jixianfangcheng1}
	-\Delta u_{\infty}+(p-2)\left\langle D^{2}u_{\infty}e_{\infty},e_{\infty}\right\rangle =0 \quad {\rm in}\;\; B_{3/4}.
\end{equation}
We only prove that $u_{\infty}$ is a viscosity supersolution, as its subsolution counterpart is entirely analogous. Let $\varphi$ be a test function such that $u_{\infty}-\varphi$ has a local minimum at $\overline{x}\in B_{3/4}$. Without loss of generality, we may assume that $\abs{\overline{x}}=u_{\infty}(\overline{x}) =\varphi(\overline{x})= 0$ and that $\varphi$ is a quadratic polynomial, namely,
$$\varphi(x)=\frac{1}{2}\left\langle Mx,x\right\rangle+\left\langle b,x\right\rangle,$$
where $A\in S^{d}$ and $b\in \rn$. Since $u_{j}\rightarrow u_{\infty}$ locally uniformly in $B_{1}$, we see that, for $j$ sufficiently large, the quadratic polynomial
$$\varphi_{j}(x):=\frac{1}{2}\left\langle M(x-x_{j}),x-x_{j}\right\rangle+\left\langle b,x-x_{j}\right\rangle+u_{j}(x_{j})$$
touches $u_{j}$ from below at $x_{j}$ belonging to a small neighbourhood of the origin. Since $u_{j}$ is a viscosity solution of \eqref{72}, we have
\begin{equation}\label{342}
	-\Phi_{j}(\abs{b+\xi_{j}}, x_{j})\left[\operatorname{Tr}(M)+(p-2)\left\langle M \frac{b+\xi_{j}}{|b+\xi_{j}|}, \frac{b+\xi_{j}}{|b+\xi_{j}|} \right\rangle\right]+
	\mathcal{H}_{j}(b+\xi_{j},x_{j}) \geq f_{j}(x_{j}).
\end{equation}
Since $\abs{\xi_{j}}\rightarrow {\infty}$ as $j\rightarrow\infty$, there exists $j^{\star}\in\mathbb{N}$ so that $\abs{\xi_{j}}\geq2\max\{1,\abs{b}\}$ for all $j\geq j^{\star}$. Then it follows 
\begin{equation}\label{Step1xiajie}
	\abs{b+\xi_{j}}\geq \abs{\xi_{j}}-|b|\geq \frac{1}{2}\abs{\xi_{j}}\geq 1.
\end{equation}
Using condition (ii), in combination with \eqref{Step1xiajie}, \eqref{71} and $p_{\rm min}\geq 0$, we deduce that
\begin{equation}\label{344}
	\Abs{\frac{f_{j}(x_{j})}{\Phi_{j}({b+\xi_{j}}, x_{j})}}\leq  \frac{\|f_{j}\|_{L^{\infty}(B_{1})}}{\abs{b+\xi_{j}}^{p_{\rm min}}}\leq \frac{2^{p_{\rm min}}}{j\abs{\xi_{j}}^{p_{\rm min}}}.
\end{equation}
Applying the following basic inequality
$$(\upsilon_{1}+\upsilon_{2})^{m}\leq 2^{m}(\upsilon_{1}^{m}+\upsilon_{2}^{m}) \quad {\rm for\;all} \;\,\upsilon_{1},\upsilon_{2}\geq 0,\;m>0,$$
in combination with \eqref{Step1xiajie}, \eqref{71}, \eqref{73} and $|\xi_{j}|\geq 2|b|$, we arrive at
\begin{equation*}
	\begin{split}
	\Abs{\frac{\mathcal{H}_{j}\left(b+\xi_{j},x_{j}\right)}{\Phi_{j}({b+\xi_{j}}, x_{j})}}&\leq  \frac{\mathcal{K}_{j}+\mathcal{M}_{j}	\abs{b+\xi_{j}}^{m}}{\abs{b+\xi_{j}}^{p_{\rm min}}}\\
	&\leq \frac{2^{p_{\rm min}}}{j\abs{\xi_{j}}^{p_{\rm min}}}+
	\frac{2^{p_{\rm min}}\mathcal{M}_{j}\abs{b+\xi_{j}}^{m}}{\abs{\xi_{j}}^{p_{\rm min}}}\\
	&\leq \frac{2^{p_{\rm min}}}{j\abs{\xi_{j}}^{p_{\rm min}}}+\frac{2^{m+p_{\rm min}}\abs{\xi_{j}}^{m-p_{\rm min}}}{j\left(\abs{\xi_{j}}^{(m-p_{\rm min})_{+}}+1\right)}\left(\Abs{\frac{b}{\xi_{j}}}^{m}+1\right)\\
	&\leq \frac{2^{p_{\rm min}}}{j\abs{\xi_{j}}^{p_{\rm min}}}+\frac{2^{m+1+p_{\rm min}}\abs{\xi_{j}}^{m-p_{\rm min}}}{j\left(\abs{\xi_{j}}^{(m-p_{\rm min})_{+}}+1\right)}.
	\end{split}
\end{equation*}
It follows from $|\xi_{j}|\geq 2$ that
\begin{equation}\label{345}
		\Abs{\frac{\mathcal{H}_{j}\left(b+\xi_{j},x_{j}\right)}{\Phi_{j}({b+\xi_{j}}, x_{j})}}\leq
	\begin{cases}
			\frac{2^{p_{\rm min}}}{j\abs{\xi_{j}}^{p_{\rm min}}}+\frac{2^{m+p_{\rm min}}}{j \abs{\xi_{j}}^{p_{\rm min}-m}}& \text{if}\;\; 0< m\leq p_{\rm min}, \\
		\frac{2^{p_{\rm min}}}{j\abs{\xi_{j}}^{p_{\rm min}}}+\frac{2^{m+p_{\rm min}+1}}{j} &\text{if}\;\; p_{\rm min}<m\leq 1+ p_{\rm min}.
	\end{cases}
\end{equation}
A combination of \eqref{342} with \eqref{344} and \eqref{345} yields that
\begin{equation}\label{sj}
	-\left[\operatorname{Tr}(M)+(p-2)\left\langle M \frac{b|\xi_{j}|^{-1}+e_{j}}{|b|\xi_{j}|^{-1}+e_{j}|}, \frac{b|\xi_{j}|^{-1}+e_{j}}{|b|\xi_{j}|^{-1}+e_{j}|} \right\rangle\right]\geq -\frac{2^{1+p_{\rm min}}}{j\abs{\xi_{j}}^{p_{\rm min}}}-\frac{2^{m+p_{\rm min}+1}}{j \abs{\xi_{j}}^{(p_{\rm min}-m)_{+}}}.
\end{equation}
By passing the limit $j\rightarrow \infty$, we obtain that 
\begin{equation*}
	-\operatorname{Tr}(M)-(p-2)\left\langle M e_{\infty},e_{\infty}\right\rangle\geq 0.
\end{equation*}
Thus, $u_{\infty}$ is a viscosity supersolution of \eqref{jixianfangcheng1}. From Remark \ref{xianxingyizhituoyuan}, we know that \eqref{jixianfangcheng1} is a linear uniformly elliptic equation with constant
coefficients. Then it follows from the classical regularity result in \cite[Corollary 5.7]{Caff1} that $u_{\infty}\in C^{1,\beta_{2}}_{\rm loc}(B_{3/4})$ for some $0<\beta_{2} <1$.

On the other hand, if the sequence $\{\xi_{j}\}_{j\in\mathbb{N}}$ is bounded, then we may assume  $\xi_{j}\rightarrow \xi_{\infty}$ as $j\rightarrow \infty$ (up to a subsequence). Then it follows that $\xi_{j}+b\rightarrow \xi_{\infty}+b \; {\rm as}\; j\rightarrow \infty.$ We aim to prove that $u_{\infty}$ is a viscosity solution to
\begin{equation}\label{boundqingxing}
	-\Delta u_{\infty} - (p-2)\left\langle D^2u_{\infty}\frac{Du_{\infty}+ \xi_{\infty}}{|Du_{\infty} + \xi_{\infty}|}, \frac{Du_{\infty} + \xi_{\infty}}{|Du_{\infty} + \xi_{\infty}|}\right\rangle = 0 \quad \text{in } B_{3/4}.
\end{equation}
Let $\varphi$ be a test function such that $u_{\infty}-\varphi$ has a local minimum at $\overline{x}\in B_{3/4}$. Without loss of generality, we may assume that $\abs{\overline{x}}=u_{\infty}(\overline{x}) =\varphi(\overline{x})= 0$ and that $\varphi$ is a quadratic polynomial, namely,
$$\varphi(x)=\frac{1}{2}\left\langle Mx,x\right\rangle+\left\langle b,x\right\rangle.$$
At this point, we consider two cases: $\abs{b+\xi_{\infty}}\neq 0$ or $\abs{b+\xi_{\infty}}= 0$. 

In the case of $\abs{b+\xi_{\infty}}\neq 0$. Note that $\Abs{\xi_{j}+b}\geq \frac{1}{2}\Abs{\xi_{\infty}+b}$ for $j$ large enough. Applying conditions (ii)-(iv) and $0\leq p_{\rm min}\leq q_{\rm max}$ to obtain 
\begin{equation*}
	\begin{split}
		\Abs{\frac{f_{j}(x_{j})-\mathcal{H}_{j}\left(b+\xi_{j},x_{j}\right)}{\Phi_{j}({b+\xi_{j}}, x_{j})}}\leq \frac{2^{1+q_{\rm max}}}{jC_{2}}+ \frac{2^{q_{\rm max}}}{jC_{2}} \frac{\abs{b+\xi_{j}}^{m}}{\abs{\xi_{j}}^{(m-p_{\rm min})_{+}}+1}\leq \frac{2^{q_{\rm max}}}{jC_{2}}\left(2+\abs{b+\xi_{j}}^{m}\right),
	\end{split}
\end{equation*}
where $C_{2}:=\min\left\{|b+\xi_{\infty}|^{p_{\rm min}},|b+\xi_{\infty}|^{q_{\rm max}}\right\}$.
This along with \eqref{342} leads to
\begin{equation}\label{347}
	-\operatorname{Tr}(M)-(p-2)\left\langle M\frac{b+ \xi_{\infty}}{|b + \xi_{\infty}|}, \frac{b + \xi_{\infty}}{|b + \xi_{\infty}|}\right\rangle \geq 0.
\end{equation}
Hence $u_{\infty}$ is a viscosity supersolution to \eqref{boundqingxing}. 

For the latter case where $\Abs{b+\xi_{\infty}}= 0$, there are two situations can occur: $b=-\xi_{\infty}$ with $|b|,\abs{\xi_{\infty}}>0$ or $|b|=\abs{\xi_{\infty}}=0$. \\
{\bf Case 1.} $b=-\xi_{\infty}$ with $|b|,\abs{\xi_{\infty}}>0$.  If there exists a subsequence still indexed by $j$ such that $|D\varphi(x_j) + \xi_j| > 0$ for all $j$ in the subsequence, then we can repeat the process above to obtain \eqref{347}. If such a subsequence does not exist, then we aim to show that
\begin{equation}\label{551}
	-\Delta\varphi(\overline{x}) - (p-2)\lambda_{\min}(D^2\varphi(\overline{x})) \geq 0, \quad p \geq 2,
\end{equation}
and
\begin{equation}\label{552}
	-\Delta\varphi(\overline{x}) - (p-2)\lambda_{\max}(D^2\varphi(\overline{x})) \geq 0, \quad 1 < p < 2.
\end{equation}
In the following, we only give the proof of \eqref{551}, since the proof of \eqref{552} is analogous. We assume by contradiction that
\begin{equation}\label{553}
-\Delta\varphi(\overline{x}) - (p-2)\lambda_{\min}(D^2\varphi(\overline{x})) < 0, \quad p \geq 2,
\end{equation}
which implies that matrix $M$ has at least one positive eigenvalue. Otherwise, we have $\operatorname{Tr}(M) +(p-2)\lambda_{\min}(M)\leq 0$, which is a contradiction to \eqref{553}. Let $\rn=T\oplus Q$ be the orthogonal sum, where $T:={\rm span}\{e_{1},e_{2},...,e_{k}\}$ is the invariant space composed of those eigenvectors corresponding to
positive eigenvalues of $M$, and $Q:=\{y\in \rn:\left\langle y,\eta  \right\rangle=0\;{\rm for\;all}\;\eta\in T\}$.
 Let $\gamma>0$ and
\begin{equation}\label{model3}
	\varphi_{\gamma}(x):=\varphi(x)+\gamma \Abs{P_{T}(x)}=\frac{1}{2}\left\langle Mx,x\right\rangle+\left\langle b,x\right\rangle+\gamma \Abs{P_{T}(x)},
\end{equation}
where $P_{T}$ stands for the orthogonal projection over $T$. Since $u_{j}\rightarrow u_{\infty}$ locally uniformly in $B_{1}$ and $\varphi$ touches $u_{\infty}$ from below at the origin, then, for $\gamma$ small enough, $\varphi_{\gamma}$ touches
$u_{j}$ from below at a point $x_{j}^{\gamma}$ belonging to a small neighbourhood of 0. Moreover, there holds that, up to a subsequence, $x_{j}^{\gamma}\rightarrow x_{*}$ for some $x_{*}\in B_{3/4}$ as $j\rightarrow\infty$. Now we examine two scenarios: $P_{T}\left(x_{j}^{\gamma}\right)=0$ and $P_{T}\left(x_{j}^{\gamma}\right)\neq 0$.

First, we consider $P_{T}\left(x_{j}^{\gamma}\right)=0$.
Notice that
\begin{equation*}
	\tilde{\varphi}_{\gamma}(x):=\varphi(x)+\gamma \left\langle e,P_{T}(x)\right\rangle
\end{equation*}
touches $u_{j}$ from below at $x_{j}^{\gamma}$ for every $e\in \mathbb{S}^{d-1}$ (i.e., $|e|=1$). Through a direct calculation, we have
\begin{equation*}
	D\tilde{\varphi}_{\gamma}(x_{j}^{\gamma})=Mx_{j}^{\gamma}+b+\gamma P_{T}(e),\quad D^{2}\tilde{\varphi}_{\gamma}(x_{j}^{\gamma})=M.
\end{equation*}
For simplicity, we denote $D_j := Mx_j^\gamma + {b} + \gamma P_{T}(e)  + \xi_j.$ 
It follows from the definition of $-\Delta_{p,\xi_j}^N$ and \eqref{553} that
\begin{equation}\label{558}
	-\Delta_{p,\xi_j}^N \tilde{\varphi}_{\gamma}(x_{j}^{\gamma}) = -\operatorname{Tr}(M) - (p-2)\left\langle M \frac{D_{j}}{|D_{j}|}, \frac{D_{j}}{|D_{j}|}\right\rangle
	\leq -\operatorname{Tr}(M) - (p-2)\lambda_{\min}(M)<0,
\end{equation}
where we use the fact $\left\langle M \frac{D_{j}}{|D_{j}|}, \frac{D_{j}}{|D_{j}|}\right\rangle\geq \lambda_{\min}(M)$. We choose $e\in T\cap \mathbb{S}^{d-1}$ such that $P_{T}(e)=e$.
Since $u_{j}$  is a viscosity solution of \eqref{72}, we get
\begin{equation}\label{353}
	-\Phi_{j}({D_{j}}, x_{j}^{\gamma})\Delta _{p,\xi_{j}}^{N} \tilde{\varphi}_{\gamma}(x_{j}^{\gamma})+\mathcal{H}_{j}\left(D_{j},x_{j}^{\gamma}\right) \geq f_{j}(x_{j}^{\gamma}).
\end{equation}
If $Mx_{*}=0$, then for $j$ sufficiently large, we have
\begin{equation*}
	\Abs{Mx_{j}^{\gamma}+b+\xi_{j}}\leq \frac{\gamma}{2}.
\end{equation*}
This along with the triangle inequality yields
\begin{equation}\label{xiaoyu1}
	\frac{\gamma}{2}\leq \Abs{Mx_{j}^{\gamma}+b+\gamma e+\xi_{j}}\leq  \frac{3\gamma}{2}.
\end{equation}
A combination of (ii)-(iv) and \eqref{xiaoyu1} yields that
\begin{equation*}
\Abs{\frac{f_{j}(x_{j}^{\gamma})-\mathcal{H}_{j}\left(D_{j},x_{j}^{\gamma}\right)}{\Phi_{j}(D_{j}, x_{j}^{\gamma})}}\leq \frac{2^{1+q_{\rm max}}}{j\gamma^{q_{\rm max}}}+\frac{2^{q_{\rm max}-m}(3\gamma)^{m}}{j\gamma^{q_{\rm max}}}.
\end{equation*}
This together with \eqref{353} leads to
\begin{equation*}
	\begin{split}
		-\Delta_{p,\xi_j}^N \tilde{\varphi}_{\gamma}(x_{j}^{\gamma})\geq \frac{f_{j}(x_{j}^{\gamma})-\mathcal{H}_{j}\left(D_{j},x_{j}^{\gamma}\right)}{\Phi_{j}(D_{j}, x_{j}^{\gamma})}\geq -\frac{2^{1+q_{\rm max}}}{j\gamma^{q_{\rm max}}}-\frac{2^{q_{\rm max}-m}(3\gamma)^{m}}{j\gamma^{q_{\rm max}}}.
	\end{split}
\end{equation*}
By passing to the limit $j\rightarrow \infty$, we reach a contradiction with \eqref{558}. Thus, we prove \eqref{551}. \\
On the other hand, if $\Abs{Mx_{*}}>0$, we start off by considering the case in which $T\equiv \rn$ and select $e\in \mathbb{S}^{d-1}$ such that
\begin{equation*}
	\Abs{Mx_{*}+\gamma P_{T}(e)}=\Abs{Mx_{*}+\gamma e}>0.
\end{equation*}
For $j$ large enough, we have
\begin{equation}\label{case1}
	\frac{1}{2}\Abs{Mx_{*}+\gamma e}\leq \Abs{Mx_{j}^{\gamma}+\gamma e}\leq \frac{3}{2}\Abs{Mx_{*}+\gamma e}\quad {\rm and}\quad \abs{\xi_{j}+b}\leq \frac{1}{8}\Abs{Mx_{*}+\gamma e}.
\end{equation}
Furthermore, if $T\neq \rn$, then we choose $e\in \mathbb{S}^{d-1}\cap T^{\perp}$ such that
\begin{equation*}
	\Abs{Mx_{*}+\gamma P_{T}(e)}=\Abs{Mx_{*}}>0.
\end{equation*}
Again for $j$ large enough, we have
\begin{equation}\label{case2}
	\frac{1}{2}\Abs{Mx_{*}}\leq \Abs{Mx_{j}^{\gamma}}\leq \frac{3}{2}\Abs{Mx_{*}}\quad {\rm and}\quad \abs{\xi_{j}+b}\leq \frac{1}{8}\Abs{Mx_{*}}.
\end{equation}
Thus, using either \eqref{case1} or \eqref{case2}, we arrive at
\begin{equation*}
	0<\frac{3}{8}\Abs{Mx_{*}+\gamma P_{T}(e)}\leq\Abs{Mx_{j}^{\gamma}+b+\gamma P_{T}(e)+\xi_{j}}\leq \frac{13}{8}\Abs{Mx_{*}+\gamma P_{T}(e)}.
\end{equation*}
By the same arguments as before, we have
\begin{equation*}
	\begin{split}
		-\Delta_{p,\xi_j}^N \tilde{\varphi}_{\gamma}(x_{j}^{\gamma})\geq -\frac{1}{jC_{3}}\left(2+\left(\frac{13}{8}|Mx_{*}+\gamma P_{T}(e)|\right)^{m}\right),
	\end{split}
\end{equation*}
where $C_{3}:=\min\left\{\left(\frac{3}{8}|Mx_{*}+\gamma P_{T}(e)|\right)^{p_{\rm min}},\left(\frac{3}{8}|Mx_{*}+\gamma P_{T}(e)|\right)^{q_{\rm max}}\right\}$.
By passing to the limit $j\rightarrow \infty$, we also obtain a contradiction with \eqref{558} and hence prove \eqref{551}.

Next, let us consider $P_{T}\left(x_{j}^{\gamma}\right)\neq 0$.
Note that $\Abs{P_{T}(x)}$ is smooth and convex in a small neighbourhood of $x_{j}^{\gamma}$. Because of $P_{T}$ being a projection, then
\begin{equation}\label{feifuding}
	\Abs{P_{T}(x_{j}^{\gamma})}D\left(\Abs{P_{T}(x_{j}^{\gamma})}\right)=P_{T}(x_{j}^{\gamma})\quad {\rm and} \quad D^{2}\left(\Abs{P_{T}(x_{j}^{\gamma})}\right)\;\,{\rm is \;nonnegative\; definite}.
\end{equation}
Then it follows that
$$D{\varphi}_{\gamma}(x_{j}^{\gamma})= Mx_j^\gamma+ {b} + \gamma D(|P_T(x_j^\gamma)|),\quad D^{2}{\varphi}_{\gamma}(x_{j}^{\gamma})=M+\gamma D^{2}\left(\Abs{P_{T}(x_{j}^{\gamma})}\right).$$
 For simplicity, we denote
$E_j := Mx_j^\gamma+ {b} + \gamma D(|P_T(x_j^\gamma)|) + \xi_j$ and $Q_{j}:=D^{2}\left(\Abs{P_{T}(x_{j}^{\gamma})}\right)$. Then applying the definition of $-\Delta_{p,\xi_j}^N$, in combination with \eqref{553} and \eqref{feifuding}, we deduce that
\begin{equation}\label{899}
	\begin{split}
		&-\Delta_{p,\xi_j}^N {\varphi}_{\gamma}(x_{j}^{\gamma})\\
		=& -\operatorname{Tr}(M)- (p-2)\left\langle M \frac{E_j}{|E_j|}, \frac{E_j}{|E_j|}\right\rangle - \gamma\left\{\operatorname{Tr}(Q_j) + (p-2)\left\langle Q_j \frac{E_j}{|E_j|}, \frac{E_j}{|E_j|}\right\rangle\right\}\\
		\leq& -\operatorname{Tr}(M) - (p-2)\lambda_{\min}(M)-\gamma\bigg(\operatorname{Tr}(Q_j) + (p-2)\lambda_{\min}(Q_{j})\bigg)<0.
	\end{split}
\end{equation}
Moreover, we have the following viscosity inequality
\begin{equation*}
	-\Phi_{j}({E_{j}}, x_{j}^{\gamma})\Delta_{p,\xi_j}^N {\varphi}_{\gamma}(x_{j}^{\gamma})
	+\mathcal{H}_{j}\left(E_{j},x_{j}^{\gamma}\right) \geq f_{j}(x_{j}^{\gamma}).
\end{equation*}
As in the case when $P_{T}\left(x_{j}^{\gamma}\right)= 0$, we can consider the scenarios that $Mx_{*}=0$ and $Mx_{*}\neq 0$ and derive a contradiction with \eqref{899}. Therefore, we show \eqref{551} holds.\\
{\bf Case 2.} $b=\xi_{\infty}=0$. In this case, the procedures become easier and is very analogous to Case 1 and so we skip the details here.

At this stage, we have shown
that $u_{\infty}$ is a viscosity solution of \eqref{boundqingxing}. It follows from the well-known regularity results in Lemma \ref{111tiaohezhengzexing} that $u_{\infty}\in C_{\rm loc}^{1,\beta_{1}}(B_{3/4})$ for some $\beta_{1}\in(0,1)$.

Finally, setting $\beta_{0}:=\min\left\{\beta_{1},\beta_{2}\right\}$, then  $u_{\infty}\in C_{\rm loc}^{1,\beta_{0}}(B_{3/4})$. Consequently, we choose $v=u_{\infty}$ and reach a contradiction with \eqref{340} for $j$ sufficiently large. This completes the proof.
\end{proof}
\subsection{Proof of Theorem \ref{thm1}}\label{section5}
This section is dedicated to prove Theorem \ref{thm1} concerning the local gradient regularity. 
Building upon the approximation lemma, we are in a position to show that $u\in C^{1,\alpha^{\prime}}$ at the origin in the degenerate case.
\begin{proposition}\label{prop3.4}
	Let $u\in C(B_{1})$ be a normalized viscosity solution to \eqref{55mainmodel} under assumptions \eqref{12}-\eqref{15} with $p_{\rm min}\geq 0$. Let us fix 
	\begin{equation}\label{zhibiao}
		\alpha^{\prime}\in (0,\beta_{0})\cap \left(0,\frac{1}{1+q_{\rm max}}\right],
	\end{equation}
	where $\beta_{0}$ is given by Lemma \ref{bijin}. There exist universal constants $\delta>0$ and $\rho\in(0,\frac{1}{2})$	such that if
	\begin{equation}\label{811}
		\max\left\{\|f\|_{L^{\infty}(B_{1})},\mathcal{K},\mathcal{M}\right\} \leq \delta,
	\end{equation} 
then $u$ is $C^{1,\alpha^{\prime}}$ at the origin with the estimate
\begin{equation}\label{zhengzexing}
	\sup\limits_{x\in B_{r}}\Abs{u(x)-\left(u(0)+Du(0)\cdot x\right)}\leq Cr^{1+\alpha^{\prime}}\quad {\rm for \;all\;} r\in (0,\rho],
\end{equation}
where $C$ is a universal positive constant.
\end{proposition}
\begin{proof}
	We aim to establish the existence of universal constants $0<\rho<\frac{1}{2}$, $C>0$, and a sequence of affine functions 
	 $$l_{j}(x)=a_{j}+b_{j}\cdot x$$
	 with $\left\{a_{j}\right\}_{j\in \mathbb{N}}\subset \mathbb{R}$ and $\left\{b_{j}\right\}_{j\in \mathbb{N}}\subset \rn$, satisfying, for all $j\in\mathbb{N}$, the following two estimates:
	\begin{equation}\label{711}
		\|u-l_{j}\|_{L^{\infty}(B_{\rho^{j}})}\leq \rho^{j(1+\alpha^{\prime})},
	\end{equation}
	\begin{equation}\label{712}
		|a_{j}-a_{j-1}|+\rho^{j-1}|b_{j}-b_{j-1}|\leq C\rho^{(j-1)(1+\alpha^{\prime})}.
	\end{equation}
	As usual, we prove this by means of an induction argument.
	For clarity of presentation, we divide the proof into three steps.\\
	{\bf Step 1. Basis of induction.} We first find an affine function $l_{1}$ and a universal constant $0<\rho<\frac{1}{2}$ satisfying 
	\begin{equation}\label{651}
		\|u-l_{1}\|_{L^{\infty}(B_{\rho})}\leq \rho^{1+\alpha^{\prime}}.
	\end{equation}
	 By Lemma \ref{bijin}, there exists a function $v\in C_{\rm loc}^{1,\beta_{0}}(B_{3/4})$ such that
	\begin{equation}\label{666}
		\|u-v\|_{L^{\infty}(B_{1/2})}\leq \varepsilon,
	\end{equation}
where $\varepsilon>0$ to be selected later. Remember that the existence of such a function $v$ is guaranteed by Lemma \ref{bijin}, provided that $\delta>0$ is small enough.
	
	According to the $C^{1,\beta_{0}}$-regularity of $v$, we have
	\begin{equation}\label{zhengzexing34}
		\sup\limits_{x\in B_{\rho}}\Abs{v(x)-\left(v(0)+Dv(0)\cdot x\right)}\leq C\rho^{1+\beta_{0}}\quad {\rm for \;all\;} \rho\in\left(0,\frac{1}{2}\right),
	\end{equation}
	\begin{equation*}
		|v(0)|+|Dv(0)|\leq C,
	\end{equation*}
	where $C>0$ is a universal constant. Let us denote
	\begin{equation*}
		l_{1}(x):=a_{1}+b_{1}\cdot x:=v(0)+Dv(0)\cdot x.
	\end{equation*}
	A combination of the triangle inequality  with \eqref{666} and \eqref{zhengzexing34} yields that
	\begin{equation}\label{652}
		\sup\limits_{x\in B_{\rho}}\Abs{u(x)-l_{1}(x)}\leq \sup\limits_{x\in B_{\rho}}\Abs{u(x)-v(x)}+\sup\limits_{x\in B_{\rho}}\Abs{v(x)-l_{1}(x)}\leq \varepsilon+C\rho^{1+\beta_{0}}.
	\end{equation}	
	In light of $\alpha^{\prime}<\beta_{0}$, we take $0<\rho<\frac{1}{2}$ such that
		\begin{equation}\label{653}
	\rho\leq \left(\frac{1}{2C}\right)^{1/(\alpha^{\prime}-\beta_{0})} \quad {\rm and}\quad \varepsilon:=\frac{1}{2}\rho^{1+\alpha^{\prime}}.
	\end{equation}	
	Consequently, combining \eqref{652} with \eqref{653} yields the desired result \eqref{651}. Once we fix the value of $\varepsilon$ here, the quantity $\sigma$ in Lemma \ref{bijin} is determined accordingly. Let $\delta>0$ be small enough such that
\begin{equation}\label{deltaxuanze}
	L\delta \left((2C)^{(m-p_{\rm min})_{+}}+1\right)\leq \sigma
\end{equation}	
with $\sigma$ appearing in the statement of Lemma \ref{bijin}.
To conclude this step, let $a_{0}=b_{0}=0$. These choices and \eqref{651} verify that estimates \eqref{711} and \eqref{712} are satisfied for $j=1$.\\
{\bf Step 2. Induction process}.  Suppose that the conclusion holds true for $j=1,2,...,k$. Now we are going to show that the claim also holds for $j=k+1$. To this end, we introduce an auxiliary function $u_{k}:B_{1}\rightarrow \mathbb{R}$ as
\begin{equation*}
	u_{k}(x):=\frac{u\left(\rho^{k}x\right)-l_{k}\left(\rho^{k}x\right)}{\rho^{k(1+\alpha^{\prime})}}.
\end{equation*}
It is easy to check that $u_{k}$ solves
\begin{equation*} 
	-\Phi_{k}({D{u_{k}}+\xi_{k}}, x)\Delta_{p,\xi_{k}}^{N}u_{k}+\mathcal{H}_{k}(D{u_{k}}+\xi_{k}, x)	= {f_{k}}(x) \quad \text{in} \quad  B_{1}
\end{equation*}
in the viscosity sense, where
	\begin{align*}
	{\Phi_{k}}(t,x):=&\rho^{-k\alpha^{\prime} p_{k}(x)}\Phi\left(\rho^{k\alpha^{\prime}}t,\rho^{k}x\right),\\
	{\Upsilon_{k}}(|t|,x):=&|t|^{p_{k}(x)}+{a_{k}}(x)|t|^{q_{k}(x)},\\
	{a_{k}}(x):=&\rho^{-k\alpha^{\prime}(p_{k}(x)-q_{k}(x))}a(\rho^{k} x),\\
	\mathcal{H}_{k}(t,x):=&\rho^{k(1-\alpha^{\prime}(1+p_{k}(x)))}\mathcal{H}\left(\rho^{k\alpha^{\prime}}t,\rho^{k}x\right),\\
	{f_{k}}(x):=&\rho^{k(1-\alpha^{\prime}(1+p_{k}(x)))}f(\rho^{k}x),\\
	p_{k}(x):=&p(\rho^{k}x),\quad q_{k}(x):=q(\rho^{k}x), \quad \xi_{k}:=\rho^{-k\alpha^{\prime}}b_{k}.
\end{align*}
Note that
\begin{equation*}
	K_{1}	\Upsilon_{k}(|t|,x)\leq	{\Phi_{k}}(t,x) \leq K_{2}	\Upsilon_{k}(|t|,x)\quad {\rm for}\;\,(t,x)\in\mathbb{R}^{d}\times B_{1},
\end{equation*}
and $0\leq p_{\rm min}\leq p_{k}(x)\leq q_{k}(x)\leq q_{\rm max}<\infty$.			
By induction assumption, we have
$$\|{u_{k}}\|_{L^{\infty}\left(B_{1}\right)}\leq 1.$$
It follows from \eqref{15} and $\rho\in(0,\frac{1}{2})$ that
\begin{equation*}
	|\mathcal{H}_{k}(t,x)|\leq \rho^{k(1-\alpha^{\prime}(1+q_{\rm max}))}\left(\mathcal{K}+\mathcal{M}\rho^{k\alpha^{\prime} m}|t|^{m}\right)=:\mathcal{K}_{k}+\mathcal{M}_{k}|t|^{m}.
\end{equation*} 
According to $\alpha^{\prime}\in \left(0,\frac{1}{1+q_{\rm max}}\right]$ and \eqref{811}, we deduce that
\begin{equation*}
	\|{f_{k}}\|_{L^{\infty}\left(B_{1}\right)}\leq \rho^{k(1-\alpha^{\prime}(1+q_{\rm max}))}\|{f}\|_{L^{\infty}\left(B_{1}\right)}\leq \delta,
\end{equation*}
\begin{equation*}
	\mathcal{K}_{k}=\rho^{k\left(1-\alpha^{\prime}(1+q_{\rm max})\right)}\mathcal{K}\leq \delta.
\end{equation*}
Now we analyze the quantity $\mathcal{M}_{k}\left(\abs{\xi_{k}}^{(m-p_{\rm min})_{+}}+1\right)$. Applying \eqref{811} and $\xi_{k}=\rho^{-k\alpha^{\prime}}b_{k}$ to obtain
\begin{equation}\label{Hamilliang}
	\mathcal{M}_{k}\left(\abs{\xi_{k}}^{(m-p_{\rm min})_{+}}+1\right)\leq \delta \rho^{k\left(1-\alpha^{\prime}(1+q_{\rm max}-m)\right)-k\alpha^{\prime}(m-p_{\rm min})_{+}}\left(\abs{b_{k}}^{(m-p_{\rm min})_{+}}+1\right).
\end{equation}		
If $0<m\leq p_{\rm min}$,
by virtue of \eqref{Hamilliang} and $1-\alpha^{\prime}(1+q_{\rm max}-m)>0$, we get
$$\mathcal{M}_{k}\left(\abs{\xi_{k}}^{(m-p_{\rm min})_{+}}+1\right)\leq 2\delta \rho^{k\left(1-\alpha^{\prime}(1+q_{\rm max}-m)\right)}\leq 2\delta.$$
If $p_{\rm min}<m\leq 1+p_{\rm min}$, then
\begin{align*}
	\mathcal{M}_{k}\left(\abs{\xi_{k}}^{(m-p_{\rm min})_{+}}+1\right)\leq \delta \rho^{k\left(1-\alpha^{\prime}(1+q_{\rm max}-p_{\rm min})\right)}\left(\abs{b_{k}}^{m-p_{\rm min}}+1\right)
	\leq \delta\left(\abs{b_{k}}^{m-p_{\rm min}}+1\right),
\end{align*}
where we use the fact that $1-\alpha^{\prime}(1+q_{\rm max}-p_{\rm min})>0$ in the last inequality.
It follows from the induction hypothesis \eqref{712} and $\rho<\frac{1}{2}$ that		
\begin{align*}
	|b_{k}|\leq |b_{0}|+\sum_{j=1}^{k}\abs{b_{j}-b_{j-1}}\leq C\sum_{j=1}^{k}\rho^{\alpha^{\prime}(j-1)}\leq \frac{C}{1-\rho^{\alpha^{\prime}}}\leq 2C.
\end{align*}
In summary, for $0<m\leq 1+p_{\rm min}$, we have
\begin{align*}
	\mathcal{M}_{k}\left(\abs{\xi_{k}}^{(m-p_{\rm min})_{+}}+1\right)\leq \delta \left((2C)^{(m-p_{\rm min})_{+}}+1\right).
\end{align*}
As a consequence, combining the information above with \eqref{deltaxuanze}, we arrive at
\begin{equation*}
	\max\left\{\|{f_{k}}\|_{L^{\infty}\left(B_{1}\right)},\mathcal{K}_{k},\mathcal{M}_{k}\left(\abs{\xi_{k}}^{(m-p_{\rm min})_{+}}+1\right)\right\}\leq \sigma.
\end{equation*}	
At this moment, the assumptions in Lemma \ref{bijin} are satisfied. Thus, we can apply the conclusion from Step 1
 to obtain
\begin{equation*}
	\|u_{k}-\tilde{l}\|_{L^{\infty}(B_{\rho})}\leq \rho^{1+\alpha^{\prime}},
\end{equation*}
where $\tilde{l}(x)$ is an affine function of the form
$\tilde{l}(x)=\tilde{a}+\tilde{b}\cdot x$ with $|\tilde{a}|+|\tilde{b}|\leq C$. In the sequel, we define the  approximating affine function $l_{k+1}$ as
\begin{equation*}
	l_{k+1}(x):=a_{k+1}+b_{k+1}\cdot x,
\end{equation*}
where
\begin{equation*}
	a_{k+1}:=a_{k}+\rho^{k(1+\alpha^{\prime})}\tilde{a}\quad {\rm and}\quad     b_{k+1}:=b_{k}+\rho^{k\alpha^{\prime}}\tilde{b}.
\end{equation*}
Scaling back, we obtain
\begin{equation*}
	\|u-l_{k+1}\|_{L^{\infty}(B_{\rho^{k+1}})}=\rho^{k(1+\alpha^{\prime})}\|u_{k}-\tilde{l}\|_{L^{\infty}(B_{\rho})}
	\leq \rho^{(k+1)(1+\alpha^{\prime})}.
\end{equation*}
In addition, we have
\begin{equation*}
	|a_{k+1}-a_{k}|+\rho^{k}|b_{k+1}-b_{k}|=\rho^{k(1+\alpha^{\prime})}\left(\abs{\tilde{a}}+\abs{\tilde{b}}\right)\leq C\rho^{k(1+\alpha^{\prime})}.
\end{equation*}
Hence, we have proven that \eqref{711} and \eqref{712} hold for all $j\in \mathbb{N}$.\\
{\bf Step 3. Convergence analysis}. From Step 2, we know that $\{a_{j}\}_{j\in\mathbb{N}}\subset \mathbb{R}$, $\{b_{j}\}_{j\in\mathbb{N}}\subset \rn$ are Cauchy sequences, hence, they converge. Let us denote
\begin{equation*}
	\overline{a}:=\lim\limits_{j\rightarrow\infty}a_{j},
	\quad \overline{b}:=\lim\limits_{j\rightarrow\infty}b_{j}, \quad \overline{l}(x):=\overline{a}+\overline{b}\cdot x.
\end{equation*}	
For any $n\geq j$, a combination of the triangle inequality and \eqref{712} yields that
\begin{align*}
	|a_{j}-a_{n}|\leq \sum_{k=j}^{n-1}\abs{a_{k}-a_{k+1}}\leq C\sum_{k=j}^{n-1}\rho^{k(1+\alpha^{\prime})}\leq C\rho^{j(1+\alpha^{\prime})} \frac{1-\rho^{(n-j)(1+\alpha^{\prime})}}{1-\rho^{(1+\alpha^{\prime})}}.
\end{align*}	
Letting $n\rightarrow \infty$, we get	
\begin{align}\label{511}
	|a_{j}-\overline{a}|\leq  \frac{C\rho^{j(1+\alpha^{\prime})}}{1-\rho^{(1+\alpha^{\prime})}}.
\end{align}	
Similarly, we have
\begin{align}\label{512}
	|b_{j}-\overline{b}|\leq  \frac{C\rho^{j\alpha^{\prime}}}{1-\rho^{\alpha^{\prime}}}.
\end{align}	

Finally, given any $0<r\leq \rho$, we can find $j\in \mathbb{N}$ such that $\rho^{j+1}<r\leq \rho^{j}$.
Thus, combining \eqref{712} with \eqref{511} and \eqref{512}, we arrive at
\begin{align*}
	\|u-\overline{l}\|_{L^{\infty}(B_{r})}&\leq \|u-\overline{l}\|_{L^{\infty}(B_{\rho^{j}})}\\&\leq \|u-{l_{j}}\|_{L^{\infty}(B_{\rho^{j}})}+\|l_{j}-\overline{l}\|_{L^{\infty}(B_{\rho^{j}})}\\
	&\leq  \rho^{j(1+\alpha^{\prime})}+|a_{j}-\overline{a}|+\rho^{j}|b_{j}-\overline{b}|\\
	&\leq \rho^{j(1+\alpha^{\prime})}\left(1+\frac{2C}{1-\rho^{\alpha^{\prime}}}\right)\\
	&\leq \frac{1}{\rho^{1+\alpha^{\prime}}}\left(1+\frac{2C}{1-\rho^{\alpha^{\prime}}}\right)r^{1+\alpha^{\prime}}\\
	&\leq Cr^{1+\alpha^{\prime}}.
\end{align*}
This implies that $u$ is $C^{1,\alpha}$ at $0$ and complete the proof.
\end{proof}

In the end, with the aid of Propositions \ref{qiyizhuantuihua} and \ref{prop3.4}, we are in a position to complete the proof of Theorem \ref{thm1}.

\begin{proof}[Proof of Theorem \ref{thm1}] Initially, we exploit the scaling features of \eqref{55mainmodel} to reduce the problem to a smallness regime. That is, we may assume that
	\begin{equation}\label{62asmall}
		\|{u}\|_{L^{\infty}(B_{1})}\leq 1, \quad  \max\left\{\|{f}\|_{L^{\infty}(B_{1})},\mathcal{K},\mathcal{M}\right\}\leq \delta
	\end{equation} 
	for a constant $\delta>0$ coming from Proposition \ref{prop3.4}. In fact, we define rescaled 
	function $\tilde{u}:B_{1}\rightarrow \mathbb{R}$ by
	$$\tilde{u}(x)=\frac{u(\tau x)}{K},$$
	where $K\geq 1\geq \tau$ are constants to be determined later. Then we can readily check that
	$\tilde{u}$ solves
	\begin{equation*}
		-\tilde{\Phi}({D\tilde{u}}, x)\Delta_{p}^{N}\tilde{u}+\tilde{\mathcal{H}}(D\tilde{u},x) = \tilde{f}(x) \quad \text{in} \quad  B_{1}
	\end{equation*}
	in the viscosity sense, where
	\begin{align*}
		\tilde{\Phi}(t,x):=&\left(\frac{\tau}{K}\right)^{p(\tau x)}\Phi\left(\frac{K}{\tau}t,\tau x\right),\\
		\tilde{\Upsilon}(|t|,x):=&|t|^{p(\tau x)}+\tilde{a}(x)|t|^{q(\tau x)},\\
		\tilde{a}(x):=&\left(\frac{\tau}{K}\right)^{p(\tau x)-q(\tau x)}a(\tau x),\\
		\tilde{\mathcal{H}}(t,x):=&\frac{\tau^{2+p(\tau x)}}{K^{1+p(\tau x)}}\mathcal{H}\left(\frac{K}{\tau}t,\tau x\right),\\
		\tilde{f}(x):=&\frac{\tau^{2+p(\tau x)}}{K^{1+p(\tau x)}}f(\tau x).
	\end{align*}
	Note that 
	\begin{equation*}
		L_{1}	\tilde{\Upsilon}(|\xi|,x)\leq	\tilde{\Phi}(\xi,x) \leq L_{2}	\tilde{\Upsilon}(|\xi|,x)\quad {\rm for}\;\,(\xi,x)\in\mathbb{R}^{d}\times B_{1},
	\end{equation*}
	that is, $\tilde{\Phi}$ satisfies the same structural assumption as $\Phi$. Combining \eqref{15} with $\frac{K}{\tau}\geq 1$ yields  
	\begin{equation*}
		|\tilde{\mathcal{H}}(t,x)|\leq \frac{\tau^{2+p_{\rm min}}}{K^{1+p_{\rm min}}}\bigg(\mathcal{K}+\mathcal{M}\left(\frac{K}{\tau}\right)^{m}|t|^{m}\bigg)=:\tilde{\mathcal{K}}+\tilde{\mathcal{M}}|t|^{m}.
	\end{equation*}
	Now, we select
	\begin{equation*}
		K:=
		\begin{cases}
			1+\|u\|_{L^{\infty}\left(B_{1}\right)}+\left(\|f\|_{L^{\infty}\left(B_{1}\right)}+\mathcal{K}\right)^{\frac{1}{1+p_{\rm min}}}+\mathcal{M}^{\frac{1}{1+p_{\rm min}-m}} & \text{for}\;\; m<1+p_{\rm min}, \\
			1+\|u\|_{L^{\infty}\left(B_{1}\right)} &\text{for}\;\;  m=1+p_{\rm min},
		\end{cases}
	\end{equation*}
	and
	\begin{equation*}
	\tau:=
		\begin{cases}
			\min\left\{1,\delta^{\frac{1}{2+p_{\rm min}}},\delta^{\frac{1}{2+p_{\rm min}-m}}\right\} & \text{for} \;\; m<1+p_{\rm min}, \\
			\min\left\{1,\left(\frac{\delta}{\|f\|_{L^{\infty}(B_{1})}+\mathcal{K}}\right)^{\frac{1}{2+p_{\rm min}}},\frac{\delta}{\mathcal{M}}\right\} &\text{for}\;\;  m=1+p_{\rm min}.
		\end{cases}
	\end{equation*}
	With such choice, we immediately deduce
	\begin{equation*}
		\|\tilde{u}\|_{L^{\infty}\left(B_{1}\right)}\leq 1\quad {\rm and }\quad
		\max\left\{\|\tilde{f}\|_{L^{\infty}\left(B_{1}\right)},\tilde{\mathcal{K}},\tilde{\mathcal{M}}\right\}
		\leq \delta.
	\end{equation*}	
	Therefore, $\tilde{u}$ solves an equation possessing the same structure as \eqref{55mainmodel} with the smallness assumptions \eqref{62asmall}.
	
	For the case $p_{\rm min}\geq 0$, we employ Proposition \ref{prop3.4} to obtain $u$ is $C^{1,\alpha^{\prime}}$ at $0$. Then by standard translation argument and covering argument, we can conclude that $u\in C_{\rm loc}^{1,\alpha^{\prime}}(B_{1})$.
	
	On the other hand, for the case $-1<p_{\rm min}<0$, with the aid of Proposition \ref{qiyizhuantuihua}, we see that $u$ is a viscosity solution of the following equation
\begin{equation*}
	-\hat{\Phi}({Du}, x)\Delta_{p}^{N}u+\hat{\mathcal{H}}(Du,x)	
	=\hat{f}(x) \quad  \text{in} \quad B_{1},
\end{equation*}
where
\begin{align*}
	\hat{\Phi}({t},x):=&\abs{t}^{-p_{\rm min}}{\Phi}({t},x),\quad
	\hat{\mathcal{H}}(t,x):=\abs{t}^{-p_{\rm min}}{\mathcal{H}}(t,x),\quad
	\hat{f}(x):=\abs{Du}^{-p_{\rm min}}f(x).
\end{align*}
Note that, for all $(t,x)\in\rn\times B_{1}$, 
\begin{equation*}
	K_{1}\left(|t|^{p(x)-p_{\rm min}}+{a}(x)|t|^{q(x)-p_{\rm min}}\right)\leq\hat{\Phi}(t,x) \leq K_{2} \left(|t|^{p(x)-p_{\rm min}}+{a}(x)|t|^{q(x)-p_{\rm min}}\right)
\end{equation*}
and $0\leq p(x)-p_{\rm min}\leq q(x)-p_{\rm min}\leq q_{\rm max}-p_{\rm min}<\infty$. In view of \eqref{15} and $0<-p_{\rm min}<m-p_{\rm min}\leq 1$,  we have
$$|\hat{\mathcal{H}}(t,x)|\leq |t|^{-p_{\rm min}}\left(\mathcal{K}+\mathcal{M}|t|^{m}\right)\leq \hat{\mathcal{K}}+\hat{\mathcal{M}}|t|^{m-p_{\rm min}}\quad {\rm for\;all} \;(t,x)\in \rn\times B_{1},$$
where $\hat{\mathcal{K}}:=\mathcal{K}$ and $\hat{\mathcal{M}}:=\mathcal{K}+\mathcal{M}$. Moreover, we apply Proposition \ref{jin2} to obtain
\begin{equation}\label{lip1}
	[u]_{C^{0,1}(B_{3/4})}\leq C
\end{equation}
for a universal constant $C>0$. As a consequence, the gradient $Du$ is bounded almost everywhere. Then we can estimate
\begin{equation*}
	\|\hat{f}\|_{L^{\infty}\left(B_{3/4}\right)}\leq C^{-p_{\rm min}}\|{f}\|_{L^{\infty}\left(B_{3/4}\right)}. 
\end{equation*}
At this point, we reduce the singular case to the degenerate case, and $\hat{\Phi}$, $\hat{\mathcal{H}}$ satisfy the assumptions \eqref{12}-\eqref{15}. Therefore, for $-1<p_{\rm min}<0$, we can obtain the $C^{1,\alpha^{\prime}}$-regularity with $\alpha^{\prime}\in (0,\beta_{0})\cap \left(0,\frac{1}{1+q_{\rm max}-p_{\rm min}}\right]$ by repeating the previous arguments. The proof is complete.
\end{proof}
\section{Optimal Local H\"{o}lder regularity of the gradient}\label{section444}
\subsection{Gradient approximation}
 With the aid of the regularity result from Theorem \ref{thm1} to show that the solutions to \eqref{55mainmodel} can be approximated by $p$-harmonic function in a $C_{\rm loc}^{1}$ fashion.
\begin{lemma}\label{lem5.1}
	Let $u\in C(B_{1})$ be a normalized viscosity solution of \eqref{55mainmodel} under assumptions \eqref{12}-\eqref{15} with $p_{\rm min}\geq 0$. Then, for any $\eta>0$, there exists $\iota\in (0,1)$  
	such that if
	\begin{equation}\label{smallcondition}
		\max\left\{\|f\|_{L^{\infty}(B_{1})},\mathcal{K},\mathcal{M}\right\}\leq \iota,
	\end{equation}
	then there exists a function $v$ which is a weak solution of
	$$-\Delta_{p} v=0 \quad {\rm in}\;\; B_{3/4}$$
such that
	\begin{equation*}
		\max\left\{\|u-v\|_{L^{\infty}(B_{1/2})},\|Du-Dv\|_{L^{\infty}(B_{1/2})}\right\}\leq \eta.
	\end{equation*}
\end{lemma}
\begin{proof}
	We prove by contradiction. Suppose that there exist $\eta_{0}>0$ and sequences
	of functions $\{\Phi_{j}\}_{j\in \mathbb{N}}$, $\{\mathcal{H}_{j}\}_{j\in \mathbb{N}}$, $\{f_{j}\}_{j\in \mathbb{N}}$, $\{u_{j}\}_{j\in \mathbb{N}}$ such that $u_{j}\in C({B_{1}})$ is a viscosity solution of 
		\begin{equation*}
			-\Phi_{j}({Du_{j}}, x)\Delta_{p}^{N} u_{j}+\mathcal{H}_{j}(Du_{j},x) =f_{j}(x) \quad  \text{in} \quad B_{1}
		\end{equation*}
		with $ \|u_{j}\|_{L^{\infty}(B_{1})}\leq 1$, where $f_{j}\in C({B_{1}})$, $\Phi_{j}$ satisfies \eqref{12} and \eqref{13} with $p_{j}(\cdot),q_{j}(\cdot)\in C(B_{1})$, $0\leq p_{\rm min}\leq p_{j}(x)\leq q_{j}(x)\leq q_{\rm max}<\infty$ and $0\leq a_{j}(\cdot)\in C(B_{1})$, $\mathcal{H}_{j}:\mathbb{R}^{d} \times B_{1}\rightarrow \mathbb{R}$ is continuous and there exist constants $\mathcal{K}_{j},\mathcal{M}_{j}>0$ such
		that
		\begin{equation*}
			|\mathcal{H}_{j}(t,x)|\leq \mathcal{K}_{j}+\mathcal{M}_{j}|t|^{m} \quad  {\rm for\; every\;}(t,x) \in\mathbb{R}^{d}\times B_{1},
		\end{equation*}
	 and
		\begin{equation*}
			\max\left\{\|f_{j}\|_{L^{\infty}(B_{1})}, \mathcal{K}_{j},\mathcal{M}_{j}\right\}\leq \frac{1}{j}.
		\end{equation*}	
Nonetheless, for any function $v$ solves 
$$-\Delta_{p} v=0 \quad {\rm in}\;\; B_{3/4},$$
it holds
\begin{equation}\label{contraction111}
	\max\left\{\|u_{j}-v\|_{L^{\infty}(B_{1/2})},\|Du_{j}-Dv\|_{L^{\infty}(B_{1/2})}\right\}> \eta_{0} \quad {\rm for\;any}\;j\in\mathbb{N}.
\end{equation}

We know from Theorem \ref{thm1} that the sequence $\{u_{j}\}_{j\in\mathbb{N}}\subset C_{\rm loc}^{1,\alpha^{\prime}}(B_{1})$ for some $\alpha^{\prime}\in (0,1)$. By applying Arzel${\rm \grave{a}}$-Ascoli theorem, up to a subsequence, we know that $u_{j}$ converges locally uniformly in $B_{1}$ to some continuous function $u_{\infty}$ in the $C^{1}$-topology. Next, by arguing as in Lemma \ref{bijin}, we can 
conclude that $u_{\infty}$ is a viscosity solution of 
	$$-\Delta_{p}^{N} u_{\infty}=0 \quad {\rm in}\;\; B_{3/4}.$$
	By Lemma \ref{dengjia}, $u_{\infty}$ is a weak solution of 
	$$-\Delta_{p} u_{\infty}=0 \quad {\rm in}\;\; B_{3/4}.$$
	Finally, taking $v=u_{\infty}$, we reach a contradiction with \eqref{contraction111} for $j$ sufficiently large. This completes the proof of the desired result.                                     	
\end{proof}
\subsection{Improved Oscillation-Type Estimate}
To begin with, by using the approximation with $p$-harmonic function, we obtain an oscillation estimate for solutions $u$ to \eqref{55mainmodel} near the critical set $\left\{x:Du(x)=0\right\}$.
\begin{lemma}\label{lem5.2}
	Under the assumptions of Lemma \ref{lem5.1}, for every $0<\alpha<{\alpha}_{0}$, there exists a universal constant $0<\varrho<\frac{1}{2}$ such that 
	\begin{equation}\label{911}
		\sup_{B_{\varrho}}\Abs{u(x)-u(0)}\leq \rho^{1+\alpha}+|Du(0)|\varrho.
	\end{equation}
\end{lemma}
\begin{proof}
Let $\eta>0$ to be ﬁxed a posteriori. From Lemma \ref{lem5.1}, we know that there exists $\iota>0$ such that whenever
\begin{equation*}
	\max\left\{\|f\|_{L^{\infty}(B_{1})},\mathcal{K},\mathcal{M}\right\}\leq \iota,
\end{equation*}
then there exist a $p$-harmonic function $v$ satisfying
\begin{equation}\label{211}
	\max\left\{\|u-v\|_{L^{\infty}(B_{1/2})},\|Du-Dv\|_{L^{\infty}(B_{1/2})}\right\}\leq \eta.
\end{equation}

By virtue of Remark \ref{ptiaohezhengzexing}, we know that $v\in C_{\rm loc}^{1,{\alpha}_{0}}(B_{3/4})$ with the estimate
\begin{equation*}
	\sup_{B_{\varrho}}\Abs{v(x)-v(0)-Dv(0)\cdot x}\leq C\varrho^{1+{\alpha}_{0}} \quad {\rm for\; all}\;\varrho\in\left(0,\frac{1}{2}\right).
\end{equation*}
This together with \eqref{211} immediately yields that
\begin{equation*}
	\begin{split}
		\sup_{B_{\varrho}}\Abs{u(x)-u(0)-Du(0)\cdot x}\leq& \sup_{B_{\varrho}}\Abs{u(x)-v(x)}+\sup_{B_{\varrho}}\Abs{v(x)-v(0)-Dv(0)\cdot x}\\
		&+\sup_{B_{\varrho}}\Abs{u(0)-v(0)+(Dv(0)-Du(0))\cdot x}\\
		\leq& 3\eta+C\varrho^{1+{\alpha}_{0}}\leq \varrho^{1+\alpha}
	\end{split}
\end{equation*}
as long as we make the following universal choices
\begin{equation*}
	\varrho\in \left(0,\min\left\{\frac{1}{2},\left(\frac{1}{2C}\right)^{{\alpha}_{0}-\alpha}\right\}\right)\quad {\rm and}\quad \eta\in \left(0,\frac{1}{6}\varrho^{1+\alpha}\right).
\end{equation*}
Therefore, we obtain \eqref{911}, thereby finishing the proof.                                                       
\end{proof}
Next we iterate the previous estimate to control the oscillation of the solutions in dyadic balls. 
\begin{lemma}\label{lem5.3}
	Under the assumptions of Lemma \ref{lem5.2}, given $\alpha\in (0,{\alpha}_{0})\cap \left(0,\frac{1}{1+q_{\rm max}}\right]$, there holds 
	\begin{equation*}
		\sup_{B_{\varrho^{k}}}\Abs{u(x)-u(0)}\leq \varrho^{k(1+\alpha)}+|Du(0)|\sum_{i=0}^{k-1}\varrho^{k+i\alpha}
	\end{equation*}
for all $k\in\mathbb{N}$, where constant $\varrho$ comes from Lemma \ref{lem5.2}.
\end{lemma}
\begin{proof}
	The proof follows from an induction argument. Clearly, the claim immediately holds for $k=1$ by Lemma \ref{lem5.2}. Suppose that the conclusion holds true for $n=1,2,...,k$. Our goal is to show that the claim also holds for $n=k+1$. To this end, we introduce an auxiliary function $u_{k}:B_{1}\rightarrow \mathbb{R}$ as
	\begin{equation*}
		u_{k}(x):=\frac{u\left(\varrho^{k}x\right)-u(0)}{\mathcal{A}_{k}}
	\end{equation*}
with $\mathcal{A}_{k}:=\varrho^{k(1+\alpha)}+|Du(0)|\sum_{i=0}^{k-1}\varrho^{k+i\alpha}$. We can readily check that $u_{k}$ solves
	\begin{equation} \label{Eq2.1}
		-\Phi_{k}({D{u_{k}}}, x)\Delta_{p}^{N}u_{k}+\mathcal{H}_{k}(D{u_{k}}, x)	= {f_{k}}(x) \quad \text{in} \quad  B_{1}
	\end{equation}
	in the viscosity sense, where
	\begin{align*}
		{\Phi_{k}}(t,x):=&\left(\frac{\varrho^{k}}{\mathcal{A}_{k}}\right)^{p_{k}(x)}\Phi\left(\frac{\mathcal{A}_{k}}{\varrho^{k}}t,\varrho^{k}x\right) \;{\rm satisfies}\;\eqref{12}\;{\rm and} \; \eqref{13}\;{\rm with}\\
		p_{k}(x):=&p(\varrho^{k}x),\quad q_{k}(x):=q(\varrho^{k}x), \quad
		{a_{k}}(x):=\left(\frac{\varrho^{k}}{\mathcal{A}_{k}}\right)^{p_{k}(x)-q_{k}(x)}a(\varrho^{k} x),\\
		\mathcal{H}_{k}(t,x):=&\frac{\varrho^{k(2+p_{k}(x))}}{\mathcal{A}_{k}^{1+p_{k}(x)}}\mathcal{H}\left(\frac{\mathcal{A}_{k}}{\varrho^{k}}t,\varrho^{k}x\right),\\
		{f_{k}}(x):=&\frac{\varrho^{k(2+p_{k}(x))}}{\mathcal{A}_{k}^{1+p_{k}(x)}}f(\varrho^{k}x).
	\end{align*}
Note that $0\leq p_{\rm min}\leq p_{k}(x)\leq q_{k}(x)\leq q_{\rm max}<\infty$. By induction assumption, we have
	$$\|{u_{k}}\|_{L^{\infty}\left(B_{1}\right)}\leq 1.$$
	Combining \eqref{15} with the definition of $\mathcal{A}_{k}$ and $\varrho\in(0,\frac{1}{2})$, we get
	\begin{equation*}
		\begin{split}
			|\mathcal{H}_{k}(t,x)|&\leq \varrho^{k}\left(\frac{\varrho^{k}}{\mathcal{A}_{k}}\right)^{1+p_{k}(x)}\left(\mathcal{K}+\mathcal{M}\left(\frac{\mathcal{A}_{k}}{\varrho^{k}}\right)^{m}|t|^{m}\right)\\
			&\leq \varrho^{k(1-\alpha(1+q_{\rm max}))}\mathcal{K}+\mathcal{M}\varrho^{k(1-\alpha(1+q_{\rm max}-m))}|t|^{m} =:\mathcal{K}_{k}+\mathcal{M}_{k}|t|^{m}.
		\end{split}
	\end{equation*} 
	It follows from \eqref{smallcondition}, $m>0$ and $\alpha\in \left(0,\frac{1}{1+q_{\rm max}}\right]$ that
	\begin{equation*}
		\|{f_{k}}\|_{L^{\infty}\left(B_{1}\right)}\leq \varrho^{k(1-\alpha(1+q_{\rm max}))}\|{f}\|_{L^{\infty}\left(B_{1}\right)} \leq \iota,
	\end{equation*}
	\begin{equation*}
		\mathcal{K}_{k}=\rho^{k\left(1-\alpha(1+q_{\rm max})\right)}\mathcal{K}\leq \iota,
	\end{equation*}
\begin{equation*}
	\mathcal{M}_{k}=\rho^{k\left(1-\alpha(1+q_{\rm max}-m)\right)}\mathcal{M}\leq \iota.
\end{equation*}
	At this moment, the assumptions in Lemma \ref{lem5.2} are satisfied. Thus, we can apply Lemma \ref{lem5.2} to $u_{k}$ and obtain
	\begin{equation*}
		\sup_{B_{\varrho}}\Abs{u_{k}(x)-u_{k}(0)}\leq \varrho^{1+\alpha}+|Du_{k}(0)|\varrho.
	\end{equation*}
Scaling back, we obtain
	\begin{equation*}
		\begin{split}
			\sup_{B_{\varrho^{k+1}}}\Abs{u(x)-u(0)}&\leq \varrho^{1+\alpha}\mathcal{A}_{k}+|Du(0)|\varrho^{k+1}\\
			&= \varrho^{(k+1)(1+\alpha)}+|Du(0)|\sum_{i=0}^{k}\varrho^{k+1+i\alpha}.
		\end{split}
\end{equation*}
	This completes the proof of the desired result.
\end{proof}
\begin{lemma}\label{lem5.4} 
	Suppose that the assumptions of Lemma \ref{lem5.3} are in force. Then, there exists a universal constant $C_{0}>1$ such that
\[
\sup_{B_r} \frac{|u(x) - u(0)|}{r^{1+\alpha}} \leq C_0 \bigg(1 + |Du(0)|\, r^{-\alpha}\bigg), \quad \forall \,r \in (0, \varrho],
\]
where constant $\varrho$ comes from Lemma \ref{lem5.3}.
\end{lemma}
\begin{proof}
Fix any $r \in (0, \varrho]$ and choose the smallest integer $k\in\mathbb{N}$ such that $\varrho^{k+1}<r\leq \varrho^{k}$. By Lemma \ref{lem5.3}, we have
	\begin{equation*}
		\begin{split}
			\sup_{B_r} \frac{|u(x) - u(0)|}{r^{1+\alpha}}&\leq \frac{1}{\varrho^{1+\alpha}}\sup_{B_{\varrho^{k}}}\frac{\Abs{u(x)-u(0)}}{\varrho^{k(1+\alpha)}}\\
			&\leq \frac{1}{\varrho^{1+\alpha}}\left(1+|Du(0)|\varrho^{-k(1+\alpha)}\sum_{i=0}^{k-1}\varrho^{k+i\alpha}\right)\\
			&\leq \frac{1}{\varrho^{1+\alpha}}\left(1+|Du(0)|\varrho^{-k\alpha}\frac{1}{1-\varrho^{\alpha}}\right)\\
			&\leq C_{0}\left(1+|Du(0)|r^{-\alpha}\right),
		\end{split}
	\end{equation*}
where $C_{0}:=\frac{1}{\varrho^{1+\alpha}(1-\varrho^{\alpha})}$. This completes the proof.
\end{proof}
\subsection{Proof of Theorem \ref{thm2}}
Finally, with the help of Lemma \ref{lem5.4}, we now complete the proof of Theorem \ref{thm2}.
\begin{proof}[Proof of Theorem \ref{thm2}] Initially, as argued in the proof of Theorem \ref{thm1}, we can assume 
	\begin{equation}\label{2asmall}
		\|{u}\|_{L^{\infty}(B_{1})}\leq 1, \quad  \max\left\{\|{f}\|_{L^{\infty}(B_{1})},\mathcal{K},\mathcal{M}\right\}\leq \iota
	\end{equation} 
	for a constant $\iota>0$ coming from Lemma \ref{lem5.2}. In addition, without loss of generality, we may assume that $x_0=0$. 

For the degenerate case $p_{\rm min}\geq 0$, fix any $r\in (0,\varrho]$, we then analyze all the possible cases.\\
	{\bf Case 1.} If $|Du(0)|\leq r^{\alpha}$, by using Lemma \ref{lem5.4}, we obtain
	\begin{equation*}
		\begin{split}
			\sup_{B_r} |u(x) - u(0)-Du(0)\cdot x|&\leq \sup_{B_r} |u(x) - u(0)|+|Du(0)|r\\
			&\leq C_{0}r^{1+\alpha}\left(1+|Du(0)|r^{-\alpha}\right)+r^{1+\alpha}\\
			&\leq 3C_{0}r^{1+\alpha}.
		\end{split}
	\end{equation*}
	{\bf Case 2.} If $r^{\alpha}<|Du(0)|\leq \varrho^{\alpha}$, we denote $r_{0}:=|Du(0)|^{1/\alpha}$ and define
	\begin{equation*}
		u_{r_{0}}(x):=\frac{u(r_{0}x)-u(0)}{r_{0}^{1+\alpha}}\quad {\rm for \;} x\in B_{1}.
	\end{equation*}
	Note that $r_{0}\leq \varrho$. By using Lemma \ref{lem5.4} again, we immediately obtain
	\begin{equation}\label{9565}
		\begin{split}
			\sup_{B_{1}}|u_{r_{0}}(x)|=\sup_{B_{r_{0}}} \frac{|u(x) - u(0)|}{r_{0}^{1+\alpha}}\leq C_{0}\left(1+|Du(0)|r_{0}^{-\alpha}\right)=2C_{0}.
		\end{split}
	\end{equation}
It is easy to check that $u_{r_{0}}$ is a viscosity solution of 
	\begin{equation*} 
	-\Phi_{r_{0}}({D{u_{r_{0}}}}, x)\Delta_{p}^{N}u_{r_{0}}+\mathcal{H}_{{r_{0}}}(D{u_{r_{0}}}, x)	= {f_{r_{0}}}(x) \quad \text{in} \quad  B_{1}
\end{equation*}
with $u_{r_{0}}(0)=0$, $|Du_{r_{0}}(0)|=\frac{|Du(0)|}{r_{0}^{\alpha}}=1$, and $\|{u}_{r_{0}}\|_{L^{\infty}(B_{1})}\leq 2C_{0}$, where
\begin{align*}
	{\Phi_{r_{0}}}(t,x):=&r_{0}^{-\alpha p_{r_{0}}(x)}\Phi\left(r_{0}^{\alpha}t,r_{0}x\right)\;{\rm satisfies}\;\eqref{12}\;{\rm and} \; \eqref{13}\;{\rm with}\\
	p_{r_{0}}(x):=&p(r_{0}x),\quad q_{r_{0}}(x):=q(r_{0}x), \quad
	{a_{r_{0}}}(x):=r_{0}^{-\alpha(p_{r_{0}}(x)-q_{r_{0}}(x))}a(r_{0} x),\\
	\mathcal{H}_{r_{0}}(t,x):=&r_{0}^{1-\alpha(1+p(r_{0} x))}
	\mathcal{H}\left(r_{0}^{\alpha}t,r_{0}x\right),\\
	{f_{r_{0}}}(x):=&r_{0}^{1-\alpha(1+p(r_{0} x))}f(r_{0} x).
\end{align*}
Clearly, $0\leq p_{\rm min}\leq p_{r_{0}}(x)\leq q_{r_{0}}(x)\leq q_{\rm max}<\infty$. It follows from \eqref{15} that
	\begin{equation*}
	\begin{split}
		|\mathcal{H}_{r_{0}}(t,x)|\leq r_{0}^{1-\alpha(1+q_{\rm max})}\left(\mathcal{K}+\mathcal{M}r_{0}^{\alpha m}|t|^{m}\right) =:\mathcal{K}_{k}+\mathcal{M}_{k}|t|^{m}.
	\end{split}
\end{equation*}
Since $r_{0}<1$ and $\alpha\in \left(0,\frac{1}{1+q_{\rm max}}\right]$, we arrive at
	\begin{equation*}
	\|{f_{r_{0}}}\|_{L^{\infty}\left(B_{1}\right)}\leq r_{0}^{1-\alpha(1+q_{\rm max})}\|{f}\|_{L^{\infty}\left(B_{1}\right)} \leq \iota,
\end{equation*}
\begin{equation*}
	\mathcal{K}_{r_{0}}=r_{0}^{1-\alpha(1+q_{\rm max})}\mathcal{K}\leq \iota.
\end{equation*}
\begin{equation*}
	\mathcal{M}_{r_{0}}=r_{0}^{1-\alpha(1+q_{\rm max}-m)}\mathcal{M}\leq \iota.
\end{equation*}
Invoking Theorem \ref{thm1}, we obtain $u_{r_{0}}\in C_{\rm loc}^{1,\alpha^{\prime}}(B_{1})$ and there exists a universal constant $C> 0$ such that for all $x\in B_{1/2}$,
$$\Abs{Du_{r_{0}}(x)-Du_{r_{0}}(0)}\leq C|x|^{\alpha^{\prime}},$$
this along with $|Du_{r_{0}}(0)|=1$ immediately yields
$$ 1-C|x|^{\alpha^{\prime}}\leq \Abs{Du_{r_{0}}(x)}\leq 1+C|x|^{\alpha^{\prime}}.$$
Then we may take a small universal radius $\varrho_{0}>0$, independent of $r_{0}$, such that
$$c_{0}\leq |Du_{r_{0}}(x)| \leq c_{0}^{-1} \quad {\rm in}\;\, B_{\varrho_{0}}$$
with a fixed constant $c_{0}\in (0,1)$. 
Hence, $u_{r_{0}}$ satisfies 
\begin{equation*} 
	-\Delta_{p}^{N}u_{r_{0}}=\tilde{f}_{r_{0}}(x):=\frac{{f_{r_{0}}}(x) -\mathcal{H}_{{r_{0}}}(D{u_{r_{0}}}, x)}{\Phi_{r_{0}}({D{u_{r_{0}}}}, x)} \quad \text{in} \quad  B_{\varrho_{0}}
\end{equation*}
in the viscosity sense. The previous statement shows that $\tilde{f}_{r_{0}}$ is universally bounded in $B_{\varrho_{0}}$ and so we applying the classical result from Lemma \ref{111tiaohezhengzexing} to obtain $u_{r_{0}}\in C_{\rm loc}^{1,\alpha_{0}^{-}}(B_{\varrho_{0}})$ with the estimate
\begin{equation*}
	\sup\limits_{x\in B_{\varrho_{1}}}\Abs{u_{r_{0}}(x)-Du_{r_{0}}(0)\cdot x}\leq C\varrho_{1}^{1+\alpha}
\end{equation*}
for every $0<\varrho_{1}\leq \frac{\varrho_{0}}{2}$ and $0<\alpha<\alpha_{0}$. Scaling back, it yields that
\begin{equation}\label{883zuoquyujielunchengli}
	\sup\limits_{x\in B_{t}}\Abs{u(x)-u(0)-Du(0)\cdot x}\leq Ct^{1+\alpha}
\end{equation}
for every $0<t\leq \frac{\varrho_{0} r_{0}}{2}$. It remains to show that the claim \eqref{883zuoquyujielunchengli} also holds on interval $\left(\frac{\varrho_{0} r_{0}}{2},r_{0}\right)$. When $t\in\left(\frac{\varrho_{0} r_{0}}{2},r_{0}\right)$, we apply \eqref{9565} and $|Du(0)|=r_{0}^{\alpha}$ to arrive at
\begin{equation*}
	\begin{split}
		\sup\limits_{x\in B_{t}}\Abs{u(x)-u(0)-Du(0)\cdot x}&\leq \sup\limits_{x\in B_{r_{0}}}\Abs{u(x)-u(0)-Du(0)\cdot x}\\
		&\leq \sup\limits_{x\in B_{r_{0}}}\Abs{u(x)-u(0)}+|Du(0)|r_{0}\\
		& \leq (2C_{0}+1)r_{0}^{1+\alpha}\\
		&=(2C_{0}+1) \left(\frac{2}{\varrho}\right)^{1+\alpha}\left(\frac{\varrho r_{0}}{2}\right)^{1+\alpha}\\
		&\leq (2C_{0}+1) \left(\frac{2}{\varrho}\right)^{1+\alpha}t^{1+\alpha}.
	\end{split}
\end{equation*}
	{\bf Case 3.} If $|Du(0)|>\varrho^{\alpha}$, we consider an auxiliary function
	$$\hat{u}(x):=\frac{\varrho^{\alpha}}{|Du(0)|}u(x)$$
	and then $|D\hat{u}(0)|=\varrho^{\alpha}$, which is back to Case 2. 
	
	On the other hand, for the singular case $-1<p_{\rm min}<0$, as argued in the proof of Theorem \ref{thm1}, we can reduce the singular case to the degenerate case and repeat the previous arguments to obtain the desired results. The proof is complete.
\end{proof}

\section{Higher H\"{o}lder regularity of gradient}\label{section555}
This section is dedicated to the proof of Theorem \ref{main} concerning an improved H\"{o}lder regularity of gradient at the origin. We start off by making several comments on the scaling features of the model \eqref{55mainmodel}, which will be used frequently in the following.
\begin{remark}[\bf Scaling feature]\label{scaling}
Let $u$ be a viscosity solution to \eqref{55mainmodel}. For constants $K\geq 1\geq \tau>0$ arbitrary, define $\tilde{u}:B_{1}\rightarrow \mathbb{R}$ as
$$\tilde{u}(x):=\frac{u(\tau x)}{K}.$$
We can readily check that
$\tilde{u}$ is a viscosity solution of
\begin{equation} \label{3shensuofangcheng}
	-\tilde{\Phi}(D\tilde{u}, x)	\Delta_{p}^{N}\tilde{u}+	\tilde{h}(x)\Abs{D\tilde{u}}^{m}= 	\tilde{f}(x) \quad \text{in} \quad  B_{1},
\end{equation}
where
\begin{align*}
	\tilde{\Phi}(\xi,x):=&\left(\frac{\tau}{K}\right)^{p(\tau x)}\Phi\left(\frac{K}{\tau}\xi,\tau x\right),\\
	\tilde{\Upsilon}(|\xi|,x):=&|\xi|^{p(\tau x)}+\tilde{a}(x)|\xi|^{q(\tau x)},\\
	\tilde{a}(x):=&\left(\frac{\tau}{K}\right)^{p(\tau x)-q(\tau x)}a(\tau x),\\
	\tilde{h}(x):=&\frac{\tau^{2+p(\tau x)-m}}{K^{1+p(\tau x)-m}}h(\tau x),\\
	\tilde{f}(x):=&\frac{\tau^{2+p(\tau x)}}{K^{1+p(\tau x)}}f(\tau x).
\end{align*}
Note that 
\begin{equation*}
	K_{1}	\tilde{\Upsilon}(|\xi|,x)\leq	\tilde{\Phi}(\xi,x) \leq K_{2}	\tilde{\Upsilon}(|\xi|,x)\quad {\rm for}\;\,(\xi,x)\in\mathbb{R}^{d}\times B_{1},
\end{equation*}
that is, $\tilde{\Phi}$ satisfies the same structural assumption as $\Phi$. It follows from \eqref{23}, $K\geq 1\geq \tau>0$, and $m\leq 1+p_{\rm min}$ that
\begin{equation*}
	\abs{\tilde{f}(x)}\leq \frac{\tau^{2+p_{\rm min}}}{K^{1+p_{\rm min}}}\abs{f(\tau x)}\leq {K}_{3}\tau^{2+p_{\rm min}+\theta_{1}}|x|^{\theta_{1}}\leq {K}_{3}|x|^{\theta_{1}},
\end{equation*}
\begin{equation*}
	\abs{\tilde{h}(x)}\leq \frac{\tau^{2+p_{\rm min}-m}}{K^{1+p_{\rm min}-m}}\abs{h(\tau x)}\leq {K}_{4}\tau^{2+p_{\rm min}-m+\theta_{2}}|x|^{\theta_{2}}\leq {K}_{4}|x|^{\theta_{2}},
\end{equation*}
which means that $\tilde{f}$ and $\tilde{h}$ satisfy the same structural assumptions as $f$ and $h$. Hence,
$\tilde{u}$ satisfies equation \eqref{3shensuofangcheng}, which is subject to the same structural assumptions as the equation for $u$.
\end{remark}

In particular, up to a normalization, i.e., by choosing $K:=1+\|u\|_{L^{\infty}(B_{1})}$, we can assume, with no loss of generality, that $\|u\|_{L^{\infty}(B_{1})}\leq 1$ is a normalized solution of \eqref{55mainmodel}. In addition, since $u(x)-u(0)$ satisfies the same equation as $u(x)$, by translation, we can assume $u(0)=0$.

In addition, to simplify our presentation let us introduce the following definition.
\begin{definition}
	Given a function $v\in C^{1}(B_{1})$, we denote the set of zero critical points as
	$$\mathscr{C}(v):=\{x\in B_{1}:v(x)=|Dv(x)|=0\}.$$
\end{definition}
\subsection{Gradient approximations}
 We begin by proving a $p$-harmonic approximation result, which essentially states that, under suitable smallness assumptions, a normalized viscosity solution $u$
of \eqref{55mainmodel} in $B_1$ can be approximated by a $p$-harmonic function $v$ with $0\in \mathscr{C}(v)$, provided that $Du(0)$ is near zero.
\begin{lemma}\label{lem3.1}
	Assume \eqref{12}-\eqref{14} and \eqref{21}-\eqref{23} hold with $p_{\rm min}\geq 0$. Let $u\in C(B_{1})$ be a normalized viscosity solution of \eqref{55mainmodel} with $u(0)=0$. Given $\varepsilon>0$, there exists constant $\delta>0$ such that if
	\begin{equation*}
		\max\left\{|Du(0)|,\|f\|_{L^{\infty}(B_{1})},\|h\|_{L^{\infty}(B_{1})}\right\}\leq \delta,
	\end{equation*}
	then there exists a $p$-harmonic function $v$ with $v\in C_{\rm loc}^{1,\alpha_{0}}(B_{1})$ such that $0\in \mathscr{C}(v)$ and
	\begin{equation*}
		\|u-v\|_{L^{\infty}(B_{1/2})}\leq \varepsilon.
	\end{equation*}
\end{lemma}
\begin{proof}
	Argue by contradiction. If the claim fails, then there exist $\varepsilon_{0}>0$ and sequences
	of functions $\{\Phi_{j}\}_{j\in \mathbb{N}}$, $\{f_{j}\}_{j\in \mathbb{N}}$, $\{u_{j}\}_{j\in \mathbb{N}}$, $\{h_{j}\}_{j\in \mathbb{N}}$ such that
	\begin{equation}\label{model333}
		-\Phi_{j}(Du_{j}, x)\Delta_{p}^{N}u_{j}+h_{j}(x)\abs{Du_{j}}^{m} =f_{j}(x) \quad  \text{in} \quad B_{1}
	\end{equation}	
	with $ \|u_{j}\|_{L^{\infty}(B_{1})}\leq 1$ and $u_{j}(0)=0$, as well as
	\begin{equation}\label{3aajintiaojiana}
		\max\left\{|Du_{j}(0)|,\|f_{j}\|_{L^{\infty}(B_{1})},\|h_{j}\|_{L^{\infty}(B_{1})}\right\}\leq \frac{1}{j},
	\end{equation}
	where $\Phi_{j}$ satisfies \eqref{12} and \eqref{13} with
	$$0\leq p_{\rm min}\leq p_{j}(x)\leq q_{j}(x)\leq q_{\rm max}<\infty,$$
	$$p_{j}(\cdot),q_{j}(\cdot)\in C(B_{1}) \quad {\rm and}\quad 0\leq a_{j}(\cdot)\in C(B_{1}).$$
	However, it holds
	\begin{equation}\label{3bijinmaodun}
		\|u_{j}-v\|_{L^{\infty}(B_{1/2})}>\varepsilon_{0}\quad {\rm for\; any \;} j\in \mathbb{N},
	\end{equation}
	for all $p$-harmonic function $v$ with $0\in \mathscr{C}(v)$.
	
Initially, by Theorem \ref{thm1}, $\{u_{j}\}_{j\in\mathbb{N}}\subset C_{\rm loc}^{1,\alpha^{\prime}}(B_{1})$ with $\alpha^{\prime}\in(0,1)$, so by the Arzel${\rm \grave{a}}$-Ascoli theorem, there exists a subsequence, still denoted by $\{u_{j}\}_{j\in\mathbb{N}}$, which converges locally uniformly to some continuous function $u_{\infty}$ in $B_{1}$ in the $C^{1}$-topology.
	Next, by arguing as in Lemma \ref{bijin}, we can 
	conclude that $u_{\infty}$ is a viscosity solution of 
	$$-\Delta_{p}^{N} u_{\infty}=0 \quad {\rm in}\;\; B_{3/4},$$ 
	with $u_{\infty}(0)=|D u_{\infty}(0)|=0$. By Lemma \ref{dengjia}, $u_{\infty}$ is a weak solution of 
	$$-\Delta_{p} u_{\infty}=0 \quad {\rm in}\;\; B_{3/4}.$$
	 Finally, taking $v=u_{\infty}$, we reach a contradiction with \eqref{3bijinmaodun} for $j$ sufficiently large. This completes the proof of the desired result.
\end{proof}
\subsection{ Regularity for small gradient} 
\begin{lemma}\label{lem3.2}
	Suppose that the hypotheses of Lemma \ref{lem3.1} are in force. Given $\alpha\in(0,\alpha_{0})$, there exist a universal constant $0<\rho<\frac{1}{2}$ 
	such that
	\begin{equation}\label{361}
		\sup\limits_{x\in B_{\rho}}|u(x)|\leq \rho^{1+\alpha}.
	\end{equation}
\end{lemma}
\begin{proof}
	Let $\varepsilon>0$ be a number to be chosen later. According to Lemma \ref{lem3.1}, there exist $\delta>0$ and
	a $p$-harmonic function $v$ such that $0\in \mathscr{C}(v)$ and
	\begin{equation}\label{33bijintiaohe}
		\sup\limits_{x\in B_{1/2}}|u(x)-v(x)|\leq \varepsilon.
	\end{equation}
	Applying the $C_{\rm loc}^{1,\alpha_{0}}$ regularity for $v$ (see Remark \ref{ptiaohezhengzexing}), together with $v(0)=|D v(0)|=0$, we have
	\begin{equation*}
		\sup\limits_{x\in B_{\rho}}|v(x)|\leq C\rho^{1+\alpha_{0}}\quad {\rm for\;any}\;\rho\in\left(0,\frac{1}{2}\right)
	\end{equation*}
	with a universal constant $C>0$.
	This along with \eqref{33bijintiaohe} yields that
	\begin{equation}\label{33zhengzexingdiyibu}
		\sup\limits_{x\in B_{\rho}}\Abs{u(x)}\leq \sup\limits_{x\in B_{\rho}}\Abs{u(x)-v(x)}+\sup\limits_{x\in B_{\rho}}\Abs{v(x)}\leq \varepsilon+C\rho^{1+\alpha_{0}}.
	\end{equation}	
	Now, fixed an exponent $\alpha<\alpha_{0}$, we make the following universal choices:
	\begin{equation}\label{choice}
		\rho\in \left(0,\min\left\{\frac{1}{2},\left(\frac{1}{2C}\right)^{{\alpha}_{0}-\alpha}\right\}\right)\quad {\rm and}\quad  \varepsilon:=\frac{1}{2}\rho^{1+\alpha}.
	\end{equation}
Once we fix the value of $\varepsilon$ here, the quantity $\delta$ in Lemma \ref{lem3.1} is determined accordingly.
	A combination of \eqref{33zhengzexingdiyibu} and \eqref{choice} immediately yields the desired estimate \eqref{361}. This completes the proof.	
\end{proof}
\begin{lemma}\label{lem3.3}
	Assume \eqref{12}-\eqref{14} and \eqref{21}-\eqref{23} hold with $p_{\rm min}\geq 0$.
	Let  $u\in C(B_{1})$ be a normalized viscosity solution of \eqref{55mainmodel} with $u(0)=0$ and $\alpha$ be as in \eqref{3aerfadefanwei}.
	There exist constants $\delta>0$ and $0<\rho<\frac{1}{2}$ both of which are the same as
	those in Lemma \ref{lem3.2}, such that if
	\begin{equation*}
		|D u(0)|\leq \delta t^{\alpha}\quad {\rm for\;all}\;t\in (0,\rho],
	\end{equation*}
	then it holds that
	\begin{equation*}
		\sup\limits_{x\in B_{t}}\Abs{u(x)}\leq Ct^{1+\alpha}
	\end{equation*}
	for a universal constant $C>0$.
\end{lemma}
\begin{proof}
	Initially we note that, by making use of the scaling feature presented in Remark \ref{scaling}, we may suppose that
	\begin{equation}\label{smalljiashe}
		\max\left\{\|f\|_{L^{\infty}(B_{1})},\|h\|_{L^{\infty}(B_{1})}\right\}\leq \delta
	\end{equation}
	for constant $\delta>0$ coming from Lemma \ref{lem3.2}. In fact, we define rescaled function $\tilde{u}:B_{1}\rightarrow \mathbb{R}$ by $\tilde{u}(x)=\frac{u(\tau x)}{K}$
	with
	\begin{equation*}
		K:=1+\|u\|_{L^{\infty}\left(B_{1}\right)}
	\end{equation*}
	and
	\begin{equation*}
		\tau:=
		\min\left\{1,\left(\frac{\delta}{\|f\|_{L^{\infty}(B_{1})}}\right)^{\frac{1}{2+p_{\rm min}}},\left(\frac{\delta}{\|h\|_{L^{\infty}(B_{1})}}\right)^{\frac{1}{2+p_{\rm min}-m}}\right\}.
	\end{equation*}
	Then $\tilde{u}$ satisfies \eqref{3shensuofangcheng} with $\tilde{u}(0)=0$ and $\|\tilde{u}\|_{L^{\infty}\left(B_{1}\right)}\leq 1$.
	In addition, a direct calculation yields that	
	\begin{equation*}
		\|\tilde{f}\|_{L^{\infty}\left(B_{1}\right)}\leq \frac{\tau^{2+p_{\rm min}}}{K^{1+p_{\rm min}}}\|{f}\|_{L^{\infty}\left(B_{1}\right)},
		\quad
		\|\tilde{h}\|_{L^{\infty}\left(B_{1}\right)}\leq \frac{\tau^{2-m+p_{\rm min}}}{K^{1-m+p_{\rm min}}}\|{h}\|_{L^{\infty}\left(B_{1}\right)}.
	\end{equation*}
	Then it follows from the choice of $K$ and $\tau$ above that
	\begin{equation*}
		\max\left\{\|\tilde{f}\|_{L^{\infty}\left(B_{1}\right)},\|\tilde{h}\|_{L^{\infty}\left(B_{1}\right)}\right\}
		\leq \delta.
	\end{equation*}		
	
	The strategy of the proof is to iterate Lemma \ref{lem3.2}.
	That is, we need to justify that for
	all $k\in\mathbb{N}$, it holds that	if
	\begin{equation*}
		|D u(0)| \leq \delta \rho^{k\alpha},
	\end{equation*}
	then
	\begin{equation}\label{3111aaa}
		\sup\limits_{x\in B_{\rho^{k}}}\Abs{u(x)}\leq \rho^{k(1+\alpha)}.
	\end{equation}
	We argue by finite induction. For $k=1$, \eqref{3111aaa} follows immediately from Lemma \ref{lem3.2}. Suppose that \eqref{3111aaa} holds true for $j=1,2,...,k$. Now we are going to prove \eqref{3111aaa} for $j=k+1$. To this end, we define  $u_{k}:B_{1}\rightarrow \mathbb{R}$ by
	\begin{equation*}
		u_{k}(x):=\frac{u\left(\rho^{k}x\right)}{\rho^{k(1+\alpha)}}.
	\end{equation*}
	Then $u_{k}$ is a viscosity solution of
	\begin{equation*}
		-\Phi_{k}(D{u_{k}}, x)\Delta_{p}^{N}{u_{k}}+{h_{k}}(x)\Abs{D{u_{k}}}^{m}= {f_{k}}(x) \quad \text{in} \quad  B_{1},
	\end{equation*}
	where
	\begin{align*}
		{\Phi_{k}}(\xi,x):=&\rho^{-k\alpha p_{k}(x)}\Phi\left(\rho^{k\alpha}\xi,\rho^{k}x\right)\;{\rm satisfies}\;\eqref{12}\;{\rm and} \; \eqref{13}\;{\rm with}\\
		p_{k}(x)&:=p(\rho^{k}x),\quad q_{k}(x):=q(\rho^{k}x), \quad
		a_{k}(x):=\rho^{k\alpha(q_{k}(x)-p_{k}(x))}a(\rho^{k}x),\\
		{h_{k}}(x):=&\rho^{k(1-\alpha(1+p_{k}(x)-m))}h(\rho^{k}x)
		,\quad
		{f_{k}}(x):=\rho^{k(1-\alpha(1+p_{k}(x)))}f(\rho^{k}x).
	\end{align*}
	Moreover, $0\leq p_{\rm min}\leq p_{k}(x)\leq q_{k}(x)\leq q_{\rm max}<\infty.$
	Applying \eqref{23} and $\alpha\leq  \min\left\{\frac{1+\theta_{1}}{1+q_{\rm max}},\frac{1+\theta_{2}}{1+q_{\rm max}-m}\right\}$ to obtain
	\begin{equation}\label{331}
		|{f_{k}}(x)|\leq K_{3} \rho^{k(1-\alpha(1+p_{k}(x))+\theta_{1})}|x|^{\theta_{1}}
		\leq K_{3} \rho^{k(1-\alpha(1+q_{\rm max})+\theta_{1})}|x|^{\theta_{1}}\leq K_{3}|x|^{\theta_{1}},
	\end{equation}
	\begin{equation}\label{332}
		|{h_{k}}(x)|\leq  K_{4} \rho^{k(1-\alpha(1+p_{k}(x)-m)+\theta_{2})}|x|^{\theta_{2}}
		\leq K_{4} \rho^{k(1-\alpha(1+q_{\rm max}-m)+\theta_{2})}|x|^{\theta_{2}}\leq K_{4}|x|^{\theta_{2}}.
	\end{equation}
	Via the hypotheses of induction, we have $|D u(0)| \leq \delta \rho^{k\alpha}$, and then it follows
$$|D u_{k}(0)|=\rho^{-k\alpha}|Du(0)| \leq \delta.$$ 
At this point, in view of the assumption \eqref{smalljiashe}, $u_{k}$ falls into the framework of Lemma \ref{lem3.2}, and hence it yields
	\begin{equation*}
		\sup\limits_{x\in B_{\rho}}\Abs{u_{k}(x)}\leq \rho^{1+\alpha}.
	\end{equation*}
	Scaling back, we reach that
	\begin{equation*}
		\sup\limits_{x\in B_{\rho^{k+1}}}\Abs{u(x)}\leq \rho^{(k+1)(1+\alpha)}.
	\end{equation*}
	By now, the proof of \eqref{3111aaa} is finished.
	
	Finally, for $t\in(0,\rho]$, there exists an integer $j\in \mathbb{N}$ such that $\rho^{j+1}<t\leq \rho^{j}$,
	applying  \eqref{3111aaa} to obtain
	\begin{equation*}
		\sup\limits_{x\in B_{t}}\Abs{u(x)}\leq
		\sup\limits_{x\in B_{\rho^{j}}}\Abs{u(x)}\leq \rho^{j(1+\alpha)}\leq \rho^{-(1+\alpha)}t^{1+\alpha},
	\end{equation*}
	provided $|Du(0)|\leq \delta t^{\alpha}$.
	This completes the proof of the desired result.
\end{proof}
\subsection{Proof of Theorem \ref{main}}
In the end, with the help of Lemma \ref{lem3.3}, we are in a position to complete the proof of Theorem \ref{main}.
\begin{proof}[Proof of Theorem \ref{main}] As previously explained, we assume that $u(0)=0$, and $\|{u}\|_{L^{\infty}\left(B_{1}\right)}\leq 1$
	by translation and normalization. We are going to
	prove this conclusion by implementing a dichotomy argument, which is divided into two cases. Let $\delta,\rho>0$ be two universal constants coming from Lemma \ref{lem3.3}.
	
	{\bf Case 1.} $|Du(0)|\leq \delta \rho^{\alpha}$. Set
	\begin{equation}\label{31112changshujifa}
		\kappa:=\left(\frac{|Du(0)|}{\delta}\right)^{1/\alpha}.
	\end{equation}
	Given $0<t\leq \rho$. We consider the following two subcases.\\
	{\bf Case 1.1} $\kappa\leq t\leq \rho$. It follows from \eqref{31112changshujifa} that
	\begin{equation*}
		|Du(0)|=\delta \kappa^{\alpha} \leq \delta t^{\alpha}.
	\end{equation*}
	Consequently, exploiting Lemma \ref{lem3.3} to obtain
	\begin{equation*}
		\sup\limits_{x\in B_{t}}\Abs{u(x)}\leq Ct^{1+\alpha}
	\end{equation*}
	for a universal constant $C>0$. Combining the last two displays with the triangle inequality, we derive
	\begin{equation*}
		\sup\limits_{x\in B_{t}}\Abs{u(x)-Du(0)\cdot x}\leq
		\sup\limits_{x\in B_{t}}\Abs{u(x)}+|Du(0)|t\leq (C+\delta)t^{1+\alpha},
	\end{equation*}
	which means that $u$ is of class $C^{1,\alpha}$ at 0.\\
	{\bf Case 1.2} $0<t<\kappa\leq \rho$. In this case, we consider the scaled function
	\begin{equation*}
		u_{\kappa}(x):=\frac{u(\kappa x)}{\kappa^{1+\alpha}}, \quad x\in B_{1}.
	\end{equation*}
	It follows from \eqref{31112changshujifa} that
	$|Du(0)|=\delta \kappa^{\alpha}$. Again, we employ Lemma \ref{lem3.3} to derive
	\begin{equation*}
		\sup\limits_{x\in B_{\kappa}}\Abs{u(x)}\leq C\kappa^{1+\alpha}.
	\end{equation*}
	Then it follows that
	\begin{equation*}
		\sup\limits_{x\in B_{1}}\Abs{u_{\kappa}(x)}=\sup\limits_{x\in B_{1}}\Abs{\frac{u(\kappa x)}{\kappa^{1+\alpha}}}=\frac{1}{\kappa^{1+\alpha}}\sup\limits_{x\in B_{\kappa}}\Abs{u(x)}
		\leq C.
	\end{equation*}
	Therefore, we can readily examine that $u_{\kappa}\in C(B_{1})$ is a bounded viscosity solution to
	\begin{equation} \label{switchequation}
		-{\Phi}_{\kappa}(D{u}_{\kappa}, x)	\Delta_{p}^{N}{u}_{\kappa}+	{h}_{\kappa}(x)\Abs{D{u}_{\kappa}}^{m}= 	{f}_{\kappa}(x) \quad \text{in} \quad  B_{1},
	\end{equation}
	where 
	\begin{align*}
		{{\Phi}_{\kappa}}(\xi,x):=&\kappa^{-\alpha {p}_{\kappa}(x)}\Phi\left(\kappa^{\alpha}\xi,\kappa x\right)\;{\rm satisfies}\;\eqref{12}\;{\rm and}\; \eqref{13}\;{\rm with}\\
		{p}_{\kappa}(x)&:=p(\kappa x),\quad {q}_{\kappa}(x):=q(\kappa x), \quad
		{a}_{\kappa}(x):=\kappa^{\alpha({q}_{\kappa}(x)-p_{\kappa}(x))}a(\kappa x),\\
		{h}_{\kappa}(x):=&\kappa^{1-\alpha(1+p_{\kappa}(x)-m)}h(\kappa x),\quad
		{f}_{\kappa}(x):=\kappa^{1-\alpha(1+p_{\kappa}(x))}f(\kappa x).
	\end{align*}	
	By the same arguments as in \eqref{331} and \eqref{332}, we can deduce that
	\begin{equation*}
		|{f}_{\kappa}(x)|\leq K_{3}|x|^{\theta_{1}},\quad |{h}_{\kappa}(x)|\leq K_{4}|x|^{\theta_{2}}.
	\end{equation*}
	Since the source term $f_{\kappa}$ and function $h_{\kappa}$ have a universal bound, the interior
	$C^{1,\alpha^{\prime}}$-regularity for $u_{\kappa}$ follows from Theorem \ref{thm1}. In addition, let us recall the fact that
	$|Du_{\kappa}(0)|=\Abs{\frac{Du(0)}{\kappa^{\alpha}}}=\delta>0.$ Since $u_{\kappa}\in C_{\rm loc}^{1,\alpha^{\prime}}(B_{1})$, there exists a universal constant $C> 0$ such that for all $x\in B_{1/2}$,
	$$\Abs{Du_{\kappa}(x)-Du_{\kappa}(0)}\leq C|x|^{\alpha^{\prime}},$$
	and further
	$$ \delta-C|x|^{\alpha^{\prime}}\leq \Abs{Du_{\kappa}(x)}\leq \delta+C|x|^{\alpha^{\prime}}.$$
	Then we may take a small universal radius $r>0$, independent of $\kappa$, such that
	$$\frac{\delta}{2}\leq |Du_{\kappa}(x)| \leq 2\delta \quad {\rm in}\;\, B_{r}.$$
	Hence, $u_{\kappa}$ satisfies the following equation with a universally bounded source term on the right-hand side
	\begin{equation*}
		-\Delta_{p}^{N}u_{\kappa}=\Phi_{\kappa}({Du_{\kappa}}, x)^{-1}\bigg[f_{\kappa}(x)- h_{\kappa}(x)\Abs{Du_{\kappa}}^{m} \bigg] \quad \text{in} \quad  B_{r}.
	\end{equation*}
	As a consequence, by virtue of the regularity results in Lemma \ref{111tiaohezhengzexing}, we know $u_{\kappa}\in C_{\rm loc}^{1,\alpha_{0}^{-}}(B_{r})$. Then it follows that
	\begin{equation*}
		\sup\limits_{x\in B_{\iota}}\Abs{u_{\kappa}(x)-Du_{\kappa}(0)\cdot x}\leq C\iota^{1+\alpha}
	\end{equation*}
	for every $0<\iota\leq \frac{r}{2}$ and $0<\alpha<\alpha_{0}$. Scaling back, it yields that
	\begin{equation}\label{3zuoquyujielunchengli}
		\sup\limits_{x\in B_{t}}\Abs{u(x)-Du(0)\cdot x}\leq Ct^{1+\alpha}
	\end{equation}
	for every $0<t\leq \frac{\kappa r}{2}$. It remains to show that the claim \eqref{3zuoquyujielunchengli} also holds on interval $\left(\frac{\kappa r}{2},\kappa\right)$. When $t\in\left(\frac{\kappa r}{2},\kappa\right)$, we apply Case 1.1 with $t=\kappa$ to arrive at
	\begin{equation*}
		\begin{split}
			\sup\limits_{x\in B_{t}}\Abs{u(x)-Du(0)\cdot x}&\leq \sup\limits_{x\in B_{\kappa}}\Abs{u(x)-Du(0)\cdot x}\\
			& \leq C\kappa^{1+\alpha}\\
			&=C \left(\frac{2}{r}\right)^{1+\alpha}\left(\frac{\kappa r}{2}\right)^{1+\alpha}\\
			&\leq C \left(\frac{2}{r}\right)^{1+\alpha}t^{1+\alpha}.
		\end{split}
	\end{equation*}
	At this point, we complete the proof of Case 1.
	
	{\bf Case 2.} $|Du(0)|>\delta \rho^{\alpha}$. In this scenario, we consider an auxiliary function
	$$\hat{u}(x):=\frac{\delta \rho^{\alpha}}{|Du(0)|}u(x)$$
	and then $|D\hat{u}(0)|=\delta \rho^{\alpha}$, which is back to Case 1. The proof is complete now.
\end{proof}
\section{Schauder-type regularity}\label{section666}
In order to give the proof of Theorem \ref{3aamain2}, we start off with a new flatness improvement result at the local extrema.
\begin{lemma}\label{3aalem4.1}
 Let  $u\in C(B_{1})$ be a normalized viscosity solution of \eqref{55mainmodel} under assumptions \eqref{12}-\eqref{14} and \eqref{21}-\eqref{23} with $p_{\rm min}\geq 0$. Assume that $x_{0}\in B_{1/2}$ is a local extremum, i.e., $u(x_{0})\leq u(x)$ or $u(x_{0})\geq u(x)$ in $B_{\nu}(x_{0})$ for some $\nu\in\left(0,\frac{1}{4}\right)$.  Given $\varepsilon>0$, there exists $\delta>0$ such that if
	\begin{equation*}
		\max\left\{\|f\|_{L^{\infty}(B_{1})},\|h\|_{L^{\infty}(B_{1})}\right\}\leq\delta,
	\end{equation*}
	then it holds
	\begin{equation*}
		\sup\limits_{x\in B_{\nu}(x_{0})}\Abs{u(x)-u(x_0)}\leq \varepsilon.
	\end{equation*}
\end{lemma}
\begin{proof}
	With no loss of generality, let us assume $u(x_{0})$ is a local minimum. The proof is based on a contradiction argument. If the claim fails, then there exist $\varepsilon_{0}>0$ and sequences
	of functions $\{\Phi_{j}\}_{j\in \mathbb{N}}$, $\{f_{j}\}_{j\in \mathbb{N}}$, $\{u_{j}\}_{j\in \mathbb{N}}$, $\{h_{j}\}_{j\in \mathbb{N}}$, and a sequence $\{x_{j}\}_{j\in\mathbb{N}}\subset B_{1/2}$ of local minimum point of $u_{j}$ such that
	\begin{equation*}
		-\Phi_{j}({Du_{j}}, x)\Delta_{p}^{N}u_{j}+h_{j}(x)\abs{Du_{j}}^{m} =f_{j}(x) \quad  \text{in} \quad B_{1}
	\end{equation*}	
	with $ \|u_{j}\|_{L^{\infty}(B_{1})}\leq 1$ and
	\begin{equation*}
		\max\left\{\|f_{j}\|_{L^{\infty}(B_{1})},\|h_{j}\|_{L^{\infty}(B_{1})}\right\}\leq \frac{1}{j},
	\end{equation*}
	where $\Phi_{j}$ satisfies \eqref{12} and \eqref{13} with
	$$0\leq p_{\rm min}\leq p_{j}(x)\leq q_{j}(x)\leq q_{\rm max}<\infty,$$
	$$p_{j}(\cdot),q_{j}(\cdot)\in C(B_{1}) \quad {\rm and}\quad 0\leq a_{j}(\cdot)\in C(B_{1}).$$
	However, it holds
	\begin{equation}\label{42maodun}
		\sup\limits_{x\in B_{\nu}(x_{j})}\Abs{u_{j}(x)-u_{j}(x_j)}>\varepsilon_{0}\quad {\rm for\; any \;} j\in \mathbb{N}.
	\end{equation}
	
	By the same arguments as in Lemma \ref{lem3.1}, we can conclude that
	$u_{\infty}$ is a weak solution of
	\begin{equation*}
		-\Delta_{p}u_{\infty}=0 \quad {\rm in}\quad  B_{3/4}.
	\end{equation*}
	Additionally, $x_{j}\rightarrow x_{\infty}$. Since $u_{j}$ converges locally uniformly in $B_{1}$ to $u_{\infty}$,  we know that $x_{\infty}$ is a
	local minimum of $u_{\infty}$. Thus, by applying the strong maximum principle from \cite{Vazquez1984AMO}, we deduce that $u_{\infty}\equiv const$. This clearly leads to a contradiction with \eqref{42maodun} for $j$ sufficiently large. 
\end{proof}
With the aid of Lemma \ref{3aalem4.1}, we now show Hessian continuity of solution to \eqref{55mainmodel} at local extrema points.
\begin{proof}[Proof of Theorem \ref{3aamain2}]
	As argued before, we can assume that $u(0)=0$, and $\|{u}\|_{L^{\infty}\left(B_{1}\right)}\leq 1$
	by translation and normalization. Let us also assume 0 is a local minimum of $u$. For constant $0<\varrho<1$ to be chosen later, consider the scaled function
	$$w_{1}(x):=u(\varrho x),\quad x\in B_{1}.$$
	Then we immediately verify that $w_{1}$ solves
	\begin{equation*}
		-\Phi_{\varrho}({D{w_{1}}}, x)\Delta_{p}^{N}{w_{1}}+{h_{\varrho}}(x)\Abs{D{w_{1}}}^{m}= {f_{\varrho}}(x) \quad \text{in} \quad  B_{1}
	\end{equation*}
	in the viscosity sense, where
	\begin{align*}
		{\Phi_{\varrho}}(\xi,x):=&\varrho^{ p_{\varrho}( x)}\Phi\left(\varrho^{-1}\xi,\varrho x\right)\;{\rm satisfies}\;\eqref{12}\;{\rm and}\;\eqref{13}\;{\rm with}\\
		p_{\varrho}(x):&=p(\varrho x),\quad q_{\varrho}(x):=q(\varrho x),\quad
		a_{\varrho}(x):=\varrho^{p_{\varrho}(x)-q_{\varrho}(x)}a(\varrho x),\\
		{h_{\varrho}}(x):=&\varrho^{2+p_{\varrho}( x)-m}h(\varrho x),\quad
		{f_{\varrho}}(x):=\varrho^{2+p_{\varrho}( x)}f(\varrho x).
	\end{align*}		
	Moreover, $0\leq p_{\rm min}\leq p_{\varrho}(x)\leq q_{\varrho}(x)\leq q_{\rm max}<\infty.$ It follows from \eqref{23} that
	\begin{align*}
		|{f_{\varrho}}(x)|&\leq K_{3} \varrho^{2+p_{\rm min}+\theta_{1}}|x|^{\theta_{1}}\leq K_{3} \varrho^{2+p_{\rm min}+\theta_{1}}, \\
		|{h_{\varrho}}(x)|&\leq K_{4} \varrho^{2+p_{\rm min}-m+\theta_{2}}|x|^{\theta_{2}}\leq K_{4} \varrho^{2+p_{\rm min}-m+\theta_{2}}.
	\end{align*}
	Now, we fix $\varepsilon:=\nu^{2+\mu}$ and let $\delta>0$ be the corresponding smallness condition from Lemma \ref{3aalem4.1}. We proceed by selecting $\varrho$
	\begin{equation}\label{canshuxuanqu}
		0<\varrho\leq \min\left\{\left(\frac{\delta}{K_{4}}\right)^{1/(2+p_{\rm min}-m+\theta_{2})},\left(\frac{\delta}{K_{3}}\right)^{1/(2+p_{\rm min}+\theta_{1})}\right\}.
	\end{equation}
	With such choice, $w_{1}$ is under the hypotheses of Lemma \ref{3aalem4.1}, and hence it yields
	\begin{equation}\label{41}
		\sup\limits_{x\in B_{\nu}}\Abs{w_{1}(x)}\leq \nu^{2+\mu}.
	\end{equation}
	
	As usual, we argue by induction to show that
	\begin{equation}\label{disanhanshu}
		\sup\limits_{x\in B_{\nu^{k}}}\Abs{w_{1}(x)}\leq \nu^{k(2+\mu)}
	\end{equation}
for all $k\in \mathbb{N}$. The initial induction hypothesis, $k=1$, is precisely  \eqref{41}. Assume \eqref{disanhanshu} has been verified for $j=k$. Now we prove that \eqref{disanhanshu} holds for $j=k+1$. Let us define $$w_{k}(x):=\nu^{-k(2+\mu)}w_{1}(\nu^{k}x),\quad x\in B_{1}.$$
	It is clear that 0 is still a local minimum for $w_{k}$. By induction assumption, we have $$\|w_{k}\|_{L^{\infty}(B_{1})}\leq 1.$$ Moreover, it is easy to verify that $w_{k}$ solves
	\begin{equation*}
		-\hat{\Phi}_{{\varrho}}({D w_{k}}, x)\Delta_{p}^{N}w_{k}+{\hat{h}_{\varrho}}(x)\Abs{Dw_{k}}^{m}= {\hat{f}_{\varrho}}(x) \quad \text{in} \quad  B_{1}
	\end{equation*}
	in the viscosity sense, where
	\begin{align*}
		{\hat{\Phi}}_{\varrho}(\xi,x):=&\nu^{-k(1+\mu)\hat{p}_{\varrho}(x)}\Phi_{\varrho}\left(\nu^{k(1+\mu)}\xi,\nu^{k}x\right) \;{\rm satisfies}\;\eqref{12}\;{\rm and}\;\eqref{13}\;{\rm with}\\
		\hat{p}_{\varrho}(x)&:={p}_{\varrho}(\nu^{k}x),\quad \hat{q}_{\varrho}(x):={q}_{\varrho}(\nu^{k}x), \quad
		\hat{a}_{\varrho}(x):=\nu^{k(1+\mu)(\hat{q}_{\varrho}(x)-\hat{p}_{\varrho}(x))}{a}_{\varrho}(\nu^{k}x),\\
		\hat{h}_{\varrho}(x):=&\nu^{-k\mu(1+\hat{p}_{\varrho}(x)-m)-k(\hat{p}_{\varrho}(x)-m)}h_{\varrho}(\nu^{k}x),\quad
		\hat{f}_{\varrho}(x):=\nu^{-k\mu(1+\hat{p}_{\varrho}(x))-k\hat{p}_{\varrho}(x)}f_{\varrho}(\nu^{k}x).
	\end{align*}		
	Next, we estimate the bound of $\hat{h}_{{\varrho}}$ and $\hat{f}_{{\varrho}}$. Applying \eqref{23} and $\mu=\min\left\{\frac{\theta_{1}-q_{\rm max}}{1+q_{\rm max}},\frac{m+\theta_{2}-q_{\rm max}}{1+q_{\rm max}-m}\right\}$ to deduce that
	\begin{align*}
		|\hat{f}_{\varrho}(x)&\leq K_{3}\nu^{-k\mu(1+q_{\rm max})-k(q_{\rm max}-\theta_{1})}\varrho^{2+p_{\rm min}+\theta_{1}}\leq K_{3}\varrho^{2+p_{\rm min}+\theta_{1}},\\
		|\hat{h}_{{\varrho}}(x)|&\leq 
		K_{4}\nu^{-k\mu(1+q_{\rm max}-m)-k(q_{\rm max}-m-\theta_{2})}\varrho^{2+p_{\rm min}-m+\theta_{2}}
		\leq K_{4}\varrho^{2+p_{\rm min}-m+\theta_{2}}.
	\end{align*}
	At this point, in view of \eqref{canshuxuanqu}, $w_{k}$ falls into the framework of Lemma \ref{3aalem4.1}, and hence it yields
	\begin{equation}\label{dierhanshu}
		\sup\limits_{x\in B_{\nu}}\Abs{w_{k}(x)}\leq \nu^{2+\mu}.	
	\end{equation}
	Rescaling \eqref{dierhanshu} back to $w_{1}$ yields
	\begin{equation*}
		\sup\limits_{x\in B_{\nu^{(k+1)}}}\Abs{w_{1}(x)}\leq \nu^{(k+1)(2+\mu)}.
	\end{equation*}
	The induction argument is complete.
	
	Finally, given any $0<r\leq \nu\varrho$, there exists $k\in\mathbb{N}$, such that $\nu^{(k+1)}<\frac{r}{\varrho}\leq \nu^{k}$. Applying \eqref{disanhanshu} to obtain
	\begin{align*}
		\sup\limits_{x\in B_{r}}\Abs{u(x)}&= 	\sup\limits_{x\in B_{r/\varrho}}\Abs{w_{1}(x)}\leq \sup\limits_{x\in B_{\nu^{k}}}\Abs{w_{1}(x)}\\
		&\leq \nu^{k(2+\mu)}=\nu^{(k+1)(2+\mu)-2-\mu}\\
		&\leq \left(\frac{r}{\varrho}\right)^{2+\mu}\nu^{-2-\mu}\\
		&=\left(\nu\varrho\right)^{2+\mu}r^{2+\mu}:=Cr^{2+\mu},
	\end{align*}
	where $C=C(d,p,K_{3},K_{4},\theta_{1},\theta_{2},p_{\rm min},q_{\rm max},m,\nu)$. Therefore, $u$ is $C^{2,\mu}$ differentiable at 0, with $|Du(0)|=|D^{2}u(0)|=0$. The proof is complete now.
\end{proof}
\section*{Data availability} Data sharing is not applicable to this article as obviously no datasets were generated or analyzed during the current study.
\section*{Conflict of interest} Author states no conflict of interest.

\end{document}